\documentclass{amsart}
\usepackage{pgf,tikz}
\usepackage{pgfplots}
\usetikzlibrary{spy}
\usetikzlibrary{backgrounds}
\usetikzlibrary{decorations}
\usepackage{graphicx,amssymb,amsmath,xcolor,enumerate}
\usepackage{hyperref}
\hypersetup{
bookmarks=true,         
pdffitwindow=false,     
pdfstartview={FitH},    
colorlinks=true,      
citecolor=red,
}

\usepackage{float}

\makeatletter
\def\l@subsection{\@tocline{2}{0pt}{2.5pc}{5pc}{}}
\makeatother

\usepackage{geometry}
\DeclareSymbolFont{largesymbol}{OMX}{yhex}{m}{n}
\DeclareMathAccent{\Widehat}{\mathord}{largesymbol}{"62}
\newcommand*\di{\mathop{}\!\mathrm{d}}

\numberwithin{equation}{section}              
\newtheorem{theorem}{Theorem}[section]

\newtheorem{lemma}[theorem]{Lemma}
\newtheorem{proposition}[theorem]{Proposition}
\newtheorem*{proposition*}{Proposition}

\newtheorem*{corollary*}{Corollary}
\newtheorem{claim}{Claim}[section]
\newtheorem{definition}[theorem]{Definition}
\newtheorem*{definitions*}{Definitions}

\newtheorem*{acknowledgements*}{Acknowledgements}

\newtheorem*{conjecture*}{\bf Conjecture}
\newtheorem{example}{\bf Example}[section]
\newtheorem*{example*}{\bf Example}
\theoremstyle{remark}
\newtheorem{remark}[theorem]{\bf Remark}

\begin{document}
\date{\today}                                     

\author[Yu Gao \and Hao Liu \and Tak Kwong Wong]
{Yu Gao$^{1}$\quad Hao Liu$^{2}$
	\quad Tak Kwong Wong$^{3}$}
\thanks{${}^1$ Department of Applied Mathematics, The Hong Kong Polytechnic University, Hung Hom, Kowloon, Hong Kong.
(Email: mathyu.gao@polyu.edu.hk)}

\thanks{${}^2$ School of Mathematical Sciences and Institute of Natural Sciences, Shanghai Jiao Tong University, Shanghai, China. (Email: mathhao.liu@sjtu.edu.cn)}

\thanks{${}^3$ Department of Mathematics, The University of Hong Kong, Pokfulam, Hong Kong. (Email: takkwong@maths.hku.hk)}
	
%
%

\title[Asymptotic behavior to Hunter-Saxton eqt]{Asymptotic behavior of conservative solutions to the Hunter-Saxton equation}

\maketitle

\begin{abstract}
In this paper we study the large time asymptotic behavior of (energy) conservative solutions to the Hunter-Saxton equation in a generalized framework that consists of the evolutions of solution and its energy measure. We describe the large time asymptotic expansions of the conservative solutions, and rigorously verify the validity of the leading order term in $L^{\infty}(\mathbb{R})$ and ${\dot{H}}^1(\mathbb{R})$ spaces respectively. The leading order term is given by a kink-wave that is determined by the total energy of the system only. As a corollary, we also show that the singular part of the energy measure converges to zero, as the time goes to either positive or negative infinity. Under some natural decay rate assumptions on the tails of the initial energy measure, 
we rigorously provide
the optimal error estimates in $L^{\infty}(\mathbb{R})$  and ${\dot{H}}^1(\mathbb{R})$.
As the time goes to infinity, the pointwise convergence and pointwise growth rate for the solution are also obtained under the same assumptions on the initial data.
The proofs of our results rely heavily on the elaborate analysis of the generalized characteristics designed for the measure-valued initial data, and explicit formulae for conservative solutions.

\end{abstract}
{\small
}
{
\hypersetup{linkcolor=blue}
\tableofcontents
}

\section{Introduction}\label{sec:intro}
In this paper, we study the large time asymptotic behavior of (energy) conservative solutions to the Hunter-Saxton equation on the whole real line $\mathbb{R}$: for any $x$, $t\in\mathbb{R}$, {the unknown $u$ satisfies}
\begin{align}\label{eq:HS}
u_t+uu_x=\frac{1}{2} \int_{-\infty}^x u_y^2 (y, t)\di y.
\end{align}
This integrable equation was first proposed by Hunter and Saxton \cite{hunter1991dynamics} to study a nonlinear instability in the director field of a nematic liquid crystal. Formally, one has the following energy conservation law:
\begin{align}\label{eq:conservationlaw}
(u_x^2)_t+(uu_x^2)_x=0.
\end{align}
According to Equation~\eqref{eq:conservationlaw}, it is natural to seek for the global-in-time (energy) conservative solution $u$ that satisfies $\|u_x(\cdot,t)\|_{L^2}=\|\bar{u}_x\|_{L^2}$ for all time $t\in\mathbb{R}$, provided that the initial data $u|_{t=0}=\bar{u}$ satisfies $\bar{u}_x\in L^2(\mathbb{R})$. However, some simple examples with explicit solution formulae show that $\|u_x(\cdot,t)\|_{L^2}$ is not conserved due to the potential blow up of the solutions; see \cite{bressan2007asymptotic,hunter1995nonlinear}. To describe the evolution of the total energy, an extra measure variable $\mu$ for the energy was introduced; see \cite{Bressan2010,antonio2019lipschitz} for instance. Hence, we consider the following  generalized framework\footnote{See also \cite[Eqs. (1.4)-(1.6)]{gao2021regularity} for instance.}:
\begin{align}
&u_t+uu_x=\frac{1}{2}\int_{-\infty}^x\di \mu(t),\label{eq:gHS1}\\
&\mu_t+(u\mu)_x=0,\label{eq:gHS2}\\
&\di \mu_{ac}(t)=u_x^2(x,t)\di x,\label{eq:gHS3}
\end{align} 
for any $x$, $t\in\mathbb{R}$.  Here, the energy measure $\mu(t)$ is a nonnegative Radon measure, and $\mu_{ac}(t)$ is the absolutely continuous part of $\mu(t)$ with respect to the Lebesgue measure $\di x$. Equation \eqref{eq:gHS2} is just Equation \eqref{eq:conservationlaw} with $u_x^2\di x$ replaced by $\di\mu$. Equation~\eqref{eq:gHS3} is a compatibility condition between $\mu$ and $u_x^2$, which means that $u$ only contributes to the absolutely continuous part of $\mu$. In \eqref{eq:gHS1}, the integral $\int_{-\infty}^x\di \mu(t)$ is not clearly well-defined when $\mu(t)$ contains pure point measures; however, it was shown in authors' previous work \cite{gao2021regularity} that there are  at most countably many times $t$ such that  the energy measure $\mu(t)$ can have pure point measures, and hence, the ambiguity of this integral will not affect the definition of the weak/conservative solutions; see Definition~\ref{def:weak} below for details. Therefore, we will keep this integral notation in \eqref{eq:gHS1}. Equation \eqref{eq:gHS2} is viewed in the sense of distributions. It is then natural to consider the  conservative solutions $(u,\mu)$ to the generalized framework \eqref{eq:gHS1}-\eqref{eq:gHS3} in the following space $\mathcal{D}$:
\begin{definition}\label{def:D}
Let $\mathcal{D}$ be the set of all pairs $(u,\mu)$ satisfying
\begin{enumerate}
\item[(i)] $u\in C_b(\mathbb{R})$, $u_x\in L^2(\mathbb{R})$;
\item[(ii)] $\mu\in \mathcal{M}_+(\mathbb{R})$;
\item[(iii)] $\di \mu_{ac}=u_x^2\di x$, where  $\mu_{ac}$ is the absolutely continuous part of measure $\mu$ with respect to the Lebesgue measure $\di x $.
\end{enumerate}
Here, $C_b(\mathbb{R})$ is the space of continuous and bounded functions defined on $\mathbb{R}$ equipped with the sup-norm, and $\mathcal{M}_+(\mathbb{R})$ stands for the set of all finite nonnegative Radon measures endowed with the weak topology. 
\end{definition}
The global existence and uniqueness among with many other properties of conservative solutions to the generalized framework \eqref{eq:gHS1}-\eqref{eq:gHS3} have been  established in \cite{gao2021regularity} via a generalized characteristics method; see also Theorem \ref{thm:main} below. Moreover, for any initial data in $\mathcal{D}$, the explicit formula of characteristics for conservative solutions was also obtained in \cite{gao2021regularity}; see also \eqref{eq:xbarbeta} below for instance. 
For conservative solutions, some Lipschitz metrics were introduced in \cite{Bressan2010,antonio2019lipschitz} for the stability with respect to initial data.  
In \cite{bressan2007asymptotic}, Bressan, Zhang and Zheng studied the following more general model: for any $x>0$, $t\in\mathbb{R}$, 
\[
u_t+g(u)_x=\frac{1}{2}\int_0^xg''(u)u_x^2\di x.
\]
Here, the flux $g$ is a given function with a Lipschitz continuous second-order derivative such that $g''(0)>0$. When $g(u)=u^2/2$, the above equation is just  the Hunter-Saxton equation. They obtained the global-in-time existence and uniqueness of conservative solutions on the half real line by using  the  method of characteristics for compactly supported initial energy measure $\bar{\mu}$. 
Except for conservative solutions, the Hunter-Saxton equation also has a class of weak solutions called dissipative solutions that dissipate the energy; for instance, see \cite{bressan2005global,hunter1995nonlinear1,hunter1995nonlinear,zhang1998oscillations,zhang1999existence,zhang2000existence} for the global-in-time existence, and  \cite{cieslak2016maximal,dafermos2011generalized,dafermos2012maximal} for the uniqueness of dissipative solutions via  characteristics methods. 

The prime objective of this paper is to study the large time asymptotic behavior of conservative solutions from different aspects.
The large time behavior of solutions to the Hunter-Saxton equation was first conjectured in \cite{hunter1991dynamics} by Hunter and Saxton; see Remark \ref{rmk:conjecture} below. In \cite{hunter1995nonlinear1}, Hunter and Zheng obtained the asymptotic behavior  in $\dot{H}^1(\mathbb{R})$ for dissipative solutions with compactly supported BV initial data; see \cite[Theorem 5.1]{hunter1995nonlinear1} for details. In the same paper, they also stated that the conservative solutions converge to kink-wave solutions\footnote{For an example of kink-wave solution, see Remark \ref{rmk:kinkwave} below.} asymptotically in ${\dot{H}}^1$; 
however, the relation between these kink-wave solutions and their initial data is not that clear, since they do not have the uniqueness of conservative solutions and the kink-wave solution does not have a well defined initial data in their framework\footnote{The major issue is that one requires the total energy $\bar{\mu}(\mathbb{R})$ to uniquely determine a particular kink-wave solution, but the singular part of the initial energy measure, namely $\bar{\mu}_{s}$, which is needed to define the initial value of the kink-wave is missing in the framework used in \cite{hunter1995nonlinear1}. The kink-wave solution does have a well-defined initial data in the generalized framework that we use in this paper; see Remark~\ref{rem:initial data for kinkwave} below for further discussion.}.
In this paper, we will prove this convergence for any initial data in $\mathcal{D}$ by using the explicit formulae for conservative solutions given by \cite{gao2021regularity}. 
Moreover, in \cite{hunter1995nonlinear1}, the asymptotic behavior of $u$ is not that clear; for example, what is the asymptotic expansion of $u$, say in $L^{\infty}(\mathbb{R})$? For any fixed $x\in\mathbb{R}$, does the pointwise limit  $\lim\limits_{t \to \pm\infty}u(x,t)$ exist? If the limit exists, will we be able to find the explicit form of the limit? We will address all these questions in Theorem \ref{thm:mainthm1} and Theorem \ref{thm:mainthm3} below. In addition, we will also provide the optimal estimates for the error terms in the asymptotic expansions; see Theorem \ref{thm:mainthm2} below. To the best of our knowledge, the large time pointwise behavior of conservative solutions, the asymptotic expansion in $L^\infty(\mathbb{R})$, and the optimal error estimates are completely new  and have never been mentioned in the literature. 

The asymptotic behavior of conservative solutions of the Hunter-Saxton equation is closely related to its self-similar solutions and scaling property: for any global smooth solution $u:=u(x,t)$ to the Hunter-Saxton equation \eqref{eq:gHS1} with $\di \mu(t)=u_x^2(x,t) \di x$ for all $t\in\mathbb{R}$, and any non-zero constant $\lambda$, the re-scaled function $u_{\lambda}(x,t):=\frac{1}{\lambda}u(\lambda^2 x, \lambda t)$  is also a solution to the Hunter-Saxton equation \eqref{eq:gHS1} with $\di \mu_{\lambda}(t)=((u_{\lambda})_x)^2(x,t) \di x$ ; a special class of solutions that are invariant under this scaling property\footnote{That is, $u_{\lambda}(x,t)=u(x,t)$ for any $\lambda\neq 0$. In particular, by choosing $\lambda:=2/t$, we have $u(x,t) = \frac{t}{2}u\left(\frac{4x}{t^2}, 2\right)=:\frac{t}{2}v\left(\frac{4x}{t^2}\right)$; equivalently, let  $y:=\frac{4x}{t^2}$, then $v(y):=\frac{2}{t}u\left(\frac{t^2}{4}y, t\right)$ for these self-similar solutions. In other words, the values of these self-similar solutions are completely determined by their values at a particular time, such as $t=2$.} is the self-similar solutions given by $u(x,t) := \frac{t}{2}v\left(\frac{4x}{t^2}\right)$ for some smooth function $v:\mathbb{R}\to\mathbb{R}$. A general belief is that the self-similar solutions to physical systems, such as the inviscid Burgers' equation and hyperbolic systems of conservation laws, may shed light on the large time asymptotic behavior of generic solutions, and hence, we expect that for a generic conservative solution $u$ to the Hunter-Saxton equation, we actually have $u(x,t)  =\frac{t}{2}v\left(\frac{4x}{t^2}\right) + o(t)$ in $L^{\infty}(\mathbb{R})$ as $t\to\pm\infty$, provided that the rescaled   limit $v(y) := \lim_{t\to\pm \infty}\frac{2}{t}u\left(\frac{t^2}{4}y, t\right)$ exists. In other words, if the rescaled limit function $v$ exists, the function $\frac{t}{2}v\left(\frac{4x}{t^2}\right)$ should serve as the leading order term in the large time asymptotic expansion.

As an example, let us consider the conservative solution $u$ to the Hunter-Saxton equation subject to the initial data $(\bar{u},\bar{\mu})=(k,\ell\delta_{x_0})$, where $\ell>0$,  $k$, $x_0\in\mathbb{R}$ are three constants, and $\delta_{x_0}$ is the Dirac delta mass at the spatial location $x=x_0$. It is worth noting that the limit $\lim\limits_{t\to\pm \infty}\frac{2}{t}{u\left(\frac{t^2}{4}x,t\right)}$ depends only on the total energy $\ell$, but not on $k$ and $x_0$; see Example \ref{exm:k} in Appendix \ref{app} for the verification. This also suggests that $v$ depends on the initial total energy only; in other words, the rescaled limit is more related to the total initial energy $\bar{\mu}(\mathbb{R})$ than the initial data $\bar{u}$. This is indeed true, namely we will show that the limit $v$ is completely determined by $\bar{\mu}(\mathbb{R})$ and unrelated to the specific form of $\bar{u}$, 
and the asymptotic expansions in $L^{\infty}(\mathbb{R})$ and $\dot{H}^1(\mathbb{R})$ will also be obtained by using this $v$. More precisely, we actually have:
\begin{theorem}\label{thm:mainthm1}
Let $(u(t),\mu(t))$ be a conservative solution to the generalized framework \eqref{eq:gHS1}-\eqref{eq:gHS3}  of the Hunter-Saxton equation subject to the initial data $(\bar{u},\bar{\mu})\in\mathcal{D}$, in the sense of Definition~\ref{def:weak}.
Then
\begin{equation}\label{eq:v}
v(x):=\lim\limits_{t\to\pm\infty} \frac{2}{t}{u\left(\frac{t^2}{4}x,t\right)} = 
\left\{
\begin{aligned}
&0, \quad x< 0,\\
&x, \quad 0\le x\le \bar{\mu}(\mathbb{R}),\\
&\bar{\mu}(\mathbb{R}), \quad x > \bar{\mu}(\mathbb{R}).
\end{aligned}
\right.
\end{equation}
Furthermore, we have the following asymptotic expansions: 
\begin{align}\label{eq:Linfty}
u(x,t) = \frac{t}{2} v\left(\frac{4x}{t^2}\right)+o(t) ~\textrm{in}~ {L^\infty}(\mathbb{R}),\quad\mbox{as }t\to\pm\infty,
\end{align}
and
\begin{align}\label{eq:L2}
u_x(x,t) = \partial_x\left[\frac{t}{2} v\left(\frac{4x}{t^2}\right)\right]  + o(1)~\textrm{in}~ {L^2}(\mathbb{R}),\quad\mbox{as }t\to\pm\infty.
\end{align}
As a direct consequence, we also have
\begin{align}\label{eq:singular part}
\lim_{t\to\pm\infty} {\mu}_{s}(t)(\mathbb{R})=0,
\end{align}
where ${\mu}_{s}(t)$ is the singular part of the energy measure ${\mu}(t)$ with respect to the Lebesgue measure.
\end{theorem}
It is worth noting that the $L^{\infty}(\mathbb{R})$ and $\dot{H}^1(\mathbb{R})$-norms of $\frac{t}{2} v\left(\frac{4x}{t^2}\right)$ are exactly $\frac{t}{2} \bar{\mu}(\mathbb{R})$ and $\sqrt{\bar{\mu}(\mathbb{R})}$, which are larger than the $o(t)$ and $o(1)$ terms respectively. Thus, Equations~\eqref{eq:Linfty} and \eqref{eq:L2} justify that the term $\frac{t}{2} v\left(\frac{4x}{t^2}\right)$ is the leading order term in the asymptotic expansions in $L^{\infty}(\mathbb{R})$  and $\dot{H}^1(\mathbb{R})$ respectively. 

On the other hand, for a weak/conservative solution $(u,\mu)$ to the generalized framework \eqref{eq:gHS1}-\eqref{eq:gHS3}, say in the sense of Definition~\ref{def:weak}, it is natural to ask whether the energy measure $\mu$ also has the corresponding asymptotic expansion. Indeed, the large time asymptotic behavior of $\mu$ is a direct consequence of that of $u_x$ and $\mu_s$. To illustrate the idea, let us derive the estimate in terms of the total variation as follows. Recall that the total variation distance between two Radon measures $\mu$ and $\nu$ can be defined as $\|\mu-\nu\|:= \sup_{A\in \mathcal{B}}|\mu(A) - \nu(A)|$, where $\mathcal{B}$ is the Borel $\sigma$-algebra of $\mathbb{R}$. Let $\di \nu(t) := \left( \partial_x\left[\frac{t}{2} v\left(\frac{4x}{t^2}\right)\right]\right)^2 \di x$, then it follows from \eqref{eq:L2} and \eqref{eq:singular part} that $\|\mu(t)-\nu(t)\|\le \left\|u_x(x,t)-\partial_x\left[\frac{t}{2} v\left(\frac{4x}{t^2}\right)\right]\right\|_{L^2}^2 +  {\mu}_{s}(t)(\mathbb{R})\to 0$, as $t \to \pm \infty$. To keep this work concise, we will not discuss the asymptotic expansions for $\mu$ separately in the rest of this paper, and leave this to interested readers.  
\begin{remark}[Kink-wave]\label{rmk:kinkwave}
It is worth noting that the special solution $\frac{t}{2} v\left(\frac{4x}{t^2}\right)$ is known as a kink-wave in the literature; see \cite{hunter1995nonlinear1} for instance. Hence,  the leading order terms of conservative solution $u$ in $L^{\infty}(\mathbb{R})$ and $\dot{H}^1(\mathbb{R})$ are given by the kink-wave determined by the total energy of the system. 
\end{remark}
\begin{remark}\label{rem:initial data for kinkwave}
{It follows from a direct checking that $(u,\mu):=\left(\frac{t}{2} v(\frac{4x}{t^2}), \left(\partial_x\left[\frac{t}{2} v\left(\frac{4x}{t^2}\right) \right] \right)^2\di x \right)$ is a conservative solution to the generalized framework \eqref{eq:gHS1}-\eqref{eq:gHS3} of the Hunter-Saxton equation subject to the special initial data $(0,\bar{\mu}(\mathbb{R}) \delta_0)$, in the sense of Definition~\ref{def:weak}. Hence, the kink-wave solution can be uniquely determined by its initial data in the generalized framework that we use in this work.
}
One important and interesting feature of this solution 
is that it would not blow up for all $t\neq 0$, i.e, the singular part of the energy measure is 0 for all $t\neq0$.
\end{remark}
\begin{remark}[Conjecture in \cite{hunter1991dynamics}]\label{rmk:conjecture}
In the original paper \cite{hunter1991dynamics}, Hunter and Saxton conjectured that {the large time behavior (as $t\to+\infty$) of all the strongly admissible solutions\footnote{In the sense of \cite[Definition 4.5]{hunter1991dynamics}.} can be described by
\[
U(x,t;\kappa):=
\left\{
\begin{aligned}
&-\kappa t,\quad -\infty<x\leq -\kappa t^2/2;\\
&2x/t,\quad -\kappa t^2/2<x<0;\\
&0,\quad 0\leq x<+\infty
\end{aligned}
\right.
\]
for some constant $\kappa\geq0$.} It is worth noting that $U(x,t;\kappa)$ is different from the kink-wave stated in Theorem~\ref{thm:mainthm1}, because the authors in \cite{hunter1991dynamics} actually studied a slightly different version of the Hunter-Saxton equation. More precisely, they defined the anti-derivative term on the right hand side of \eqref{eq:HS} as
\[
-\frac{1}{2}\int_x^{+\infty}u_x^2(y,t)\di y.
\]
For this definition, all the characteristics will  {propagate the energy} from right to left, and hence, the right part (for $x\ge0$) becomes zero asymptotically. However, for  \eqref{eq:HS}, all the characteristics will transfer the energy from left to right, and hence, the left part (for $x\le0$) becomes zero asymptotically. Moreover, there are some papers \cite{bressan2005global,Bressan2010,antonio2019lipschitz} using the following anti-derivative term (see \eqref{eq:HSanother} below):
\[
\frac{1}{4}\left(\int^x_{-\infty}u_x^2(y,t)\di y-\int_x^{+\infty}u_x^2(y,t)\di y\right).
\]
Then the asymptotic behavior of conservative solutions will also be different, especially the pointwise behavior; see Theorem \ref{mainthm:pointwiselimit2} for instance.
\end{remark}

To the best of authors' knowledge, the results below (i.e., Theorem~\ref{thm:mainthm2},  Theorem \ref{thm:mainthm3} and Theorem \ref{mainthm:pointwiselimit2}) have not been mentioned in the literature yet.
Equality~\eqref{eq:Linfty} tells us the error $\left\|u(\cdot,t) -\frac{t}{2}v\left(\frac{4x}{t^2}\right)\right\|_{L^\infty(\mathbb{R})}$ is $o(t)$, however,  it does not provide any precise size of it. While Equality \eqref{eq:L2} tells us the energy distribution $u_x^2(\cdot, t)$ will {asymptotically} behave like a piecewise constant function $\frac{2}{t}\chi_{\{0\leq x \leq \frac{t^2}{4} \bar{\mu}(\mathbb{R})\}}$ ($\chi$ is the characteristic function) as time goes by, but it also does not provide any information about the convergence rate. 
	For example, if we have some further information on the decay rate of the tails of the initial energy measure $\bar{\mu}$, then we will be able to rigorously show some explicit growth rate of $\left\|u(\cdot,t)-\frac{t}{2} v\left(\frac{4x}{t^2}\right)\right\|_{L^\infty(\mathbb{R})}$ and decay rate of $\left\|u_x(\cdot,t)-\partial_x\left[\frac{t}{2} v\left(\frac{4x}{t^2}\right)\right]\right\|_{L^2(\mathbb{R})}$, {which enhance the estimates in \eqref{eq:Linfty} and \eqref{eq:L2}.}
More precisely, we have {the following optimal error estimates in $\dot{H}^1(\mathbb{R})$ and $L^\infty(\mathbb{R})$ respectively}:
\begin{theorem}[Error estimates in $\dot{H}^1(\mathbb{R})$ and $L^\infty(\mathbb{R})$]\label{thm:mainthm2}
Let $(u,\mu)$ be a conservative solution to the generalized framework \eqref{eq:gHS1}-\eqref{eq:gHS3} of the Hunter-Saxton equation subject to the initial data $(\bar{u},\bar{\mu})\in\mathcal{D}$, in the sense of Definition~\ref{def:weak}. Let the function $v$ be defined by \eqref{eq:v}.  
\begin{enumerate}[(i)]
	\item If $\bar{\mu}$ satisfies the following asymptotic behavior near $x=-\infty$:
	\begin{align}\label{eq:leftcondition1}
		\limsup_{x\to-\infty}|x|^{1-\theta_1}\bar{\mu}((-\infty,-|x|^{1+\theta_1})) \le  A_1,\tag{L1}
	\end{align} 
	for some constants $0\leq 
	\theta_1 < 1$ and $A_1>0$, then we have, as $t\to\pm\infty$,
	\begin{equation}\label{eq:Linfinitycontrolleft}
		\begin{aligned}
			&\left\|u(x,t)-\frac{t}{2} v\left(\frac{4x}{t^2}\right)\right\|_{L^\infty((-\infty, 0))}= O\left({|t|}^{\theta_1}\right),\\
			&\left\|u_x(x,t)-\partial_x\left[\frac{t}{2} v\left(\frac{4x}{t^2}\right)\right]\right\|_{L^2((-\infty, 0))} = O\left({|t|}^{\frac{\theta_1-1}{2}}\right).
		\end{aligned}
	\end{equation}
	\item If $\bar{\mu}$ satisfies the following asymptotic behavior near $x=+\infty$: 
	\begin{align}\label{eq:rightcondition1}
		\limsup_{x\to+\infty}x^{1-\theta_2}\bar{\mu}((x^{1+\theta_2},+\infty)) \le  A_2,\tag{R1}
	\end{align} 
	for some constants $0\leq 
	\theta_2<1$ and $A_2$, 
	then we have, as $t\to\pm\infty$,
	\begin{equation}\label{eq:Linfinitycontrolright}
		\begin{aligned}
			&\left\|u(x,t)-\frac{t}{2} v\left(\frac{4x}{t^2}\right)\right\|_{L^\infty(( \frac{t^2}{4}\bar{\mu}(\mathbb{R}), +\infty))}= O\left({|t|}^{\theta_2}\right),\\
			&\left\|u_x(x,t)-\partial_x\left[\frac{t}{2} v\left(\frac{4x}{t^2}\right)\right]\right\|_{L^2(( \frac{t^2}{4}\bar{\mu}(\mathbb{R}), +\infty))} = O\left({|t|}^{\frac{\theta_2-1}{2}}\right).
		\end{aligned}
	\end{equation}
	\item If $\bar{\mu}$ satisfies both \eqref{eq:leftcondition1} and \eqref{eq:rightcondition1}, then we have 
	\begin{equation}\label{eq:Linfinitycontrolmiddle}
		\begin{aligned}
			&\left\|u(x,t)-\frac{t}{2} v\left(\frac{4x}{t^2}\right)\right\|_{L^\infty(\mathbb{R})}= \left\|u(x,t)-\frac{t}{2} v\left(\frac{4x}{t^2}\right)\right\|_{L^\infty((0, \frac{t^2}{4}\bar{\mu}(\mathbb{R})))} = O\left({|t|}^{\theta}\right),\\
			&\left\|u_x(x,t)-\partial_x\left[\frac{t}{2} v\left(\frac{4x}{t^2}\right)\right]\right\|_{L^2(\mathbb{R})}=\left\|u_x(x,t)-\partial_x\left[\frac{t}{2} v\left(\frac{4x}{t^2}\right)\right]\right\|_{L^2((0, \frac{t^2}{4}\bar{\mu}(\mathbb{R})))} = O\left({|t|}^{\frac{\theta-1}{2}}\right),
		\end{aligned}
	\end{equation}
	and 
	\begin{align}\label{eq:singularpartdeacyrate2}
		{\mu}_{s}(t)(\mathbb{R})=O\left({|t|}^{\frac{\theta-1}{2}}\right)
	\end{align}
	as $t\to \pm\infty$, where $\theta:=\max\{\theta_1,\theta_2\}$. 
	\item In particular, for any compactly supported $\bar{\mu}$, the above estimates hold with $\theta_1=\theta_2=0$.
\end{enumerate} 
\end{theorem}
The estimates \eqref{eq:Linfinitycontrolleft}-\eqref{eq:Linfinitycontrolmiddle} in Theorem \ref{thm:mainthm2} are indeed optimal. For example, the estimates \eqref{eq:Linfinitycontrolleft} and \eqref{eq:Linfinitycontrolright} will be attainable, if the conditions \eqref{eq:leftcondition1} and \eqref{eq:rightcondition1} are changed to \eqref{eq:leftcondition2} and \eqref{eq:rightcondition2} below; see Remark \ref{rmk:L2R2} below for further details. The estimate \eqref{eq:Linfinitycontrolmiddle} is also optimal, since it basically relies on \eqref{eq:Linfinitycontrolleft}-\eqref{eq:Linfinitycontrolright}. However, \eqref{eq:singularpartdeacyrate2} can be further improved; see Remark \ref{rem:singularpart} below for further discussions.

At last, we have the following pointwise convergence results for $u$ if we  know the  exact decay rate of the left tail of the initial energy measure $\bar{\mu}$ and the value of the limit $\bar{u}(-\infty):=\lim_{x\to-\infty}\bar{u}(x)$; here, the  {key observation} is that  we only need the  information of $\bar{u}$ and $\bar{\mu}$ in a neighbourhood of $-\infty$.

\begin{theorem}[Large time Pointwise behavior of $u$]\label{thm:mainthm3}	
Let $(u,\mu)$ be a conservative solution to the generalized framework \eqref{eq:gHS1}-\eqref{eq:gHS3} of the Hunter-Saxton equation  subject to the initial data $(\bar{u},\bar{\mu})\in\mathcal{D}$, in the sense of Definition~\ref{def:weak}.
Assume that $\bar{\mu}$ satisfies  the following asymptotic behavior near the left tail of the initial energy measure $\bar{\mu}$:
\begin{align}\label{eq:leftcondition2}
\lim_{x\to -\infty}|x|^{1-\theta_1}{\bar{\mu}((-\infty, -|x|^{1+\theta_1}))} = A_1 \tag{L2}
\end{align} 
for some constants $0\leq \theta_1<1$ and $A_1\geq0$. Then for any $x\in\mathbb{R}$,
\begin{enumerate}
\item[(i)] if $A_1>0$ and $0<\theta_1<1$, we have
\begin{equation}\label{eq:uxtlimit1}
\lim_{t\to+\infty}\frac{u(x,t)}{t^{\theta_1}}=2\left(\frac{A_1}{4}\right)^{\frac{1+\theta_1}{2}};
\end{equation}
\item[(ii)] if $A_1>0$ and $\theta_1=0$ {and $\bar{u}(-\infty)$ exists}, we have
\begin{equation}\label{eq:uxtlimit2}
\lim_{t\to+\infty}u(x,t)=\sqrt{\bar{u}^2(-\infty)+A_1};
\end{equation}
\item[(iii)] if $A_1=0$, $0<\theta_1<1$, and $\bar{\mu}((-\infty,z))>0$ for any $z\in\mathbb{R}$, we have
\begin{equation}\label{eq:uxtlimit3}
\lim_{t\to+\infty}\frac{u(x,t)}{t^{\theta_1}}=0;
\end{equation}
\item[(iv)] if $A_1=\theta_1=0$, and $\bar{\mu}((-\infty,z))>0$ for any $z\in\mathbb{R}$,  and $\bar{u}(-\infty)$ exists, then we have
\begin{equation}\label{eq:limtingforfixedx}
\lim_{t\to+\infty}u(x,t)=|\bar{u}(-\infty)|.
\end{equation}
\end{enumerate}
If $\bar{\mu}$ satisfies
\begin{align}\label{eq:leftcondition3}
\ell:=\inf\mathrm{supp}\{\bar{\mu}\}>-\infty,\tag{L3}
\end{align}
then we also have \eqref{eq:limtingforfixedx} provided that   $\bar{u}(-\infty)$ exists.
\end{theorem}
\begin{remark}
	It is worth noting that the conditions \eqref{eq:leftcondition2} above and \eqref{eq:rightcondition2} below are attainable for any $0\leq \theta_1,\theta_2<1$. For a precise construction of such initial data, see Example \ref{ex:nosingular} in Appendix \ref{app} for instance.
\end{remark}

We can see from the above theorem that for $\theta_1>0$, the effect of   the tail of the initial energy measure is dominant and it accounts for the pointwise sub-linear growth $t^{\theta_1}$. The more interesting case is $\theta_1=0$, the effects of the tails of the initial energy measure $\bar{\mu}$ and $\bar{u}(-\infty)$ are at the same level, and hence, they affect each other in a non-trivial way; see \eqref{eq:uxtlimit2} for instance. Moreover, \eqref{eq:limtingforfixedx} can also be seen as the limiting case of \eqref{eq:uxtlimit2} as $A_1\to 0^+$. We refer readers to the beginning of Section~\ref{sec:decayrate} for further discussion, as well as Remark~\ref{rem:Discussion_of_Figure1} and Figure~\ref{fig:chara} for some more explanations of \eqref{eq:limtingforfixedx}.
It is worth noting that if $\left(u(\cdot,t), \mu(t)\right)$ is a conservative solution to the Hunter-Saxton equation, then so is $\left(\tilde{u}(\cdot,t), \tilde{\mu}(t) \right)$ $=\left(-u(\cdot,-t), \mu(-t)\right)$.
Using this fact, one can immediately obtain similar results to Theorem \ref{thm:mainthm3} and Theorem \ref{mainthm:pointwiselimit2} for the case  $t\to-\infty$. 

Finally, we will also study the asymptotic behavior of conservative solutions to the following form of Hunter-Saxton equation, which also appears in the literature (see \cite{bressan2005global,Bressan2010,antonio2019lipschitz} for instance):
\begin{align}\label{eq:HSanother} 
u_t+uu_x=\frac{1}{4}\left(\int_{-\infty}^x u_y^2 (y, t)\di y- \int_x^{+\infty} u_y^2 (y, t)\di y\right).
\end{align}
The generalized framework of this form can also be considered in the same manner as that for \eqref{eq:HS}; for instance, see \eqref{eq:gHS21}-\eqref{eq:gHS23} below. In Section~\ref{sec: another form} we will only state (without proof) the global characteristics and the rescaled limit $v_1(x):=\lim_{t \to \pm\infty}\frac{2}{t}{u\left(\frac{t^2}{4}x,t\right)}$ for conservative solutions to the generalized framework \eqref{eq:gHS21}-\eqref{eq:gHS23} corresponding to \eqref{eq:HSanother}; see Theorem \ref{thm:limitfunction2} for the precise form of $v_1(x)$. Similar asymptotic expansions and error estimates stated in Theorem \ref{thm:mainthm1} and Theorem \ref{thm:mainthm2} respectively for the generalized framework \eqref{eq:gHS21}-\eqref{eq:gHS23} to the Hunter-Saxton Equation \eqref{eq:HSanother} can be also obtained exactly in the same way based on this new limiting function $v_1(x)$, and we leave them to interested readers.  
However, we will provide the pointwise limit $\lim_{t\to\pm\infty} u(x,t)$ for conservative solutions to the generalized framework \eqref{eq:gHS21}-\eqref{eq:gHS23} of \eqref{eq:HSanother} below, 
as it is shown that the pointwise limit can be very different. Compared with Theorem \ref{thm:mainthm3}, the situation now is much  simpler; more precisely, the limit $\lim_{t\to+\infty} u(x,t)$ not only always exists for all $x$, but is just a constant independent of $x$.
\begin{theorem}[Pointwise convergence of $u$ to the Hunter-Saxton equation \eqref{eq:HSanother}]\label{mainthm:pointwiselimit2}
Let $(u,\mu)$ be the conservative solution to the generalized framework \eqref{eq:gHS21}-\eqref{eq:gHS23} of the Hunter-Saxton equation \eqref{eq:HSanother} subject to the initial data $(\bar{u},\bar{\mu})\in\mathcal{D}$.
Then for any fixed $x\in\mathbb{R}$,
\begin{equation}\label{eq:pointwizsev2}
\lim_{t\to +\infty}u(x,t) = -\bar{u}(\bar{x}(\alpha^*)),
\end{equation}
where $\alpha^*$ is any real number such that $\alpha^*-\bar{x}(\alpha^*)= \frac{1}{2}\bar{\mu}(\mathbb{R})$, and the function $\bar{x}(\alpha)$ will be defined by \eqref{eq:barx1}.
\end{theorem}
It is worth noting that the definition of $\alpha^*$ above is not unique in general, however, one can easily verify that this will not affect the value of $\bar{u}(\bar{x}(\alpha^*))$, namely the numerical value of $\bar{u}(\bar{x}(\alpha^*))$ in \eqref{eq:pointwizsev2} is uniquely determined; see the proof of Theorem \ref{mainthm:pointwiselimit2} in Section \ref{sec: another form}. 

The rest of this paper is organized as follows. Section \ref{sec:characteristicsmain1} and Section \ref{sec:decayrate} will be devoted to the study of large time behaviors of the generalized framework \eqref{eq:gHS1}-\eqref{eq:gHS3}.
In Section \ref{sec:characteristicsmain1}, we will first recall the explicit formulae for the generalized characteristics and the global-in-time conservative solutions $(u,\mu)$. Then we will prove Theorem \ref{thm:mainthm1} based on them. In Section~\ref{sec:decayrate}, we will first derive the optimal error estimates in $L^{\infty}(\mathbb{R})$ and $\dot{H}^1(\mathbb{R})$ respectively, under the assumptions on the tails of initial energy measure $\bar{\mu}$; then we will also study the pointwise behavior of conservative solutions. In particular, Theorem \ref{thm:mainthm2} and Theorem \ref{thm:mainthm3} will be shown in this section. 
In Section~\ref{sec: another form}, we will study the asymptotic behavior of conservative solutions to the generalized framework~\eqref{eq:gHS21}-\eqref{eq:gHS23} of \eqref{eq:HSanother}, especially we will prove Theorem \ref{mainthm:pointwiselimit2}.

\section{Asymptotic behavior and expansions in $L^{\infty}(\mathbb{R})$ and $\dot{H}^1(\mathbb{R})$}\label{sec:characteristicsmain1}

In this section we will first recall the unique existence
of conservative solutions and their global characteristics for the generalized framework \eqref{eq:gHS1}-\eqref{eq:gHS3} of the Hunter-Saxton equation~\eqref{eq:HS}, and then prove Theorem \ref{thm:mainthm1}.

\subsection{Preliminaries} 
Recall the following definition of conservative  solutions to the generalized framework \eqref{eq:gHS1}-\eqref{eq:gHS3} of the Hunter-Saxton equation from \cite[Definition 1.2]{gao2021regularity}:
\begin{definition}[Conservative solutions]\label{def:weak}
Let $(\bar{u},\bar{\mu})\in \mathcal{D}$ be a given initial data, where the space $\mathcal{D}$ was defined in Definition~\ref{def:D}. The pair $(u(t), \mu(t))$ is said to be a global-in-time conservative solution to the generalized framework \eqref{eq:gHS1}-\eqref{eq:gHS3} subject to the initial data $(\bar{u},\bar{\mu})$, if the pair $(u(t),\mu(t))$ satisfies all of the following:
\begin{enumerate}
\item[(i)] $u\in C(\mathbb{R};C_b(\mathbb{R}))\cap C^{1/2}_{loc}(\mathbb{R}\times\mathbb{R})$, $u_t\in L_{loc}^2(\mathbb{R}\times\mathbb{R})$, $u_x(\cdot, t) \in L^2(\mathbb{R})$ for all $t \in \mathbb{R}$,  and $\mu\in C(\mathbb{R};\mathcal{M}_+(\mathbb{R}))$;
\item[(ii)] $(u(\cdot,0),\mu(0))=(\bar{u},\bar{\mu})$, and $\di\mu(t)={u}_x^2(x,t)\di x$ for a.e. $t\in\mathbb{R}$;
\item[(iii)] the equation
\begin{align}\label{eq:weakformula}
\int_{\mathbb{R}}\int_{\mathbb{R}}u\phi_t-\phi\left(uu_x-\frac{1}{2}F\right)\di x\di t=0
\end{align}
holds for all $\phi\in C_c^\infty(\mathbb{R}\times \mathbb{R})$; the function $F(x,t)$ is defined by  $F(x,t) :=\int_{-\infty}^x\di \mu(t)$; 
\item[(iv)] the conservation of energy
\begin{equation}\label{eq:fourth}
\int_{\mathbb{R}}\int_{\mathbb{R}} \left( \phi_{t} + u \phi_{x}\right) \di \mu(t) \di t =0 
\end{equation}
holds for all $\phi\in C^{\infty}_c(\mathbb{R}\times \mathbb{R})$; and
\item[(v)] Equation~\eqref{eq:gHS3} holds for all $t \in \mathbb{R}$. 
\end{enumerate}
The condition (v) above is equivalent to $(u(t), \mu(t)) \in \mathcal{D}$ for all $t \in \mathbb{R}$.   
\end{definition}
In the following, the main tools for studying the large time asymptotic behaviors of conservative solutions will be the explicit formulae for the  globally-in-time well-defined (generalized) characteristics. We will briefly recall the generalized characteristics below, and state the existence and uniqueness results in \cite{gao2021regularity}.  
Consider an initial data $(\bar{u},\bar{\mu})\in\mathcal{D}$. For any $\alpha\in\mathbb{R}$, we can  define the function $\bar{x}(\alpha)$ 
via
\begin{align}\label{eq:barx1}
\bar{x}(\alpha)+\bar{\mu}((-\infty, \bar{x}(\alpha)))\leq \alpha \leq \bar{x}(\alpha)+\bar{\mu}((-\infty, \bar{x}(\alpha)]). 
\end{align}
It follows directly from the above definition of $\bar{x}$ that $\bar{x}(\alpha)\le \alpha$ for all $\alpha\in\mathbb{R}$, and $\bar{x}(\alpha)$ is a nondecreasing and Lipschitz continuous function with Lipschitz constant bounded by 1; see \cite[Proposition 2.1]{gao2021regularity} for the verification of these facts.
Then we further define $y(\alpha, t)$ as follows: for any $\alpha$, $t\in\mathbb{R}$,
\begin{align}\label{eq:xbarbeta}
y(\alpha,t):=\bar{x}(\alpha)+\bar{u}(\bar{x}(\alpha))t+\frac{t^2}{4}(\alpha-\bar{x}(\alpha)).
\end{align}
For any $t\in\mathbb{R}$, the conservative solution $(u(x,t),\mu(t))$ is then given by 
\begin{align}\label{eq:measuresolutionu}
u(x,t):=\frac{\partial}{\partial t}y(\alpha,t)=\bar{u}(\bar{x}(\alpha))+\frac{t}{2}(\alpha-\bar{x}(\alpha))~\textrm{ for }~ x=y(\alpha,t),
\end{align}
and
\begin{equation}\label{eq:measuresolutionmu}
\mu(t):=y(\cdot,t)\# (f\di \alpha),
\end{equation}
where $f(\alpha) := 1-\bar{x}'(\alpha)$. The following theorem was proved in \cite[Section 3]{gao2021regularity}:
\begin{theorem}[Existence and uniqueness]\label{thm:main}
Let $(\bar{u},\bar{\mu})\in \mathcal{D}$ be a given initial data. Let $u$ and $\mu$ be defined by \eqref{eq:measuresolutionu} and \eqref{eq:measuresolutionmu} respectively. Then $(u(t), \mu(t))$ is the unique global-in-time conservative solution to the generalized framework \eqref{eq:gHS1}-\eqref{eq:gHS3} subject to the initial data $(\bar{u},\bar{\mu})$, in the sense of Definition~\ref{def:weak}. 
\end{theorem}
Using these explicit formulae, one can also effectively study the regularity structure of conservative solutions to the generalized framework \eqref{eq:gHS1}-\eqref{eq:gHS3}; see \cite[Theorem 2.1]{gao2021regularity} for more details.

\subsection{Asymptotic behavior and expansions of conservative solutions}\label{sec:Asymptotic behavior}
In this subsection, we are going to prove Theorem~\ref{thm:mainthm1}. Consider a solution $(u(\cdot,t),\mu(t)) \in \mathcal{D}$ that solves  \eqref{eq:gHS1}-\eqref{eq:gHS3}, in the sense of Definition~\ref{def:weak}. As we briefly analyzed in the introduction (i.e., Section~\ref{sec:intro}), the first step towards the large time asymptotic behavior is to find $v(x):=\lim_{t \to \pm\infty}\frac{2}{t}{u\left(\frac{t^2}{4}x,t\right)}$. Let us begin with the expression for $\frac{2}{t}{u\left(\frac{t^2}{4}x,t\right)}$.
According to \eqref{eq:measuresolutionu}, we have
\begin{equation*}
\frac{2}{t}{u\left(\frac{t^2}{4}x,t\right)} = \frac{2}{t} \bar{u}(\bar{x}(\alpha)) + 
\alpha-\bar{x}(\alpha),
\end{equation*}
where $\alpha$ is chosen so that $\frac{t^2}{4}x=y(\alpha,t)$.
From \eqref{eq:xbarbeta}, we have the following relation between $x$ and $\alpha$:
\begin{equation*}
\frac{t^2}{4}x= y(\alpha,t)=\bar{x}(\alpha)+\bar{u}(\bar{x}(\alpha))t+\frac{t^2}{4}(\alpha-\bar{x}(\alpha)),
\end{equation*}
or equivalently,
\begin{equation}\label{eq:relation between x and alpha}
x= \frac{4}{t^2}y(\alpha,t)=\frac{4}{t^2}\bar{x}(\alpha)+ \frac{4}{t}\bar{u}(\bar{x}(\alpha))+ (\alpha-\bar{x}(\alpha)).
\end{equation}
Now, let us introduce the concept of pseudo inverse function for \eqref{eq:relation between x and alpha} as follows, which will be useful in the mathematical analysis below.
\begin{definition}\label{def:pseudo}
For any $t\neq 0$, let $\alpha(\cdot,t):\mathbb{R}\to\mathbb{R}$ be the pseudo inverse function defined by
\begin{align}\label{eq:defalpha}
\alpha(x,t): =\inf\left\{\alpha\in\mathbb{R}:\;x=\frac{4}{t^2}y(\alpha,t)\right\},
\end{align}
where the function $y(\alpha,t)$ is defined by \eqref{eq:xbarbeta}.
\end{definition}
\noindent
Since $y$ is continuous, we actually have $x=\frac{4}{t^2}y(\alpha(x,t),t)$.
Then the first aim becomes to  find the limit of 
\begin{equation*}
\frac{2}{t}{u\left(\frac{t^2}{4}x,t\right)}= \frac{2}{t} \bar{u}(\bar{x}(\alpha(x,t))) + 
\alpha(x,t)-\bar{x}(\alpha(x,t)),
\end{equation*}
as $t$ tends to positive or negative infinity. Let us recall that the initial data $(\bar{u},\bar{\mu})$ is assumed to be in $\mathcal{D}$, so it follows from part (i) of Definition~\ref{def:D} that $\bar{u}$ is bounded, and hence, we have
\begin{equation}\label{eq:limit}
v(x):=\lim\limits_{t\to\pm \infty}\frac{2}{t}{u\left(\frac{t^2}{4}x,t\right)}= \lim\limits_{t\to \pm\infty}
[\alpha(x,t)-\bar{x}(\alpha(x,t))].
\end{equation}
On the other hand, it follows from Equation~\eqref{eq:relation between x and alpha} that 
\begin{equation}\label{eq:charater relation}
\alpha(x,t)-\bar{x}(\alpha(x,t)) = x - \frac{4}{t^2}\bar{x}(\alpha(x,t))- \frac{4}{t}\bar{u}(\bar{x}(\alpha(x,t))).
\end{equation}
The relation \eqref{eq:charater relation} is of great importance, this relation and its variants will be  used very often to derive many useful consequences, including the asymptotic behaviors of  $\alpha(x,t)$ and $\bar{x}(\alpha(x,t))$; see the proofs of Lemma~\ref{lem:important} and Lemma~\ref{lemma:properties} below 
for instance.
Now, substituting \eqref{eq:charater relation} into \eqref{eq:limit}, and using the boundedness of $\bar{u}$ again, we also have
\begin{equation}\label{eq:asy}
v(x)=\lim\limits_{t\to\pm \infty}\frac{2}{t}{u\left(\frac{t^2}{4}x,t\right)} = x - \lim\limits_{t\to\pm \infty} \frac{4}{t^2}\bar{x}(\alpha(x,t)) = x - \lim\limits_{t\to \pm\infty} \frac{4}{t^2}\alpha(x,t),
\end{equation}
where the last identity follows from the fact that 
\begin{align}\label{eq:diffbetalphaandbarx}
\bar{x}(\alpha(x,t))+\bar{\mu}((-\infty, \bar{x}(\alpha(x,t))))\leq \alpha(x,t) \leq \bar{x}(\alpha(x,t))+\bar{\mu}((-\infty, \bar{x}(\alpha(x,t))])
\end{align}
and $\bar{\mu}(\mathbb{R})<+\infty$. 
One then sees that if we can find 
$\lim\limits_{t\to \pm\infty}[\alpha(x,t)-\bar{x}(\alpha(x,t))]$ or 
$ \lim\limits_{t\to \pm\infty} \frac{4}{t^2}\alpha(x,t)$, 
we can use  \eqref{eq:limit} or \eqref{eq:asy} to find the value of $v$. This indeed can be done and we have the following important Lemma:
\begin{lemma}\label{lem:important}
For the pseudo inverse function $\alpha(x,t)$ defined in Definition \ref{def:pseudo}, we have
\begin{equation}\label{eq:differencealphaxt}
\lim\limits_{t\to \pm\infty}
[\alpha(x,t)-\bar{x}(\alpha(x,t))]=\left\{
\begin{aligned}
&0, \quad x< 0,\\
&x, \quad 0\le x\le \bar{\mu}(\mathbb{R}),\\
&\bar{\mu}(\mathbb{R}), \quad x > \bar{\mu}(\mathbb{R}),
\end{aligned} \right.
\end{equation}
and
\begin{equation}\label{eq:grwothalphaxt}
\lim\limits_{t\to \pm\infty} \frac{4}{t^2}{\bar{x}(\alpha(x,t))}=\lim\limits_{t\to \pm\infty} \frac{4}{t^2}\alpha(x,t)=\left\{
\begin{aligned}
&x, \quad x< 0,\\
&0, \quad 0\le x\le \bar{\mu}(\mathbb{R}),\\
&x-\bar{\mu}(\mathbb{R}), \quad x > \bar{\mu}(\mathbb{R}).
\end{aligned} \right.
\end{equation}
\end{lemma}

\begin{proof}
In the trivial case $\bar{\mu}(\mathbb{R})=0$, we actually have $\bar{u}\equiv C$ for some constant $C\in\mathbb{R}$. Furthermore, it follows from \eqref{eq:barx1} that
$\bar{x}(\alpha)\equiv \alpha$, so $\alpha(x,t)-\bar{x}(\alpha(x,t))=\alpha(x,t)-\alpha(x,t)\equiv0$, and hence, \eqref{eq:differencealphaxt} holds. In addition, using \eqref{eq:xbarbeta}, we also obtain $y(\alpha,t)=\alpha+Ct$, which and \eqref{eq:defalpha} imply that $\alpha(x,t)=\frac{t^2x}{4}-Ct$, and hence, \eqref{eq:grwothalphaxt} follows with $\bar{\mu}(\mathbb{R})=0$.

From now on, we will, without loss of generality, assume that $\bar{\mu}(\mathbb{R}) > 0$.
The proof will be separated into the following four cases.

\noindent\textbf{Case 1 ($x<0$).}
For any given $x<0$, we claim that $\lim\limits_{t\to \pm\infty} \alpha(x,t)= -\infty$. Seeking for a contradiction, we assume that there exists a constant $M>0$ and a sequence $t_n \to \pm\infty$ such that $\alpha(x,t_n)\ge -M$ for all $n\in\mathbb{N}$. It follows from \eqref{eq:charater relation} that for each  $t_n$,
\begin{equation}\label{eq: character relation for tn}
0\leq\alpha(x,t_n)-\bar{x}(\alpha(x,t_n)) = x - \frac{4}{t_n^2}\bar{x}(\alpha(x,t_n))- \frac{4}{t_n}\bar{u}(\bar{x}(\alpha(x,t_n))),
\end{equation}
where the first inequality follows from the nonnegativity of $\bar{\mu}$ and the first inequality in \eqref{eq:diffbetalphaandbarx}.  Furthermore, using the second inequality in \eqref{eq:diffbetalphaandbarx}, we also have
\begin{align*}
\bar{x}(\alpha(x,t_n)) \ge\alpha(x,t_n) - \bar{\mu}((-\infty, \bar{x}(\alpha(x,t_n))]) \ge -M - \bar{\mu}(\mathbb{R}).
\end{align*}
Hence, taking the limit superior in \eqref{eq: character relation for tn} and using the boundedness of $\bar{u}$, we finally obtain
\begin{equation*}
\begin{aligned}
0\le \limsup_{n\to +\infty} \left[\alpha(x,t_n)-\bar{x}(\alpha(x,t_n)) \right] &=  \limsup_{n\to +\infty} \left( x - \frac{4}{t_n^2}\bar{x}(\alpha(x,t_n))- \frac{4}{t_n}\bar{u}(\bar{x}(\alpha(x,t_n))) \right)\\
& \le x + \limsup_{n\to +\infty} \frac{4}{t_n^2}( M + \bar{\mu}(\mathbb{R})) = x <0,
\end{aligned}
\end{equation*}
which is a contradiction. Therefore, for any given $x<0$, we have just shown that $\lim\limits_{t\to \pm\infty} \alpha(x,t)= -\infty$. 
A direct consequence of \eqref{eq:diffbetalphaandbarx}  is that  $\left|\alpha(x,t)-\bar{x}(\alpha(x,t))\right|\leq \bar{\mu}(\mathbb{R})<+\infty$, so the limit $\lim\limits_{t\to \pm\infty} \alpha(x,t)= -\infty$ immediately implies $\lim\limits_{t\to \pm \infty} \bar{x}(\alpha(x,t))= -\infty$, and hence, $\lim\limits_{t\to \pm\infty} \bar{\mu}((-\infty, \bar{x}(\alpha(x,t))])= 0$. Another consequence of \eqref{eq:diffbetalphaandbarx} is that $0\leq\alpha(x,t)-\bar{x}(\alpha(x,t))\leq\bar{\mu}((-\infty, \bar{x}(\alpha(x,t))])$, so $\lim\limits_{t\to \pm\infty} \bar{\mu}((-\infty, \bar{x}(\alpha(x,t))])= 0$ and the squeeze theorem imply
$\lim\limits_{t\to \pm\infty} \left[\alpha(x,t)-\bar{x}(\alpha(x,t))  \right]=0$, which is \eqref{eq:differencealphaxt} indeed.
Then passing to the limit in \eqref{eq:charater relation} as $t\to\pm\infty$, we obtain
	  $\lim\limits_{t\to \pm\infty} \frac{4 \bar{x}(\alpha(x,t))}{t^2}=x$. Since $\left|\alpha(x,t)-\bar{x}(\alpha(x,t))\right|\leq \bar{\mu}(\mathbb{R})<+\infty$, we also have $ \lim\limits_{t\to \pm\infty} \frac{4\alpha(x,t)}{t^2}=x$. This shows \eqref{eq:grwothalphaxt}.

\noindent\textbf{Case 2 ($x>\bar{\mu}(\mathbb{R})$).}
For any given $x> \bar{\mu}(\mathbb{R})$, we claim that $\lim\limits_{t\to \pm\infty} \alpha(x,t) = +\infty$. Seeking for a contradiction, we assume that there exists a constant $M>0$ and  a sequence $t_n \to \pm\infty$ such that  $\alpha(x,t_n)\le M$. Then it follows from the first inequality in \eqref{eq:diffbetalphaandbarx}  and the nonnegativity of $\bar{\mu}$ that for any $t_n$, 
\begin{align*}
\bar{x}(\alpha(x,t_n)) \le\alpha(x,t_n) - \bar{\mu}((-\infty, \bar{x}(\alpha(x,t_n)))) \le M.
\end{align*}
Hence, using the second inequality of \eqref{eq:diffbetalphaandbarx}, Equation~\eqref{eq:charater relation} and the boundedness of $\bar{u}$, we have
\begin{equation*}
\begin{aligned}
\bar{\mu}(\mathbb{R}) &\ge \limsup_{n\to +\infty}[\alpha(x,t_n)-\bar{x}(\alpha(x,t_n))] = \limsup_{n\to +\infty} \left(x - \frac{4}{t_n^2}\bar{x}(\alpha(x,t_n))- \frac{4}{t_n}\bar{u}(\bar{x}(\alpha(x,t_n))) \right)\\
&\ge x + \limsup_{n\to +\infty} \frac{4}{t_n^2}( -M )  =x > \bar{\mu}(\mathbb{R}),
&
\end{aligned}
\end{equation*}
which is a contradiction. Thus, for any given $x > \bar{\mu}(\mathbb{R})$, we  have proved that $\lim\limits_{t\to \pm\infty} \alpha(x,t)=+\infty$. 
Similar to Case 1, we can use $	\lim\limits_{t\to \pm\infty} \alpha(x,t)=+\infty$ and \eqref{eq:diffbetalphaandbarx} to show that $\lim\limits_{t\to \pm\infty} \bar{x}(\alpha(x,t))= +\infty$, and hence, $\lim\limits_{t\to \pm\infty} \bar{\mu}((-\infty, \bar{x}(\alpha(x,t))))=\bar{\mu}(\mathbb{R})$. Using \eqref{eq:diffbetalphaandbarx} again and the squeeze theorem, 
we  have
$\lim\limits_{t\to \pm\infty} \left[\alpha(x,t)-\bar{x}(\alpha(x,t))  \right]=\bar{\mu}(\mathbb{R})$, which is \eqref{eq:differencealphaxt}. Then passing to the limit in \eqref{eq:charater relation} as $t\to\pm\infty$, we finally have 
 $\lim\limits_{t\to \pm\infty} \frac{4 \bar{x}(\alpha(x,t))}{t^2}$ = $ \lim\limits_{t\to \pm\infty} \frac{4\alpha(x,t)}{t^2}=x-\bar{\mu}(\mathbb{R})$, which is \eqref{eq:grwothalphaxt}.


\noindent\textbf{Case 3 ($0<x<\mu(\mathbb{R})$).} For any given $x\in (0,\mu(\mathbb{R}))$, we claim that  $\alpha(x,t)$ is  bounded in $t$, namely there exists a positive number $M(x)$ depending on $x$ such that $|\alpha(x,t)| \le M(x)$ for all $t\in\mathbb{R}$. Seeking for a contradiction, we assume that there exists a sequence $t_n \to \pm \infty$ such that  $\alpha(x,t_n)\to -\infty$ or $\alpha(x,t_n)\to +\infty$. We consider the case that $\alpha(x,t_n)\to -\infty$  first, by using an argument similar to the argument used in Case 1. Since $\alpha(x,t_n)\to -\infty$, it follows from \eqref{eq:diffbetalphaandbarx} that $\bar{x}(\alpha(x,t_n)) \to -\infty$   and $\alpha(x,t_n)-\bar{x}(\alpha(x,t_n)) \to 0^+$. Now, for this particular sequence $\{t_n\}$, evaluating \eqref{eq:charater relation} at $t=t_n$, and then passing to the limit as $n\to +\infty$ and using the boundedness of $\bar{u}$, we have 
\begin{equation*}
\begin{aligned}
0 = \lim_{n\to +\infty} \left[\alpha(x,t_n)-\bar{x}(\alpha(x,t_n)) \right] &=  \lim_{n\to +\infty} \left( x - \frac{4}{t_n^2}\bar{x}(\alpha(x,t_n))- \frac{4}{t_n}\bar{u}(\bar{x}(\alpha(x,t_n))) \right)\ge x,
\end{aligned}
\end{equation*}
which contradicts the hypothesis $x>0$. Similarly, one can also use the same argument to show that the case $\alpha(x,t_n)\to +\infty$ will contradict with $x<\bar{\mu}(\mathbb{R})$. In conclusion, we have just verified that $|\alpha(x,t)| \le M(x)$ for all $0<x<\bar{\mu}(\mathbb{R})$, this also shows that $\bar{x}(\alpha(x,t)) $ is uniformly bounded in $t$, again by \eqref{eq:diffbetalphaandbarx}. 
Hence, we have $\lim\limits_{t\to \pm\infty} \frac{4 \bar{x}(\alpha(x,t))}{t^2}$ = $ \lim\limits_{t\to \pm\infty} \frac{4\alpha(x,t)}{t^2}=0$.
Passing to the limit in \eqref{eq:charater relation} as $t\to\pm\infty$, we obtain $\lim\limits_{t\to \pm\infty}
[\alpha(x,t)-\bar{x}(\alpha(x,t))]=x$.


\noindent\textbf{Case 4 ($x=0$ and $x=\mu(\mathbb{R})$).}
For $x=0$, we first show that 
the case $\alpha(0,t_n) \to +\infty$ for some sequence $t_n\to \pm\infty$ will not happen. Otherwise, similar to the arguments used in previous cases, it follows from \eqref{eq:diffbetalphaandbarx}, \eqref{eq:charater relation} and boundedness of $\bar{u}$ that 
\begin{equation*}
\begin{aligned}
\bar{\mu}(\mathbb{R}) = \lim_{n\to +\infty} \left[\alpha(0,t_n)-\bar{x}(\alpha(0,t_n)) \right]&=  \lim_{n\to +\infty} \left( 0 - \frac{4}{t_n^2}\bar{x}(\alpha(0, t_n))- \frac{4}{t_n}\bar{u}(\bar{x}(\alpha(0, t_n))) \right) \le 0,
\end{aligned}
\end{equation*}
which contradicts the assumption that $\bar{\mu}(\mathbb{R}) >0$. 
We have just proved that $\limsup_{t\to\pm\infty}\alpha(0,t)<+\infty$, and using \eqref{eq:diffbetalphaandbarx} again, we also have $\limsup_{t\to\pm\infty}\bar{x}(\alpha(0,t))<+\infty$. 

Next, we will prove \eqref{eq:differencealphaxt} at $x=0$ by contradiction as follows. Suppose that  \eqref{eq:differencealphaxt} does not hold at $x=0$. Since $\alpha-\bar{x}(\alpha)\geq0$ for any $\alpha\in\mathbb{R}$, there exists a positive constant $k$ and a sequence $t_n\to\pm\infty$ such that 
\begin{equation}\label{eq:Limitalpha-xbaralpha=k}
\lim_{n\to+\infty}[\alpha(0,t_n)-\bar{x}(\alpha(0,t_n))]=k>0.
\end{equation}
As in the argument in Case 3, one can apply \eqref{eq:diffbetalphaandbarx} to show that if $\liminf_{n\to+\infty}\bar{x}(\alpha(0,t_n))=-\infty$, then $\liminf_{n\to+\infty}[\alpha(0,t_n)-\bar{x}(\alpha(0,t_n))]=0$. Thus, \eqref{eq:Limitalpha-xbaralpha=k} indeed implies $\liminf_{n\to+\infty}\bar{x}(\alpha(0,t_n))>-\infty$. Together with $\limsup_{t\to\pm\infty}\bar{x}(\alpha(0,t))<+\infty$ that we have already shown in this case, there exists a constant $M>0$ such that
\[
	|\bar{x}(\alpha(0,t_n))|\leq M.
\]
Hence, evaluating \eqref{eq:charater relation} at $t=t_n$, and then passing to the limit as $n\to +\infty$, we obtain
\[
0<k=\lim_{n\to+\infty}[\alpha(0,t_n)-\bar{x}(\alpha(0,t_n))]=\lim_{n\to +\infty} \left( - \frac{4}{t_n^2}\bar{x}(\alpha(0, t_n))- \frac{4}{t_n}\bar{u}(\bar{x}(\alpha(0, t_n))) \right)=0,
\]
which is a contradiction. Therefore, the limit \eqref{eq:differencealphaxt} holds at $x=0$. Finally, the limit \eqref{eq:grwothalphaxt} at $x=0$ follows immediately from \eqref{eq:charater relation} and \eqref{eq:differencealphaxt}. 
In a completely same manner, one can also verify \eqref{eq:differencealphaxt} and \eqref{eq:grwothalphaxt} at $x=\bar{\mu}(\mathbb{R})$, namely $\lim\limits_{t\to \pm\infty}
[\alpha(\bar{\mu}(\mathbb{R}),t)-\bar{x}(\alpha(\bar{\mu}(\mathbb{R}),t))]=\bar{\mu}(\mathbb{R})$ and $\lim\limits_{t\to \pm\infty} \frac{4}{t^2}\alpha(\bar{\mu}(\mathbb{R}),t) =\lim\limits_{t\to \pm\infty} \frac{4}{t^2}\bar{x}(\alpha(\bar{\mu}(\mathbb{R}),t)) = 0$.
%
\end{proof}

Next, we consider the following difference
\begin{gather}\label{eq:diff}
u(x,t)-\frac{t}{2} v\left(\frac{4x}{t^2}\right)=\left\{
\begin{split}
&u(x,t),\quad x<0,\\
&u(x,t)-\frac{2}{t}x,\quad 0\leq x\leq \frac{t^2}{4}\bar{\mu}(\mathbb{R}),\\
&u(x,t)-\frac{t}{2}\bar{\mu}(\mathbb{R}),\quad x>\frac{t^2}{4}\bar{\mu}(\mathbb{R}).
\end{split}
\right.
\end{gather}
Evaluating \eqref{eq:diff} at $x=y(\alpha,t)$, and using the explicit formulae \eqref{eq:xbarbeta} and \eqref{eq:measuresolutionu} for $y$ and $u$, we eventually obtain
\begin{gather}\label{eq:difference}
u(y(\alpha,t),t)-\frac{t}{2} v\left(\frac{4y(\alpha,t)}{t^2}\right)=\left\{
\begin{split}
&\bar{u}(\bar{x}(\alpha))+\frac{t}{2}[\alpha-\bar{x}(\alpha)],\quad y(\alpha,t)<0,\\
&-\frac{2}{t}\bar{x}(\alpha)-\bar{u}(\bar{x}(\alpha)),\quad 0\leq y(\alpha,t)\leq \frac{t^2}{4}\bar{\mu}(\mathbb{R}),\\
&\bar{u}(\bar{x}(\alpha))+\frac{t}{2}[\alpha-\bar{x}(\alpha)]-\frac{t}{2}\bar{\mu}(\mathbb{R}),\quad y(\alpha,t)>\frac{t^2}{4}\bar{\mu}(\mathbb{R}).
\end{split}
\right.
\end{gather}
One can see from the expressions above and also from the  proof of Lemma \ref{lem:important} that there are two important families of points (before our scaling procedure), namely $x=0$ and $x=\frac{t^2}{4}\bar{\mu}(\mathbb{R})$, which correspond to 0 and $\bar{\mu}(\mathbb{R})$ after the scaling. Accordingly, we define the following two functions (of $t$) for later use:
\begin{definition}\label{def:alpha01}
	\begin{align}\label{eq:alpha01}
		\alpha_l(t):=\sup\{\alpha:~y(\alpha,t)<0\},\quad\mbox{and}\quad \alpha_r(t):=\inf\left\{\alpha:~y(\alpha,t)>\frac{t^2}{4}\bar{\mu}(\mathbb{R})\right\}.
	\end{align}
\end{definition}
\noindent
Since $y$ is continuous, the above definition actually implies that
\begin{equation}\label{eq:y(alpha_0,t)=0_and_y(alpha_1,t)=t^2barmu/4}
	\begin{aligned}
		y(\alpha_l(t),t)=0,\quad\mbox{and}\quad y(\alpha_r(t),t)=\frac{t^2}{4}\bar{\mu}(\mathbb{R}).
	\end{aligned}
\end{equation}
Hence, it follows from the explicit formula \eqref{eq:xbarbeta} for $y(\alpha,t)$ that
\begin{align}\label{eq:pro1}
	\bar{x}(\alpha_l(t))+\bar{u}(\bar{x}(\alpha_l(t)))t+\frac{t^2}{4}[\alpha_l(t)-\bar{x}(\alpha_l(t))]=0,
\end{align}
and
\begin{align}\label{eq:pro2}
	\bar{x}(\alpha_r(t))+\bar{u}(\bar{x}(\alpha_r(t)))t+\frac{t^2}{4}[\alpha_r(t)-\bar{x}(\alpha_r(t))]=\frac{t^2}{4}\bar{\mu}(\mathbb{R}).
\end{align}
We first prove two important facts about $\alpha_l$ and $\alpha_r$:
\begin{lemma}\label{lmm:importanddifference} 
For $\alpha_l(t)$ and $\alpha_r(t)$ defined by Definition \ref{def:alpha01}, we have
\begin{equation}\label{eq:Limitpminftyalpha_o-xbaralpha_0}
\lim_{t\to\pm\infty}[\alpha_l(t)-\bar{x}(\alpha_l(t))]=0,
\end{equation}
and 
\begin{align}\label{eq:differencelimitalpha1}
\lim_{t\to\pm\infty} \left[\alpha_r(t)-\bar{x}(\alpha_r(t))-\bar{\mu}(\mathbb{R})\right]=0.
\end{align}
\end{lemma}

\begin{proof}
Using Definition~\ref{def:pseudo} and Definition~\ref{def:alpha01}, one can actually show that
\begin{align}\label{eq:relation}
\alpha(0,t)=\alpha_l(t).
\end{align}
Therefore, the limit \eqref{eq:Limitpminftyalpha_o-xbaralpha_0} follows immediately from Lemma~\ref{lem:important}, namely \eqref{eq:differencealphaxt} at $x=0$.
	
Next, we prove \eqref{eq:differencelimitalpha1}. Although the proof is similar to Case 4 of the proof of Lemma~\ref{lem:important}, we will provide the details here for completeness.
It follows from \eqref{eq:pro2} that
\begin{align}\label{eq:alpha1t}
\alpha_r(t)-\bar{x}(\alpha_r(t))-\bar{\mu}(\mathbb{R})=-\frac{4}{t^2}\bar{x}(\alpha_r(t))-\frac{4}{t}\bar{u}(\bar{x}(\alpha_r(t))).
\end{align}
First of all, we claim that
\begin{claim}\label{claim:Limitinfxbaralpha_1>-infty}
	\[
		\liminf_{t\to\pm\infty}\bar{x}(\alpha_r(t)) > -\infty.
	\]	
\end{claim}
\begin{proof}
	Seeking for a contradiction, we first assume that there exists a sequence $t_n\to \pm \infty$ such that $\bar{x}(\alpha_r(t_n))\to -\infty$. 
	Then by \eqref{eq:alpha1t} and the boundedness of $\bar{u}$, we have
	\[
	\liminf_{n\to +\infty}[\alpha_r(t_n)-\bar{x}(\alpha_r(t_n))-\bar{\mu}(\mathbb{R})]\geq 0.
	\]
	On the other hand, it follows from \eqref{eq:barx1} that
	\begin{equation}\label{eq:differencealpha} 
		\bar{\mu}((-\infty, \bar{x}(\alpha_r(t_n))))\leq \alpha_r(t_n) - \bar{x}(\alpha_r(t_n)) \leq \bar{\mu}((-\infty, \bar{x}(\alpha_r(t_n))]),
	\end{equation}
	and hence 
	\[
	\lim_{n\to +\infty}[\alpha_r(t_n)-\bar{x}(\alpha_r(t_n))-\bar{\mu}(\mathbb{R})]=-\bar{\mu}(\mathbb{R}) <0,
	\]
	which is a contradiction.
\end{proof} 
Now, we are going to prove \eqref{eq:differencelimitalpha1} by using Claim~\ref{claim:Limitinfxbaralpha_1>-infty}. Seeking for a contradiction, we assume that \eqref{eq:differencelimitalpha1} does not hold. It follows from \eqref{eq:barx1} that $\alpha_r(t)-\bar{x}(\alpha_r(t))-\bar{\mu}(\mathbb{R})\leq0$ for all $t\in\mathbb{R}$, so there exists a negative number $m$ and a sequence $t_n\to+\infty$ such that
\[
\lim_{n\to+\infty}[\alpha_r(t_n)-\bar{x}(\alpha_r(t_n))-\bar{\mu}(\mathbb{R})]=m<0.
\]
One can also apply \eqref{eq:barx1} to show that $\limsup_{n\to+\infty}\bar{x}(\alpha_r(t_n))<+\infty$.
Together with Claim~\ref{claim:Limitinfxbaralpha_1>-infty}, there exists a constant $M>0$ such that $|\bar{x}(\alpha_r(t_n))|<M$ for any $n$. Evaluating \eqref{eq:alpha1t} at $t=t_n$, and then passing to the limit as $n\to +\infty$, we have, due to the boundedness of $\bar{x}(\alpha_r(t_n))$ and $\bar{u}$,
\[
0>m=\lim_{n\to+\infty}[\alpha_r(t_n)-\bar{x}(\alpha_r(t_n))-\bar{\mu}(\mathbb{R})]=\lim_{n\to+\infty}\left(-\frac{4}{t_n^2}\bar{x}(\alpha_r(t_n))-\frac{4}{t_n}\bar{u}(\bar{x}(\alpha_r(t_n)))\right)=0,
\]
which is a contradiction. Therefore, \eqref{eq:differencelimitalpha1} actually holds.
%
%
\end{proof}
\begin{remark}

It follows from Definition~\ref{def:alpha01} and the continuity of $y$ that 
\[
\alpha_r(t)=\sup\left\{\alpha\in\mathbb{R}:~~\bar{\mu}(\mathbb{R})=\frac{4}{t^2}y(\alpha,t)\right\}.
\]
Hence, the limit \eqref{eq:differencelimitalpha1} in Lemma \ref{lmm:importanddifference} just means that the same conclusion in Lemma \ref{lem:important} (i.e., the limit \eqref{eq:differencealphaxt} at $x=\bar{\mu}(\mathbb{R})$) also holds for $\alpha_r(t)$, which is defined by using ``$\sup$''. In contrast, $\alpha\left(\bar{\mu}(\mathbb{R}),t\right)$ is defined by using ``$\inf$'' in Definition \ref{def:pseudo}.
Indeed, this holds true for any $x$: that is Lemma~\ref{lem:important} still holds if we change the definition of $\alpha(x,t)$ in Definition \ref{def:pseudo} to $\alpha(x,t):=\sup \{\alpha\in\mathbb{R}:\;x=$ $\frac{4}{t^2}y(\alpha,t)\}$. As one can see from the proof of Lemma~\ref{lem:important} that all the arguments are based on \eqref{eq:charater relation} and \eqref{eq:diffbetalphaandbarx} (or \eqref{eq:barx1}),  and do not depend on the choice of $``\inf"$ or $``\sup"$ in the definition of $\alpha(x,t)$. In a certain sense, using different variants of definition of pseudo inverse function $\alpha$ does not really affect some limiting results, such as Lemma~\ref{lem:important}.
\end{remark}
Now we are ready to prove Theorem \ref{thm:mainthm1}.
\begin{proof}[Proof Theorem \ref{thm:mainthm1}]
First of all, the limit \eqref{eq:v} follows directly from \eqref{eq:limit} and \eqref{eq:differencealphaxt} in Lemma  \ref{lem:important}. The proofs of \eqref{eq:Linfty}-\eqref{eq:singular part} will be given as follows.
	
\noindent\textbf{Proof of \eqref{eq:Linfty}:}
For any fixed $t\in\mathbb{R}$, the mapping $\alpha\mapsto y(\alpha,t)$ is a surjective mapping from $\mathbb{R}$ to $\mathbb{R}$, so in order to show \eqref{eq:Linfty}, it suffices to prove that as $t\to\pm\infty$, the right hand side of \eqref{eq:difference} is (uniformly in $\alpha$) of order $o(t)$.
We consider the following three cases separately.

\textbf{Case 1.} Consider $y(\alpha,t)<0$. We want to show that as $t\to\pm\infty$,
\[
\bar{u}(\bar{x}(\alpha))+\frac{t}{2}(\alpha-\bar{x}(\alpha))=o(t).
\]
Since $\bar{u}$ is bounded, it suffices to show that 
\[
\sup\left\{|\alpha-\bar{x}(\alpha)|;\;y(\alpha,t)<0\right\}\to 0 ~\textrm{ as }~t\to\pm\infty.
\]
Consider $\alpha_l$ defined by \eqref{eq:alpha01}. Due to \eqref{eq:barx1}, we know that $\alpha-\bar{x}(\alpha)$ is a  continuous, non-negative and non-decreasing function of $\alpha$, so for all $\alpha$ satisfying $y(\alpha,t)<0$, we have
\[
0\leq \alpha-\bar{x}(\alpha)\leq \alpha_l(t)-\bar{x}(\alpha_l(t)).
\]
In other words,
\begin{equation}\label{eq:estiatleft}
\left\|u(x,t)- \frac{t}{2}v\left(\frac{4x}{t^2}\right)\right\|_{L^{\infty}((-\infty,0))}\leq
\frac{t}{2}	\left(\alpha_l(t)-\bar{x}(\alpha_l(t)) \right) +\|\bar{u}\|_{L^{\infty}(\mathbb{R})}.
\end{equation}
Hence, by \eqref{eq:Limitpminftyalpha_o-xbaralpha_0}, this completes the proof of Case 1.

\textbf{Case 2.} Consider $y(\alpha,t)>\frac{t^2}{4}\bar{\mu}(\mathbb{R})$.
We want to show that as $t\to\pm\infty$,
\[
\bar{u}(\bar{x}(\alpha))+\frac{t}{2}[\alpha-\bar{x}(\alpha)]-\frac{t}{2}\bar{\mu}(\mathbb{R})=o(t),
\]
which is equivalent to show that
\[
\sup\left\{|\alpha-\bar{x}(\alpha)-\bar{\mu}(\mathbb{R})|;\;y(\alpha,t)>\frac{t^2}{4}\bar{\mu}(\mathbb{R})\right\} \to 0 ~\textrm{ as }~t\to\pm\infty.
\] 
Consider the function $\alpha_r$ defined by \eqref{eq:alpha01}.
Due to \eqref{eq:barx1} again, we have, for any $\alpha$ satisfying $y(\alpha,t)>\frac{t^2}{4}\bar{\mu}(\mathbb{R})$,
\[
0\geq \alpha-\bar{x}(\alpha)-\bar{\mu}(\mathbb{R}) \geq \alpha_r(t)-\bar{x}(\alpha_r(t))-\bar{\mu}(\mathbb{R}),
\]
where the last inequality follows from the monotonicity of $\alpha-\bar{x}(\alpha)$.
In other words,
\begin{equation}\label{eq:estiatright}
	\left\|u(x,t)- \frac{t}{2}v\left(\frac{4x}{t^2}\right)\right\|_{L^{\infty}((\frac{t^2}{4}\bar{\mu}(\mathbb{R}),+\infty))}\leq
	\frac{t}{2}	\left[\bar{\mu}(\mathbb{R})-(\alpha_r(t)-\bar{x}(\alpha_r(t)) )\right] +\|\bar{u}\|_{L^{\infty}(\mathbb{R})}.
\end{equation}
Hence, by \eqref{eq:differencelimitalpha1}, this completes the proof of Case 2. 

\textbf{Case 3.} Consider $0\leq y(\alpha,t)\leq \frac{t^2}{4}\bar{\mu}(\mathbb{R})$. In this case, we want to show that as $t\to\pm\infty$,
\[
-\frac{2}{t}\bar{x}(\alpha)-\bar{u}(\bar{x}(\alpha)) = o(t),
\]
uniformly in $\alpha$.
Since $\bar{u}$ is bounded, it suffices to show that 
\[
\sup\left\{|\bar{x}(\alpha)|;\;0\leq y(\alpha,t)\leq \frac{t^2}{4}\bar{\mu}(\mathbb{R})\right\} = o(t^2) ~\textrm{ as }~t\to\pm\infty.
\]
Since $\bar{x}$ is a non-decreasing function of $\alpha$, it suffices to show $\max\left\{ |\bar{x}(\alpha_l(t))|, |\bar{x}(\alpha_r(t))| \right\} = o(t^2)$ as $t\to\pm\infty$. In other words, it suffices to show $\lim_{t \to \pm\infty}\frac{\bar{x}(\alpha_l(t))}{t^2} =0$ and $\lim_{t \to \pm\infty}\frac{\bar{x}(\alpha_r(t))}{t^2} =0$. 
Indeed, it follows immediately from \eqref{eq:pro1} and \eqref{eq:Limitpminftyalpha_o-xbaralpha_0} that $\lim_{t \to \pm\infty}\frac{\bar{x}(\alpha_l(t))}{t^2} =0$, since $\bar{u}$ is bounded. Similarly, it follows from \eqref{eq:differencelimitalpha1}, \eqref{eq:alpha1t} and the boundedness of $\bar{u}$ that $\lim_{t \to \pm\infty}\frac{\bar{x}(\alpha_r(t))}{t^2} =0$ as well. This shows Case 3, and hence, completes the proof of \eqref{eq:Linfty}.

\noindent\textbf{Proof of \eqref{eq:L2}:} According to \eqref{eq:diff}, we have
\begin{multline}\label{eq:I123}
\left\|u_x(x,t)-\partial_x\left[\frac{t}{2} v\left(\frac{4x}{t^2}\right)\right]\right\|_{L^2}^2=\int_{(-\infty,0)}u_x^2(x,t)\di x+\int_{(0,\frac{t^2}{4}\bar{\mu}(\mathbb{R}))}\left[u_x(x,t)-\frac{2}{t}\right]^2\di x\\
+\int_{(\frac{t^2}{4}\bar{\mu}(\mathbb{R}),+\infty)}u_x^2(x,t)\di x=:I_1+I_2+I_3.
\end{multline}
We will apply the change of variables to estimate $I_i$ for $i=1$, $2$, $3$.
Following \cite{gao2021regularity}, we define $B_t^L$ as follows:
\[
B_t^L=\{\alpha:~y_\alpha(\alpha,t)>0\}.
\] 
The following results hold (see \cite[Eq. (2.43)]{gao2021regularity} for instance):
\[
u_x^2(y(\alpha,t),t)y_\alpha(\alpha,t)=f(\alpha) := 1-\bar{x}'(\alpha),\quad \alpha\in B_t^L,
\]
and
\[
	\mathcal{L}(\mathbb{R}\setminus y(B_t^L,t))=0, 
\]
where $\mathcal{L}$ is the Lebesgue measure. Combining the above two facts yields 
\begin{equation}\label{eq:termI1}
\begin{aligned}
I_1=&\left\|u_x(x,t)-\partial_x\left[\frac{t}{2} v\left(\frac{4x}{t^2}\right)\right]\right\|^2_{L^2(( -\infty,0))} =
\int_{(-\infty,0)}u_x^2(x,t)\di x=\int_{(-\infty,0)\cap y(B_t^L,t)}u_x^2(x,t)\di x\\
=&\int_{(-\infty,\alpha_l(t))\cap B_t^L}u_x^2(y(\alpha,t),t)y_\alpha(\alpha,t)\di \alpha
\leq \int_{(-\infty,\alpha_l(t))}f(\alpha)\di \alpha=\alpha_l(t)-\bar{x}(\alpha_l(t))\to 0,
\end{aligned}
\end{equation}
as $t\to\pm\infty$, because of \eqref{eq:Limitpminftyalpha_o-xbaralpha_0}. 

For $I_3$, we similarly obtain
\begin{equation}\label{eq:termI3}
\begin{aligned}
I_3=&\left\|u_x(x,t)-\partial_x\left[\frac{t}{2} v\left(\frac{4x}{t^2}\right)\right]\right\|^2_{L^2(( \frac{t^2}{4}\bar{\mu}(\mathbb{R}), +\infty))} \\
=&\int_{(\frac{t^2}{4}\bar{\mu}(\mathbb{R}),+\infty)}u_x^2(x,t)\di x=\int_{(\alpha_r(t),+\infty)\cap B_t^L}u_x^2(y(\alpha,t),t)y_\alpha(\alpha,t)\di \alpha\\
\leq& \int_{(\alpha_r(t),+\infty)}f(\alpha)\di \alpha=\bar{\mu}(\mathbb{R})-[\alpha_r(t)-\bar{x}(\alpha_r(t))]\to 0,
\end{aligned}
\end{equation}
as $t\to\pm\infty$, because of \eqref{eq:differencelimitalpha1}.

For $I_2$, we have 
\begin{equation}\label{eq: term I2}
\begin{aligned}
I_2&=\int_{(0,\frac{t^2}{4}\bar{\mu}(\mathbb{R}))}\left[u_x(x,t)-\frac{2}{t}\right]^2\di x\\
&= \int_{(\alpha_l(t),\alpha_r(t))\cap B_t^L}u_x^2(y(\alpha,t),t)y_\alpha(\alpha,t)\di \alpha + 
\int_{(\alpha_l(t),\alpha_r(t))} \left[ \frac{4}{t^2}y_\alpha(\alpha,t)-\frac{4}{t}u_x(y(\alpha,t),t)y_\alpha(\alpha,t) \right]\di \alpha
 \\
&\leq \int_{(\alpha_l(t),\alpha_r(t))}f(\alpha)\di \alpha+\frac{4}{t^2}[y(\alpha_r(t),t)-y(\alpha_l(t),t)]-\frac{4}{t}[u(y(\alpha_r(t),t),t)-u(y(\alpha_l(t),t),t)]\\
&=[\alpha_r(t)-\bar{x}(\alpha_r(t))- \bar{\mu}(\mathbb{R})]  -  [\alpha_l(t)-\bar{x}(\alpha_l(t))] 
+ 2\left[ \bar{\mu}(\mathbb{R})-\frac{2}{t}u\left(\frac{t^2}{4}\bar{\mu}(\mathbb{R}),t\right)\right]   +  \frac{4}{t}u(0,t).\\
\end{aligned}
\end{equation}
Combining \eqref{eq:v}, \eqref{eq:Limitpminftyalpha_o-xbaralpha_0} and \eqref{eq:differencelimitalpha1}, we finally have $I_2\to0$ as $t\to\pm\infty.$
This completes the proof of \eqref{eq:L2}. 

\noindent\textbf{Proof of \eqref{eq:singular part}:}
It follows directly from \eqref{eq:L2} that  $\lim_{t\to\pm\infty} \|u_x(\cdot, t)\|_{L^2}^2 = \left\|\partial_x\left[\frac{t}{2} v\left(\frac{4x}{t^2}\right)\right]\right\|_{L^2}^2 = \bar{\mu}(\mathbb{R})$. On the other hand, the energy conservation implies that  $\|u_x(\cdot, t)\|_{L^2}^2  + {\mu}_{s}(t)(\mathbb{R})=\mu(t)(\mathbb{R}) = \bar{\mu}(\mathbb{R})$ for all $t\in \mathbb{R}$, so passing to the limit as $t\to\pm\infty$, we finally obtain \eqref{eq:singular part}. 
This completes the whole proof of Theorem~\ref{thm:mainthm1}.
\end{proof}
\begin{remark}\label{rem:limit for fixed x}
Using the limit \eqref{eq:Linfty} and the continuity of $v$, we have the following pointwise estimates: for any fixed $x\in\mathbb{R}$,
\begin{align*}
\lim_{t\to\pm\infty} \frac{u(x,t)}{t}= \lim_{t\to\pm\infty}  \frac{1}{2}v\left(\frac{4x}{t^2}\right)= \frac{1}{2}v(0)=0.
\end{align*}
Hence, we know that for each fixed $x\in\mathbb{R}$, the growth rate of  $u(x,t)$ is { at most} of order $o(t)$. However, the { precise} growth rate may be very close to $O(t)$; see Theorem \ref{thm:mainthm3} for more details.
\end{remark}

\section{Optimal error estimates and large time pointwise behavior}\label{sec:decayrate}
The results in Theorem \ref{thm:mainthm1} only show the asymptotic expansions without any decay rate estimates for the error terms. In this section, we are going to show the optimal error estimates for $\left\|u(\cdot,t)-\frac{t}{2} v\left(\frac{4x}{t^2}\right)\right\|_{L^\infty(\mathbb{R})}$ and $\left\|u_x(\cdot,t)-\partial_x\left[\frac{t}{2} v\left(\frac{4x}{t^2}\right)\right]\right\|_{L^2(\mathbb{R})}$ and large time pointwise behavior under natural assumptions on the tails of initial energy measure $\bar{\mu}$, namely proving Theorem \ref{thm:mainthm2} and \ref{thm:mainthm3}. 

Regarding these error estimates, the decay rates of the tails of initial energy measure $\bar{\mu}$ play an essential role, and will heavily affect the growth rates of $\alpha_l(t)$ and $\alpha_r(t)$, which are defined by \eqref{eq:alpha01}. Since the growth rates of $\bar{x}(\alpha_l(t))$ and $\bar{x}(\alpha_r(t))$ are essentially the same as that of $\alpha_l(t)$ and $\alpha_r(t)$ respectively, it follows from \eqref{eq:pro1} and \eqref{eq:alpha1t} that the decay rate of $\alpha_l(t) - \bar{x}(\alpha_l(t))$ and  $\alpha_r(t) - \bar{x}(\alpha_r(t))-\bar{\mu}(\mathbb{R})$ will be indirectly determined by the tails of the initial energy measure $\bar{\mu}$; see \eqref{eq:idea} below for a heuristic.
Furthermore, one can see from the proof of Theorem \ref{thm:mainthm1} that the two quantities $\alpha_l(t) - \bar{x}(\alpha_l(t))$ and  $\alpha_r(t) - \bar{x}(\alpha_r(t))-\bar{\mu}(\mathbb{R})$ essentially determine the size of $\left\|u(x,t)-\frac{t}{2} v\left(\frac{4x}{t^2}\right)\right\|_{L^\infty(\mathbb{R})}$ and $\left\|u_x(x,t)-\partial_x\left[\frac{t}{2} v\left(\frac{4x}{t^2}\right)\right]\right\|_{L^2(\mathbb{R})}$. For example, if we have the estimate $\alpha_l(t) = O(t^{1+\theta})$ for some $0\leq \theta<1$, then $\bar{x}(\alpha_l(t))=O(t^{1+\theta})$, and hence, it follows from \eqref{eq:pro1} that  
$\alpha_l(t) - \bar{x}(\alpha_l(t))= O(t^{\theta-1})$. In general, $\alpha_l(t)$ or $\alpha_r(t)$ cannot perform better than $O(t)$, if no further condition is assumed on the initial data; as we can see from Example \ref{exm:k} that for the initial data $(\bar{u},\bar{\mu})=(1,\delta_0)$, we actually have $\alpha_l(t)=-t$. Now we will turn our attention to obtaining the growth rate control for $\alpha_l(t)$ and $\alpha_r(t)$.

For the growth rate of $\alpha_l(t)$, the  limit
\[
A:=\lim_{x\to -\infty}|x|{\bar{\mu}((-\infty, -|x|))}
\]
is a crucial quantity, provided that this limit exists { or equals to $+\infty$}. As we will see, when $A=+\infty$, the effect of the left tail of the energy measure $\bar{\mu}$ will defeat the effect of $\bar{u}(-\infty)$; on the other hand, when $A=0$, the effect of the value $\bar{u}(-\infty)$ is more dominant. The case $A>0$ can be regarded as the critical case, in which the effects of the left tail of the energy measure $\bar{\mu}$ and $\bar{u}(-\infty)$ are { comparable}.  

To illustrate the idea, let us consider the scenario that $\alpha_l(t)\to - \infty$ as $t \to +\infty$. Heuristically, it follows from \eqref{eq:barx1} and \eqref{eq:pro1} that for any sufficiently large $t$,
\begin{equation}\label{eq:idea}
\begin{aligned}
	\bar{x}(\alpha_l(t))\bar{\mu}((-\infty, \bar{x}(\alpha_l(t)))) &\approx  \bar{x}(\alpha_l(t))\left[\alpha_l(t)-\bar{x}(\alpha_l(t))\right]\\
	&=-4\left(\frac{\bar{x}(\alpha_l(t))}{t}\right)^2-4\left(\frac{\bar{x}(\alpha_l(t))}{t}\right)\bar{u}(\bar{x}(\alpha_l(t))).
\end{aligned}
\end{equation}
As $t \to +\infty$, the left hand side will tend to the numerical value $A$, and hence, the value of $A$ indicates the growth rate of $\alpha_l(t)$. More precisely, $A=+\infty$ (or $0$) means that  $|\bar{x}(\alpha_l(t))|$ grows faster (or no faster) than $t$; on the other hand, a finite value of $A$ means that $\bar{x}(\alpha_l(t))=O(t)$ and $\bar{u}(-\infty)$ becomes significant. Similar analysis with more subtleties will lead to the explicit growth rate of $\alpha_l(t)$, provided that more information (e.g., \eqref{eq:leftcondition1} or \eqref{eq:leftcondition2}) about the tails is known; see Lemma~\ref{lemma:properties} and Remark~\ref{rem:sublinear} for further details, in which we will prove the growth rates for general $\alpha_x(t)$ to be defined in \eqref{eq:alphaxt}, and indeed $\alpha_l(t)$ is a special case of $\alpha_x(t)$, namely $\alpha_l(t)=\alpha_x(t)$ when $x=0$.

In the propositions below, we will further impose conditions on $\lim_{x\to -\infty}|x|^{1-\theta}{\bar{\mu}((-\infty, -|x|^{1+\theta}))} $ for some $0 \leq \theta <1$, as we will see that it is more favourable for our presentations.
One can of course  impose conditions on $\lim_{x\to -\infty}$$|x|^{\beta_1}$ ${\bar{\mu}((-\infty, -|x|^{\beta_2}))}$; for example, if the limit
\[
B:=\lim_{x\to -\infty}|x|^{\beta_1}{\bar{\mu}((-\infty, -|x|^{\beta_2}))} \in [0,\infty)
\] 
for some positive constants $\beta_1$ and $\beta_2$,  then by a change of variables, one can verify that 
\begin{equation*}
\lim_{x\to -\infty}|x|^{1-\theta}{\bar{\mu}((-\infty, -|x|^{1+\theta}))} =\lim_{y\to -\infty}|y|^{\beta_1}{\bar{\mu}((-\infty, -|y|^{\beta_2}))}  = B\geq 0
\end{equation*}
where $\theta=\frac{\beta_2-\beta_1}{\beta_1+\beta_2}$, provided that $\beta_2\geq \beta_1$. On the other hand, 
if $\beta_1 >\beta_2$, then $$\lim_{x\to -\infty}|x|{\bar{\mu}((-\infty, -|x|))} = \lim_{x\to -\infty}|x|^{1-\frac{\beta_1}{\beta_2}} \cdot\lim_{x\to -\infty}|x|^{\frac{\beta_1}{\beta_2}}{\bar{\mu}((-\infty, -|x|))} = 0.$$
Hence, for $\beta_1>\beta_2$, the condition on $\lim_{x\to -\infty}|x|^{\beta_1}{\bar{\mu}((-\infty, -|x|^{\beta_2}))} $ can also be seen as a sort of  condition on $\lim_{x\to -\infty}|x|{\bar{\mu}((-\infty, -|x|))} $, which corresponds to $\theta=0$. Finally, we remark that all the discussions above apply equally well to $\alpha_r(t)$ with the assumptions on $\lim_{x\to +\infty}x{\bar{\mu}((x,+\infty))}$; see Lemma~\ref{lemma:propertiesright} for details. 


For any fixed $x$, we define 
\begin{align}\label{eq:alphaxt}
\alpha_x(t):=\alpha\left(\frac{4}{t^2}x,t\right)=\inf\left\{\alpha:~x=y(\alpha,t)\right\},
\end{align}
where the pseudo inverse function $\alpha$ is defined in \eqref{eq:defalpha}. Then it follows from \eqref{eq:alpha01}, \eqref{eq:alphaxt} and the continuity of $y$ that $\alpha_l(t)=\alpha_0(t)$ for all $t\neq 0$. 
In the following, we will see that for any fixed $x\in\mathbb{R}$, all the functions $\alpha_x(t)$ behaves like $\alpha_l(t)$ as $t\to +\infty$; heuristically, this is because $\frac{4}{t^2}x \to 0$ as $t\to +\infty$. 
Now, let us start with the following facts (i.e., Lemma~\ref{lemma:properties} and Lemma~\ref{lemma:propertiesright}) for $\alpha_x(t)$ and $\alpha_r(t)$. 
\begin{lemma}\label{lemma:properties}
Let $(u,\mu)$ be a conservative solution to the Hunter-Saxton equation subject to an initial data $(\bar{u},\bar{\mu})\in\mathcal{D}$. Let $0\leq \theta_1<1$ and $A_1\geq0$ be some constants. Then for any $x\in\mathbb{R}$, the function $\alpha_x(t)$ defined by \eqref{eq:alphaxt} satisfies the following statements:
\begin{enumerate}
\item[(i)] If $A_1>0$, $0<\theta_1<1$ and $\bar{\mu}$ satisfies \eqref{eq:leftcondition1}, then we have
\begin{align}\label{eq:leftlimit1}
-\left(\frac{A_1}{4}\right)^{\frac{1+\theta_1}{2}}\leq \liminf_{t\to+\infty}\frac{\bar{x}(\alpha_x(t))}{t^{1+\theta_1}}\leq \limsup_{t\to+\infty}\frac{\bar{x}(\alpha_x(t))}{t^{1+\theta_1}}\leq 0.
\end{align}
If $A_1>0$, $\theta_1=0$ and $\bar{\mu}$ satisfies \eqref{eq:leftcondition1}, then we have
\begin{align}\label{eq:leftlimit11}
-\infty<\liminf_{t\to+\infty}\frac{\bar{x}(\alpha_x(t))}{t}\leq \limsup_{t\to+\infty}\frac{\bar{x}(\alpha_x(t))}{t}\leq 0.
\end{align}

\item[(ii)] Assume that $\bar{u}$ satisfies $\lim_{x\to-\infty}\bar{u}(x)=\bar{u}(-\infty)\in\mathbb{R}$  and $\bar{\mu}$ satisfies  \eqref{eq:leftcondition2}.
Then 
\begin{enumerate}
\item[(a)] if $A_1>0$ and $0\leq \theta_1<1$, we have 
\begin{equation}\label{eq:leftlimit2}
\lim_{t\to+\infty}\frac{\bar{x}(\alpha_x(t))}{t^{1+\theta_1}} = \left\{
\begin{aligned}
&-\left(\frac{A_1}{4}\right)^{\frac{1+\theta_1}{2}}~\textrm{ for }~0<\theta_1<1;\\
&-\frac{1}{2}\bar{u}(-\infty) -\frac{1}{2}\sqrt{\bar{u}^2(-\infty) +{A_1}}~\textrm{ for }~\theta_1=0;
\end{aligned}
\right.
\end{equation}
\item[(b)] 
if $A_1=0$, $0<\theta_1<1$  and $\bar{\mu}$ satisfies $\bar{\mu}((-\infty,z))>0$ for any $z\in\mathbb{R}$, we have
\begin{align}\label{eq:leftlimit3}
\lim_{t\to+\infty}\frac{\bar{x}(\alpha_x(t))}{t^{1+\theta_1}}=0;
\end{align}
if $A_1=0$, $\theta_1=0$ and $\bar{\mu}$ satisfies $\bar{\mu}((-\infty,z))>0$ for any $z\in\mathbb{R}$, we have
\begin{equation}\label{eq:leftlimit4}
\lim_{t\to+\infty}\frac{\bar{x}(\alpha_x(t))}{t}=\left\{
\begin{aligned}
& -\bar{u}(-\infty)~~\textrm{ for }~~\bar{u}(-\infty)>0;\\
&0~\textrm{ for }~\bar{u}(-\infty)\leq0.
\end{aligned}\right.
\end{equation}
\end{enumerate}

\item[(iii)] If $\bar{\mu}$ satisfies \eqref{eq:leftcondition3} for some constant $\ell\in\mathbb{R}$, then we have
\begin{equation}\label{eq:compact10}
\lim_{t\to+\infty}\bar{x}(\alpha_x(t))=\left\{
\begin{aligned}
&-\infty~\textrm{ if }~\bar{u}(-\infty)>0,\\
&\ell~\textrm{ if }~\bar{u}(-\infty)<0,
\end{aligned}
\right.
\end{equation}
and the limit \eqref{eq:leftlimit4} also holds.

\end{enumerate}
Moreover, all of the above results will also hold, if we replace $\bar{x}(\alpha_x(t))$ by $\alpha_x(t)$.
\end{lemma}
\begin{proof}
For any fixed $x\in\mathbb{R}$, evaluating \eqref{eq:xbarbeta} at $\alpha=\alpha_x(t)$ and using the fact that $y(\alpha_x(t),t)=x$, we have
\begin{align}\label{eq:pro11}
\frac{4x}{t^{1+\theta_1}}-\frac{4}{t^{1+\theta_1}}\bar{x}(\alpha_x(t))-\frac{4}{t^{\theta_1}}\bar{u}(\bar{x}(\alpha_x(t)))=t^{1-\theta_1}[\alpha_x(t)-\bar{x}(\alpha_x(t))].
\end{align}
We first claim the following fact:
\begin{claim}\label{claim:alpha_x->infty}
	{
	Let $0\leq \theta_1<1$, and $\bar{\mu}$ satisfy $\bar{\mu}((-\infty,z))>0$ for any $z\in\mathbb{R}$. Then
	}
	\begin{align}\label{eq:negative}
		\alpha_x(t)\to-\infty~\textrm{ and }~\bar{x}(\alpha_x(t))\to-\infty~\textrm{ as }~t\to+\infty.
	\end{align}
\end{claim}
\begin{proof}
	Seeking for a contradiction, we assume that there exist a constant $M>0$ and a sequence of $t_n\to +\infty$ such that
	$\bar{x}(\alpha_x(t_n))\geq -M$ for all $n$.
	Since $\bar{\mu}$ is not compactly supported, we have
	\begin{equation}\label{eq:differencealphax}
		\liminf_{n\to +\infty}[\alpha_x(t_n)-\bar{x}(\alpha_x(t_n))] \geq \liminf_{n\to +\infty}\bar{\mu}((-\infty, \bar{x}(\alpha_x(t_n)))
		\geq \bar{\mu}((-\infty, -M))>0,
	\end{equation}
	where we used \eqref{eq:barx1} in the first inequality above. Evaluating \eqref{eq:pro11} at $t:=t_n$, and then passing to the limit superior as $n\to+\infty$, the left hand side satisfies
	\begin{equation}\label{eq:limsupLHS}
		\limsup_{n\to+\infty}\left[\frac{4x}{t_n^{1+\theta_1}}-\frac{4}{t_n^{1+\theta_1}}\bar{x}(\alpha_x(t_n))-\frac{4}{t_n^{\theta_1}}\bar{u}(\bar{x}(\alpha_x(t_n)))\right]\leq \limsup_{n\to+\infty}\frac{4M}{t_n^{1+\theta_1}}+4\|\bar{u}\|_{L^\infty}= 4\|\bar{u}\|_{L^\infty},
	\end{equation}
while the right hand side satisfies 
\[
\lim_{n\to+\infty}t_n^{1-\theta_1}[\alpha_x(t_n)-\bar{x}(\alpha_x(t_n))]=+\infty
\]
due to \eqref{eq:differencealphax},
which contradicts \eqref{eq:limsupLHS}. This shows that $\bar{x}(\alpha_x(t))\to-\infty$ as $t\to+\infty$. In addition, evaluating \eqref{eq:barx1} at $\alpha=\alpha_x(t)$, and then passing to the limit as $t\to +\infty$, as well as using the fact that $\bar{\mu}(\mathbb{R})<+\infty$, we also have 
$\alpha_x(t)\to -\infty$ as $t\to+\infty$.
\end{proof}

(i)  
For $0<\theta_1<1$, passing to the limit inferior in \eqref{eq:pro11} as $t\to +\infty$, and using the boundedness of $\bar{u}$, we have $\limsup_{t\to+\infty}\frac{\bar{x}(\alpha_x(t))}{t^{1+\theta_1}}\leq 0$ because the right hand of  \eqref{eq:pro11} is nonnegative by the first inequality in \eqref{eq:barx1}. Now, it remains to prove the first inequality in \eqref{eq:leftlimit1}. Seeking for a contradiction, we assume that 
\[
\liminf_{t\to+\infty}\frac{\bar{x}(\alpha_x(t))}{t^{1+\theta_1}}\leq -\left(\frac{A_1}{4}\right)^{\frac{1+\theta_1}{2}}-\epsilon,
\]
for some $\epsilon>0$. 
Then there exists a sequence $t_n \to +\infty$ such that $\bar{x}(\alpha_x(t_n)) \leq -\left(\frac{A_1}{4}\right)^{\frac{1+\theta_1}{2}} t^{1+\theta_1}_n - 
\frac{\epsilon}{2} t^{1+\theta_1}_n$ for all $n\in\mathbb{N}$. Then using \eqref{eq:barx1}  and \eqref{eq:leftcondition1}, we have for any sufficiently large $n$, 
\[
\alpha_x(t_n)-\bar{x}(\alpha_x(t_n))\leq \bar{\mu}((-\infty,\bar{x}(\alpha_x(t_n))])\leq \bar{\mu}\left(\left(-\infty, -\left(\frac{A_1}{4}\right)^{\frac{1+\theta_1}{2}} t^{1+\theta_1}_n\right)\right)\leq \frac{4}{t_n^{1-\theta_1}}\left( \left(\frac{A_1}{4}\right)^{\frac{1+\theta_1}{2}} +\frac{\epsilon}{4}\right).
\]
Hence, it follows from \eqref{eq:xbarbeta} that for any given $x\in\mathbb{R}$,
\begin{equation*}
\begin{aligned}
x=y(\alpha_x(t_n),t_n) &= \bar{x}(\alpha_x(t_n))+\bar{u}(\bar{x}(\alpha_x(t_n)))t_n+\frac{t_n^2}{4}(\alpha_x(t_n)-\bar{x}(\alpha_x(t_n)))\\
&\leq -\left(\frac{A_1}{4}\right)^{\frac{1+\theta_1}{2}} t^{1+\theta_1}_n -\frac{\epsilon}{2} t^{1+\theta_1}_n +\|\bar{u}\|_{L^{\infty}}t_n + \left(\frac{A_1}{4}\right)^{\frac{1+\theta_1}{2}}{t_n^{1+\theta_1}} +\frac{\epsilon}{4}t^{1+\theta_1}_n \\
&=-\frac{\epsilon}{4}t^{1+\theta_1}_n + \|\bar{u}\|_{L^{\infty}}t_n\to-\infty,
\end{aligned}
\end{equation*}
as $n\to+\infty$, which is a contradiction. 

For $\theta_1=0$, we separate the proof into two cases.

\textbf{Case 1.} 
Consider $\bar{\mu}((-\infty,z))>0$ for any $z\in\mathbb{R}$.
In particular, \eqref{eq:negative} holds with $\theta_1=0$, which implies
\[
\limsup_{t\to+\infty}\frac{\bar{x}(\alpha_x(t))}{t}\leq 0.
\]
Rearranging \eqref{eq:pro11} yields
\begin{align}\label{eq:pro111}
	-\frac{4\bar{x}(\alpha_x(t))}{t}=t[\alpha_x(t)-\bar{x}(\alpha_x(t))]+4\bar{u}(\bar{x}(\alpha_x(t)))-\frac{4x}{t}.
\end{align}
Next, we prove the first inequality in \eqref{eq:leftlimit11} by contradiction. Seeking for a contradiction, we assume that 
\[
\liminf_{t\to+\infty}\frac{\bar{x}(\alpha_x(t))}{t}=-\infty.
\]
Then there exists $t_n\to+\infty$ such that $\lim_{n\to+\infty}\frac{\bar{x}(\alpha_x(t_n))}{t_n}=-\infty$. Hence $\bar{x}(\alpha_x(t_n))<-t_n$ for any sufficiently large $n$. Evaluating \eqref{eq:pro111} at $t=t_n$, and then passing to the limit as $n\to +\infty$, we obtain
\begin{equation}
\begin{aligned}
+\infty=\lim_{n\to+\infty}\left(-\frac{4\bar{x}(\alpha_x(t_n))}{t_n}\right)=&\lim_{n\to+\infty}\left(t_n[\alpha_x(t_n)-\bar{x}(\alpha_x(t_n))]+4\bar{u}(\bar{x}(\alpha_x(t_n)))-\frac{4x}{t_n}\right)\\
\leq&\limsup_{n\to+\infty}t_n\bar{\mu}((-\infty,\bar{x}(\alpha_x(t_n))])+4\|\bar{u}\|_{L^\infty}\\
\leq&\limsup_{n\to+\infty}t_n\bar{\mu}((-\infty,-t_n))+4\|\bar{u}\|_{L^\infty}\leq A_1+4\|\bar{u}\|_{L^\infty},
\end{aligned}
\end{equation}
which is a contradiction.

\textbf{Case 2.} Consider $\bar{\mu}((-\infty,z_0))=0$ for some $z_0\in\mathbb{R}$. 
Then one can directly apply the results in part (iii), which will be proved in Appendix \ref{app:compact}, and obtain \eqref{eq:leftlimit4}, which immediately implies \eqref{eq:leftlimit11}. It is worth noting that the proof of part (iii) does not rely on any result in part (i), so there is no circular argument here indeed.

(ii)(a)
{ It follows from \eqref{eq:leftcondition2} and $A_1>0$ that $\bar{\mu}$ satisfies $\bar{\mu}((-\infty,z))>0$ for any $z\in\mathbb{R}$, so by Claim~\ref{claim:alpha_x->infty}, the limit \eqref{eq:negative} holds.}
Then we claim that there exist some positive constants $a$ and $b$ such that
\begin{align}\label{eq:leftlimclaim1}
\liminf_{t\to +\infty}\frac{4}{t^{1+\theta_1}}\bar{x}(\alpha_x(t))=-a<0~\textrm{ and }~ \limsup_{t\to +\infty}\frac{4}{t^{1+\theta_1}}\bar{x}(\alpha_x(t))=-b<0.
\end{align}
Since the condition \eqref{eq:leftcondition2} is stronger than \eqref{eq:leftcondition1}, we can apply part (i) and obtain \eqref{eq:leftlimit1}. Hence, we only need to prove $\limsup_{t\to +\infty}\frac{4}{t^{1+\theta_1}}\bar{x}(\alpha_x(t))=-b<0.$ 
Seeking for a contradiction, we assume that there exists a sequence $t_n \to +\infty$ such that
\begin{equation}\label{eq:SfaCLimit=0}
	\lim_{n\to +\infty}\frac{4}{t_n^{1+\theta_1}}\bar{x}(\alpha_x(t_n))=0.
\end{equation}
Then for any sufficiently large $n$, we have $\bar{x}(\alpha_x(t_n))>-t_n^{1+\theta_1}$, and hence, it follows from the first inequality in \eqref{eq:barx1} that
\[
\alpha_x(t_n)-\bar{x}(\alpha_x(t_n))\geq\bar{\mu}((-\infty, \bar{x}(\alpha_x(t_n))))\geq\bar{\mu}((-\infty, -t_n^{1+\theta_1})).
\]
For $0<\theta_1<1$, it follows from \eqref{eq:pro11}, \eqref{eq:SfaCLimit=0} and \eqref{eq:leftcondition2} that 
\begin{equation}\label{eq:onlydiff1}
\begin{aligned}
0&=-\lim_{n\to +\infty}\frac{4}{t_n^{1+\theta_1}}\bar{x}(\alpha_x(t_n))=\lim_{n\to +\infty}t_n^{1-\theta_1}[\alpha_x(t_n)-\bar{x}(\alpha_x(t_n))] \\
&\geq \lim_{n\to +\infty}t_n^{1-\theta_1}{\bar{\mu}((-\infty, -t_n^{1+\theta_1}))} = A_1,
\end{aligned}
\end{equation}
which contradicts the hypothesis $A_1>0$. 
For $\theta_1=0$, using \eqref{eq:pro11}, \eqref{eq:negative}, \eqref{eq:SfaCLimit=0} and \eqref{eq:leftcondition2} again, we have 
\begin{equation}\label{eq:onlydiff}
\begin{aligned}
0&=-\lim_{n\to +\infty}\frac{4}{t_n}\bar{x}(\alpha_x(t_n))=\lim_{n\to +\infty}t_n[\alpha_x(t_n)-\bar{x}(\alpha_x(t_n))] + 4\bar{u}(-\infty)\\
&\geq \lim_{n\to +\infty}\frac{t_n}{\bar{x}(\alpha_x(t_n))} \bar{x}(\alpha_x(t_n)){\bar{\mu}((-\infty, \bar{x}(\alpha_x(t_n))))} +4\bar{u}(-\infty)=+\infty,
\end{aligned}
\end{equation}
which is also a contradiction.  Hence, the claim \eqref{eq:leftlimclaim1} holds.

Next, we are going to give the exact values of  $a$ and $b$  in \eqref{eq:leftlimclaim1}; as a consequence, they are indeed equal.
For any sequence $t_n$ satisfying $\lim_{n\to+\infty}\frac{4}{t_n^{1+\theta_1}}\bar{x}(\alpha_x(t_n))=-a$, and for any constant $s>1$, 
\[
-\frac{as}{4} t_n^{1+\theta_1}<\bar{x}(\alpha_x(t_n))<-\frac{as^{-1}}{4} t_n^{1+\theta_1}
\] 
for all sufficiently large $n$, and hence, using \eqref{eq:leftcondition2}, we have
\begin{multline*}
{A_1}\left(\frac{as}{4}\right)^{-\frac{1-\theta_1}{1+\theta_1}}= \lim_{n\to +\infty}t_n^{1-\theta_1}{\bar{\mu}\left(\left(-\infty, -\frac{as}{4}t_n^{1+\theta_1}\right)\right)} \leq \liminf_{n\to +\infty}t_n^{1-\theta_1}\bar{\mu}\left(\left(-\infty, \bar{x}(\alpha_x(t_n))\right)\right) \\
\leq \limsup_{n\to +\infty}t_n^{1-\theta_1}\bar{\mu}\left(\left(-\infty, \bar{x}(\alpha_x(t_n))\right]\right)\leq \lim_{n\to +\infty}t_n^{1-\theta_1}{\bar{\mu}\left(\left(-\infty, -\frac{as^{-1}}{4}t_n^{1+\theta_1}\right)\right)} = 	{A_1}\left(\frac{as^{-1}}{4}\right)^{-\frac{1-\theta_1}{1+\theta_1}}.
\end{multline*}
Since this holds for any $s>1$, passing to the limit in the above inequalities as $s\to 1^+$ yields 
\begin{align}\label{eq:important}
\lim_{n\to +\infty}t_n^{1-\theta_1}\bar{\mu}\left(\left(-\infty, \bar{x}(\alpha_x(t_n))\right)\right)=\lim_{n\to +\infty}t_n^{1-\theta_1}\bar{\mu}\left(\left(-\infty, \bar{x}(\alpha_x(t_n))\right]\right) ={A_1}\left(\frac{a}{4}\right)^{-\frac{1-\theta_1}{1+\theta_1}}.
\end{align}

For $0<\theta_1<1$, evaluating \eqref{eq:pro11} at $t=t_n$, and then passing to the limit as $n\to +\infty$, as well as using \eqref{eq:leftlimclaim1}, the boundedness of $\bar{u}$, \eqref{eq:barx1} and \eqref{eq:important}, we have
\begin{equation}\label{eq:equationfora}
a=\lim_{n\to +\infty}t_n^{1-\theta_1}[\alpha_x(t_n)-\bar{x}(\alpha_x(t_n))] = \lim_{n\to \infty}t_n^{1-\theta_1}\bar{\mu}\left(\left(-\infty, \bar{x}(\alpha_x(t_n))\right)\right)
=A_1\left(\frac{a}{4}\right)^{-\frac{1-\theta_1}{1+\theta_1}},
\end{equation}
and hence, $a=2^{1-\theta_1}{A_1}^{\frac{1+\theta_1}{2}}$. Similarly, by passing to a sequence $t_n\to +\infty$ that  achieves the limit superior of $\frac{4}{t^{1+\theta_1}}\bar{x}(\alpha_x(t))$, we can also obtain that $b=2^{1-\theta_1}{A_1}^{\frac{1+\theta_1}{2}}$, and hence, $b=a$ indeed; in other words, we have just shown that
\begin{equation*}\label{eq:lim4xbar/t^1+theta}
	\lim_{t\to +\infty}\frac{4}{t^{1+\theta_1}}\bar{x}(\alpha_x(t))=-2^{1-\theta_1}{A_1}^{\frac{1+\theta_1}{2}},
\end{equation*} 
which is equivalently to the limit \eqref{eq:leftlimit2}.

For $\theta_1=0$, using \eqref{eq:pro11},  \eqref{eq:leftlimclaim1} and \eqref{eq:important} in a similar manner (as in the case $0<\theta_1<1$ above),  as well as using \eqref{eq:negative}, we can also obtain 
\begin{equation*}
	\begin{aligned}
		a-4 \bar{u}(-\infty) &= -\liminf_{n\to+\infty}\frac{4}{t_n}\bar{x}(\alpha_x(t_n)) -4\lim_{n\to +\infty}\bar{u}(\bar{x}(\alpha_x(t_n))\\
		&= \lim_{n\to+\infty}t_n[\alpha_x(t_n)-\bar{x}(\alpha_x(t_n))] =\lim_{n\to+\infty}t_n\bar{\mu}((-\infty,\bar{x}(\alpha(t_n))))=\frac{4{A_1}}{a},
	\end{aligned}
\end{equation*}
and hence, the positive real number $a$ indeed satisfies the quadratic equation
$
a^2 - 4 \bar{u}(-\infty)a -4{A_1}=0.
$
Solving this quadratic equation, we obtain $a = 2\bar{u}(-\infty) + 2\sqrt{\bar{u}^2(-\infty) +{A_1}}$ by choosing the (unique) positive root. The same argument also works for computing $b$, so we can also have $b = 2\bar{u}(-\infty) + 2\sqrt{\bar{u}^2(-\infty) +{A_1}}$, which implies that $a=b$. This shows the limit \eqref{eq:leftlimit2} for the case $\theta_1=0$, and completes the proof of part (ii)(a).

(ii)(b) The main idea of the proof is similar to that of part (ii)(a), since one could also think the results in this case as the limiting case of part (ii)(a) by passing to the limit as $A_1\to 0^+$. In the proof, we will effectively use \eqref{eq:pro11} and various limits that will be derived by using \eqref{eq:leftcondition2} with $A_1=0$. The details will be provided in Appendix~\ref{app:proofiib}.

(iii) The proof of \eqref{eq:compact10} is in the same spirit to that of part (ii)(b) (i.e., for $A_1=\theta_1=0$ under \eqref{eq:leftcondition2}), and one can find the proof in Appendix \ref{app:compact}.

Finally, if we replace $\bar{x}(\alpha_x(t))$ by $\alpha_x(t)$ in \eqref{eq:leftlimit1}-\eqref{eq:leftlimit4}, then due to the fact that
\[
0\leq \alpha_x(t)-\bar{x}(\alpha_x(t))\leq \bar{\mu}(\mathbb{R})<+\infty ,
\]	
one can verify that all the limit indeed exist and \eqref{eq:leftlimit1}-\eqref{eq:leftlimit4} also hold. For replacing $\bar{x}(\alpha_x(t))$ by $\alpha_x(t)$ in \eqref{eq:compact10}, the same conclusion also holds; see Appendix~\ref{app:compact} for further details.
\end{proof}

\begin{remark}\label{rem:sublinear}
If one knows more explicit decay information on the tail of $\bar{\mu}$, then one will be able to provide a more detailed description here. For example, consider the case $\theta_1=A_1=0$ in \eqref{eq:leftcondition2} and $\bar{u}(-\infty)<0$. If there exist constants $A>0$ and $\eta>0$ such that 
\[
\lim_{x\to -\infty}|x|^{1+\eta}{\bar{\mu}((-\infty, -|x|))}=A>0,
\]
then similar to the proof of \eqref{eq:leftlimit2}, one can obtain
\[
\lim_{t \to +\infty} \frac{|\alpha_x(t)|}{t^{\frac{1}{1+\eta}}} =\lim_{t \to +\infty} \frac{|\bar{x}(\alpha_x(t))|}{t^{\frac{1}{1+\eta}}} = \left(\frac{A}{- 4\bar{u}(-\infty)}\right)^{\frac{1}{1+\eta}};
\]
in other words, as functions of $t$, both the $\alpha_x(t)$ and $\bar{x}(\alpha_x(t))$ must have a sublinear growth rate that depends on the $\bar{u}(-\infty)$ and the refined asymptotic behavior of $\bar{\mu}$, namely the parameters $A$ and $\eta$.
\end{remark}

Lemma~\ref{lemma:properties} provides the asymptotic behavior of $\alpha_x(t)$ for any fixed $x\in\mathbb{R}$ under the decay assumptions on the tail of $\bar{\mu}$ at $-\infty$ and the sign condition on $\bar{u}(-\infty)$. We can also apply similar methods to obtain estimates for $\alpha_r(t)$ under the assumptions on the decay of the tail of $\bar{\mu}$ at $+\infty$  and the sign condition on $\bar{u}(+\infty)$ as follows. 
\begin{lemma}\label{lemma:propertiesright}
Let $(u,\mu)$ be a conservative solution to the Hunter-Saxton equation subject to an initial data $(\bar{u},\bar{\mu})\in\mathcal{D}$. Let $0\leq \theta_2<1$ and $A_2\geq0$ be some constants. Then for the function $\alpha_r(t)$ defined by \eqref{eq:alpha01}, we have the following statements:
\begin{enumerate}
\item[(i)] If $A_2>0$, $0<\theta_2<1$ and $\bar{\mu}$ satisfies \eqref{eq:rightcondition1}, then we have
\begin{align}\label{eq:rightlimit1}
0\leq \liminf_{t\to+\infty}\frac{\bar{x}(\alpha_r(t))}{t^{1+\theta_2}}\leq \limsup_{t\to+\infty}\frac{\bar{x}(\alpha_r(t))}{t^{1+\theta_2}}\leq \left(\frac{A_2}{4}\right)^{\frac{1+\theta_2}{2}}.
\end{align}
If $A_2>0$, $\theta_2=0$ and $\bar{\mu}$ satisfies  \eqref{eq:rightcondition1}, then we have
\begin{align}\label{eq:rightlimit11}
	0\leq \liminf_{t\to+\infty}\frac{\bar{x}(\alpha_r(t))}{t^{1+\theta_2}}\leq \limsup_{t\to+\infty}\frac{\bar{x}(\alpha_r(t))}{t^{1+\theta_2}}< +\infty.
\end{align}
\item[(ii)] Assume that $\bar{u}$ satisfies    $\lim_{x\to+\infty}\bar{u}(x)=\bar{u}(+\infty)\in\mathbb{R}$ and $\bar{\mu}$ satisfies 
\begin{align}\label{eq:rightcondition2}
\lim_{x\to +\infty}x^{1-\theta_2}{\bar{\mu}((x^{1+\theta_2},+\infty))} = A_2. \tag{R2}
\end{align} 
Then
\begin{enumerate}
\item[(a)] if $A_2>0$ and $0\leq \theta_2<1$, we have 
\begin{equation}\label{eq:rightlimit2}
\lim_{t\to+\infty}\frac{\bar{x}(\alpha_r(t))}{t^{1+\theta_2}} = \left\{
\begin{aligned}
&\left(\frac{A_2}{4}\right)^{\frac{1+\theta_2}{2}}~\textrm{ for }~0<\theta_2<1;\\
&-\frac{1}{2}\bar{u}(+\infty)+\frac{1}{2}\sqrt{\bar{u}^2(+\infty) +{A_2}}~\textrm{ for }~\theta_2=0;
\end{aligned}
\right.
\end{equation}
\item[(b)] 
if $A_2=0$, $0<\theta_2<1$  and $\bar{\mu}$ satisfies $\bar{\mu}((z,+\infty))>0$ for any $z\in\mathbb{R}$, we have
\begin{align}\label{eq:rightlimit3}
\lim_{t\to+\infty}\frac{\bar{x}(\alpha_r(t))}{t^{1+\theta_2}}=0;
\end{align}
if $A_2=0$, $\theta_2=0$ and $\bar{\mu}$ satisfies $\bar{\mu}((z,+\infty))>0$ for any $z\in\mathbb{R}$, we have
\begin{equation}\label{eq:rightlimit4}
\lim_{t\to+\infty}\frac{\bar{x}(\alpha_r(t))}{t}=\left\{
\begin{aligned}
& -\bar{u}(+\infty)~~\textrm{ if }~~\bar{u}(+\infty)\leq 0;\\
&0~\textrm{ if }~\bar{u}(+\infty)>0.
\end{aligned}\right.
\end{equation}
\end{enumerate}

\item[(iii)] If $\bar{\mu}$ satisfies 
$
r:=\sup\mathrm{supp}\{\bar{\mu}\}<+\infty,
$
then we have
\begin{equation}\label{eq:compact20}
	\lim_{t\to+\infty}\alpha_r(t)=\left\{
	\begin{aligned}
		&r+\bar{\mu}(\mathbb{R})~\textrm{ if }~\bar{u}(+\infty)>0,\\
		&+\infty~\textrm{ if }~\bar{u}(+\infty)<0,
	\end{aligned}
	\right.\quad
		\lim_{t\to+\infty}\bar{x}(\alpha_r(t))=\left\{
	\begin{aligned}
		&r~\textrm{ if }~\bar{u}(+\infty)>0,\\
		&+\infty~\textrm{ if }~\bar{u}(+\infty)<0,
	\end{aligned}
	\right.
\end{equation}
and \eqref{eq:rightlimit4} also holds.

\end{enumerate}
Moreover, all of the above results in (i) and (ii) will also hold if we replace $\bar{x}(\alpha_r(t))$ by $\alpha_r(t)$.
\end{lemma}
\begin{proof}[Outline of Proofs]
The proof of this lemma is essentially the same as that of Lemma~\ref{lemma:properties}. Hence, to illustrate the idea, we will only provide the proof of \eqref{eq:rightlimit2} for some $A_2>0$ and $0<\theta_2<1$ here, and the rest, which can be obtained by using similar arguments as in Lemma \ref{lemma:properties}, will be left to interested readers. 

It follows from \eqref{eq:pro2} that 
\begin{align*}
\alpha_r(t)-\bar{x}(\alpha_r(t))-\bar{\mu}(\mathbb{R})=-\frac{4}{t^2}\bar{x}(\alpha_r(t))-\frac{4}{t}\bar{u}(\bar{x}(\alpha_r(t))).
\end{align*}
Suppose \eqref{eq:rightcondition2} holds for some $A_2>0$ and $0<\theta_2<1$, then by the same arguments above for $\alpha_x(t)$ (as in the proof of part (ii) (a) of Lemma~\ref{lemma:properties}), we conclude that there exist positive constants $a$ and $b$ such that $\liminf_{t\to +\infty}\frac{4}{t^{1+\theta_2}}\bar{x}(\alpha_r(t))=a>0$ and $\limsup_{t\to +\infty}\frac{4}{t^{1+\theta_2}}\bar{x}(\alpha_r(t))=b>0$. Taking the limit inferior and superior respectively on both side of 
\begin{align}\label{eq:pro22}
t^{1-\theta_2}\left[\alpha_r(t)-\bar{x}(\alpha_r(t))-\bar{\mu}(\mathbb{R}) \right]=-\frac{4}{t^{1+\theta_2}}\bar{x}(\alpha_r(t))-\frac{4}{t^{\theta_2}}\bar{u}(\bar{x}(\alpha_r(t))),
\end{align}
we obtain some identities similar to \eqref{eq:equationfora}, namely $a=A_2\left(\frac{a}{4}\right)^{-\frac{1-\theta_2}{1+\theta_2}}$ and $b=A_2\left(\frac{b}{4}\right)^{-\frac{1-\theta_2}{1+\theta_2}}$, which further imply that $a=b= 2^{1-\theta_2}{A_2}^{\frac{1+\theta_2}{2}}$. As a result, $\lim_{t\to +\infty}\frac{\alpha_r(t)}{t^{1+\theta_2}} = \lim_{t\to +\infty}\frac{\bar{x}(\alpha_r(t))}{t^{1+\theta_2}} =  \left(\frac{A_2}{4}\right)^{\frac{1+\theta_2}{2}}$.

\end{proof}

Applying the estimates for $\alpha_x(t)$ and $\alpha_r(t)$ stated in Lemma \ref{lemma:properties} and Lemma \ref{lemma:propertiesright}, we can prove Theorem \ref{thm:mainthm2} under the assumption \eqref{eq:leftcondition1} and \eqref{eq:rightcondition1} as follows.
\begin{proof}[Proof of Theorem \ref{thm:mainthm2}]
In this proof we will only consider the case $t \to +\infty$, because the case $t \to -\infty$ can be obtained immediately  by considering\footnote{It is worth noting that if $(u,\mu)$ is a conservative solution (in the sense of Definition~\ref{def:weak}), then so is $(\tilde{u}, \tilde{\mu})$.} $\left(\tilde{u}(\cdot,t), \tilde{\mu}(t) \right)$ $=\left(-u(\cdot,-t), \mu(-t)\right)$.

\begin{enumerate}[(i)]
	\item 
	To prove ${\eqref{eq:Linfinitycontrolleft}}_1$ and ${\eqref{eq:Linfinitycontrolleft}}_2$, we only need to show
	\begin{equation}\label{eq:estimatesforleftdifference}
		\alpha_l(t)-\bar{x}(\alpha_l(t)) = O\left(\frac{1}{t^{1-\theta_1}}\right),
	\end{equation} 
	because of \eqref{eq:estiatleft} and \eqref{eq:termI1} respectively.
	To prove \eqref{eq:estimatesforleftdifference}, we recall that $\alpha_l(t)=\alpha_x(t)$ for $x=0$, and $\frac{1}{t^{1+\theta_1}}\bar{x}(\alpha_l(t))$  is  uniformly bounded in $t$, which is a direct consequence of part (i) of Lemma \ref{lemma:properties}. Now, it follows from \eqref{eq:pro11}  that \eqref{eq:estimatesforleftdifference} holds.
	
	\item To prove ${\eqref{eq:Linfinitycontrolright}}_1$ and ${\eqref{eq:Linfinitycontrolright}}_2$, it is worth noting that 
	by \eqref{eq:estiatright} and \eqref{eq:termI3} respectively,
	it remains to show that under assumption \eqref{eq:rightcondition1}, 
we have
	\begin{equation}\label{eq:estimates for rightdifference}
		\alpha_r(t)-\bar{x}(\alpha_r(t)) -  \bar{\mu}(\mathbb{R}) = O\left(\frac{1}{t^{1-\theta_2}}\right),
	\end{equation} 
	whose proof follows from the same argument as in the proof of part (i) by using part (i) of Lemma~\ref{lemma:propertiesright}.

	\item Now we prove \eqref{eq:Linfinitycontrolmiddle}. We will only consider the case for $0<\theta_i<1$, $i=1,2$. For the other cases, the proof is highly similar. First of all, it is worth noting that  \eqref{eq:estimatesforleftdifference} and \eqref{eq:estimates for rightdifference} hold with $\theta_1$ and $\theta_2$  replaced by $\theta:=\max\{\theta_1,\theta_2\}$, and hence, \eqref{eq:Linfinitycontrolleft} and \eqref{eq:Linfinitycontrolright} also hold with $\theta_1$ and $\theta_2$ replaced by $\theta$. To prove \eqref{eq:Linfinitycontrolmiddle}, we are left to show $\left\|u(x,t)-\frac{t}{2} v\left(\frac{4x}{t^2}\right)\right\|_{L^\infty((0, \frac{t^2}{4}\bar{\mu}(\mathbb{R})))} = O\left(|t|^{\theta}\right)$ and 
	$\left\|u(x,t)-\frac{t}{2} v\left(\frac{4x}{t^2}\right)\right\|_{L^2((0, \frac{t^2}{4}\bar{\mu}(\mathbb{R})))} = O\left(|t|^{\frac{\theta-1}{2}}\right)$. To do so, we will have to consider the generalized characteristics corresponding to all the $t>1$ and $\alpha\in\mathbb{R}$ so that $0\leq y(\alpha,t)\leq \frac{t^2}{4}\bar{\mu}(\mathbb{R})$. It follows from \eqref{eq:leftlimit1} and \eqref{eq:rightlimit1} that 
	\[
	- \left(\frac{A_1}{4}\right)^{\frac{1+\theta_1}{2}}\leq \liminf_{t\to+\infty}\frac{\bar{x}(\alpha_l(t))}{t^{1+\theta}}\leq 0\leq  \limsup_{t\to+\infty}\frac{\bar{x}(\alpha_r(t))}{t^{1+\theta}}\leq \left(\frac{A_2}{4}\right)^{\frac{1+\theta_2}{2}},
	\]
	provided that $\theta:=\max\{\theta_1,\theta_2\}$. Then for any $t>1$ and $\alpha$ satisfying $0\leq y(\alpha,t)\leq \frac{t^2}{4}\bar{\mu}(\mathbb{R})$, we have $\alpha_l(t)\leq\alpha\leq \alpha_r(t)$, and hence,
	\[
	- \left(\frac{A_1}{4}\right)^{\frac{1+\theta_1}{2}}\leq \liminf_{t\to+\infty}\frac{\bar{x}(\alpha)}{t^{1+\theta}}\leq \limsup_{t\to+\infty}\frac{\bar{x}(\alpha)}{t^{1+\theta}}\leq \left(\frac{A_2}{4}\right)^{\frac{1+\theta_2}{2}}.
	\]
	As a result, one can conclude from \eqref{eq:difference} that 
	\[
	\begin{aligned}
		&\qquad \left\|u(x,t)-\frac{t}{2} v\left(\frac{4x}{t^2}\right)\right\|_{L^\infty((0, \frac{t^2}{4}\bar{\mu}(\mathbb{R})))} \\
		&= \sup\left\{\left|-\frac{2}{t}\bar{x}(\alpha)-\bar{u}(\bar{x}(\alpha))\right| \; :\; 0\leq y(\alpha,t)\leq \frac{t^2}{4}\bar{\mu}(\mathbb{R}) \right\} = O\left(|t|^{\theta}\right),
	\end{aligned}
	\]
	and hence, ${\eqref{eq:Linfinitycontrolmiddle}}_1$ follows directly from the above  estimate, as well as $\eqref{eq:Linfinitycontrolleft}_1$ and $\eqref{eq:Linfinitycontrolright}_1$.
	
	To show ${\eqref{eq:Linfinitycontrolmiddle}}_2$, due to $\eqref{eq:Linfinitycontrolleft}_2$ and $\eqref{eq:Linfinitycontrolright}_2$, it suffices to show $\left\|u(x,t)-\frac{t}{2} v\left(\frac{4x}{t^2}\right)\right\|_{L^2((0, \frac{t^2}{4}\bar{\mu}(\mathbb{R})))} = O\left(|t|^{\frac{\theta-1}{2}}\right)$ as follows. It follows from \eqref{eq: term I2} that 
	\begin{equation*}
		\begin{aligned}
			I_2 
			&\leq [\alpha_r(t)-\bar{x}(\alpha_r(t))- \bar{\mu}(\mathbb{R})] - [\alpha_l(t)-\bar{x}(\alpha_l(t))] 
			+ 2\left[ \bar{\mu}(\mathbb{R})-\frac{2}{t}u\left(\frac{t^2}{4}\bar{\mu}(\mathbb{R}),t\right)\right]   +  \frac{4}{t}u(0,t)\\
			&= [\alpha_r(t)-\bar{x}(\alpha_r(t))- \bar{\mu}(\mathbb{R})] - [\alpha_l(t)-\bar{x}(\alpha_l(t))] 
			\\
			&\quad+ 2\left[ \bar{\mu}(\mathbb{R})-\frac{2}{t}\left(\bar{u}(\bar{x}(\alpha_r(t))  + \frac{t}{2}(\alpha_r(t) -\bar{x}(\alpha_r(t))) \right)\right]  +  \frac{4}{t}\left[\bar{u}(\bar{x}(\alpha_l(t))  + \frac{t}{2}(\alpha_l(t) -\bar{x}(\alpha_l(t))) \right]\\
			&= -[\alpha_r(t)-\bar{x}(\alpha_r(t))- \bar{\mu}(\mathbb{R})] + [\alpha_l(t)-\bar{x}(\alpha_l(t))]  -\frac{4}{t}\bar{u}(\bar{x}(\alpha_r(t))) + \frac{4}{t}\bar{u}(\bar{x}(\alpha_l(t))),
		\end{aligned}
	\end{equation*} 
	where we applied \eqref{eq:measuresolutionu} and \eqref{eq:y(alpha_0,t)=0_and_y(alpha_1,t)=t^2barmu/4} in the second last equality. 
	Then it follows from \eqref{eq:estimatesforleftdifference} and \eqref{eq:estimates for rightdifference} that 
	\begin{equation*}
		\begin{aligned}
			\left\|u_x(x,t)-\partial_x\left[\frac{t}{2} v\left(\frac{4x}{t^2}\right)\right]\right\|_{L^2((0, \frac{t^2}{4}\bar{\mu}(\mathbb{R})))} = \sqrt{I_2} =  O\left({|t|}^{\frac{\theta-1}{2}}\right).
		\end{aligned}
	\end{equation*}

	Finally, the relation \eqref{eq:singularpartdeacyrate2} follows immediately from ${\eqref{eq:Linfinitycontrolmiddle}}_2$ and the fact that
		$\|u_x(\cdot, t)\|_{L^2}^2  + {\mu}_{s}(t)(\mathbb{R})=\mu(t)(\mathbb{R}) \equiv \bar{\mu}(\mathbb{R})=\left\|\partial_x\left[\frac{t}{2} v\left(\frac{4x}{t^2}\right)\right]\right\|_{L^2}^2$ hold for all $t\in \mathbb{R}$. 
	
	\item If $\bar{\mu}$ has compact support, then both \eqref{eq:leftcondition1} and \eqref{eq:rightcondition1} hold for $\theta_1=\theta_2=0$. Hence, all the estimates in \eqref{eq:Linfinitycontrolleft}-\eqref{eq:singularpartdeacyrate2} hold with $\theta_1$=$\theta_2=0$.
\end{enumerate}
\end{proof}
\begin{remark}\label{rem:exampleforcompactsupp} 
In this remark, we provide an explicit example to illustrate that the estimates \eqref{eq:Linfinitycontrolleft}-\eqref{eq:Linfinitycontrolmiddle} in Theorem \ref{thm:mainthm2} are  optimal when $\theta=\theta_1=\theta_2=0$, where $\theta = \max\{\theta_1, \theta_2\}$. Indeed, one can also prove that the estimates \eqref{eq:Linfinitycontrolleft}-\eqref{eq:Linfinitycontrolmiddle} in Theorem \ref{thm:mainthm2}   are optimal for any $0\le \theta = \max\{\theta_1, \theta_2\}<1$; see Remark \ref{rmk:L2R2} below for a sketch of proof.
We consider the initial data $(\bar{u},\bar{\mu}):=(\bar{u},\bar{u}_x^2 \di x)$, where $\bar{u}$ is given by
\begin{equation*}
\bar{u}(x)=\left\{
\begin{aligned}
&0,\quad x\leq0,\\
&-x,\quad 0<x< 1,\\
&-1,\quad x\geq 1.
\end{aligned}
\right.
\end{equation*}
The measure $\bar{u}_x^2 \di x$ is of compact support, so it satisfies the conditions \eqref{eq:leftcondition1} and \eqref{eq:rightcondition1} with $\theta_1=\theta_2=0$. 
The unique conservative solution $(u,\mu)$, where $\mu:= u_x^2 \di x + \mu_s$, subject to this initial data, in the sense of Definition~\ref{def:weak}, is 
\begin{equation*}
	u(x,t)=\left\{
	\begin{aligned}
		&0,\quad x\leq0,\\
		&-\frac{2x}{2-t},\quad 0<x< \left(1-\frac{t}{2}\right)^2,\\
		&-1+\frac{t}{2},\quad x\geq  \left(1-\frac{t}{2}\right)^2,
	\end{aligned}
	\right. \quad \mbox{and}\quad
	\mu_s(t)=\left\{
	\begin{aligned}
		&\delta_0,\quad t = 2,\\
		&0,\quad \mbox{otherwise},
	\end{aligned}
	\right.
\end{equation*}
and hence, for any $t>1$,
\begin{equation*}
u(x,t)-\frac{t}{2} v\left(\frac{4x}{t^2}\right)=\left\{
\begin{split}
&0, \quad x\leq 0,\\
&\frac{4x}{t(t-2)},\quad 0< x< \left(1-\frac{t}{2}\right)^2,\\
&-1+\frac{t}{2}-\frac{2}{t}x,\quad \left(1-\frac{t}{2}\right)^2 \leq x < \frac{t^2}{4},\\
&-1, \quad x \geq \frac{t^2}{4}.
\end{split}
\right.
\end{equation*}
Now, a direct calculation yields that for all $t>1$,
\begin{align*}
\left\|u(x,t)-\frac{t}{2} v\left(\frac{4x}{t^2}\right)\right\|_{L^{\infty}(\mathbb{R})}= 1,
\end{align*}
and for all $t>2$,
\begin{align*}
\left\|u_x(x,t)-\partial_x\left[\frac{t}{2} v\left(\frac{4x}{t^2}\right)\right]\right\|^2_{L^2(\mathbb{R})} = \frac{4}{t^2} + \frac{4}{t^2}(t-1)  = \frac{4}{t}.
\end{align*}
This actually verifies that the estimates in \eqref{eq:Linfinitycontrolmiddle} are indeed optimal for $\theta=0$, as $t\to +\infty$. One can also use a similar direct checking to show that this example also attains the optimal rates in \eqref{eq:Linfinitycontrolmiddle} as $t\to-\infty$.
\end{remark}
\begin{remark}\label{rmk:L2R2}
It is worth noting that in Theorem \ref{thm:mainthm2} we only assume the upper bounds of decay rates of $\bar{\mu}$, namely condition \eqref{eq:leftcondition1} or \eqref{eq:rightcondition1}. If the corresponding limits of the tail of $\bar{\mu}$ actually exist (i.e., $\bar{\mu}$ satisfies \eqref{eq:leftcondition2} or \eqref{eq:rightcondition2}), then one can make use of Lemma \ref{lemma:properties} and Lemma \ref{lemma:propertiesright} to show that the estimates \eqref{eq:Linfinitycontrolleft}-\eqref{eq:Linfinitycontrolmiddle} in Theorem \ref{thm:mainthm2} are actually attainable, namely these estimates are optimal. To illustrate the idea, let us consider the following example: if  $\bar{\mu}$ satisfies \eqref{eq:leftcondition2} for some $0<\theta_1 < 1$ and $A_1>0$, then we actually have
\begin{equation}\label{eq:Linfinitycontrolleft1}
\begin{aligned}
&\lim_{t\to +\infty}\frac{\left\|u(x,t)-\frac{t}{2} v\left(\frac{4x}{t^2}\right)\right\|_{L^\infty((-\infty,0))}}{t^{\theta_1}}= 2 \left(\frac{A_1}{4}\right)^{\frac{1+\theta_1}{2}}>0,\\
&\limsup_{t\to +\infty}\frac{\left\|u_x(x,t)-\partial_x\left[\frac{t}{2} v\left(\frac{4x}{t^2}\right)\right]\right\|_{L^2((-\infty,0))} }{{t}^{\frac{\theta_1-1}{2}}}=  2 \left(\frac{A_1}{4}\right)^{\frac{1+\theta_1}{4}}>0.
\end{aligned}
\end{equation}

Let us sketch the proof of $\eqref{eq:Linfinitycontrolleft1}_1$ as follows. It follows from the monotonicity of $y(\cdot,t)$ and definition~\eqref{eq:alpha01} of $\alpha_l(t)$ that
\[
0\leq \alpha-\bar{x}(\alpha)\leq \alpha_l(t)-\bar{x}(\alpha_l(t))
\]
for all $\alpha$ satisfying $y(\alpha,t)<0$. It follows from \eqref{eq:difference} that for any $t>0$,
\begin{equation*}
\frac{u(x,t)-\frac{t}{2} v\left(\frac{4x}{t^2}\right)}{t^{\theta_1}}=\frac{\bar{u}(\bar{x}(\alpha))}{t^{\theta_1}}+\frac{t^{1-\theta_1}}{2}[\alpha-\bar{x}(\alpha)],\quad\mbox{for all } x=y(\alpha,t)<0.
\end{equation*}
Since $u(x,t)-\frac{t}{2} v\left(\frac{4x}{t^2}\right)$ is continuous, we can obtain the following lower bound by using the continuity and evaluating the above identity at $x=0$, namely when $\alpha=\alpha_l(t)$:
\begin{equation}\label{eq:eq1}
\frac{\left\|u(x,t)-\frac{t}{2} v\left(\frac{4x}{t^2}\right)\right\|_{L^\infty((-\infty,0))}}{t^{\theta_1}}\geq \frac{|u(0,t)-\frac{t}{2} v\left(0\right)|}{t^{\theta_1}} = \left| \frac{\bar{u}(\bar{x}(\alpha_l(t)))}{t^{\theta_1}}+\frac{t^{1-\theta_1}}{2}[\alpha_l(t)-\bar{x}(\alpha_l(t))] \right|.
\end{equation}
On the other hand, it follows from \eqref{eq:estiatleft} that
\begin{equation}\label{eq:eq2}
\frac{\left\|u(x,t)-\frac{t}{2} v\left(\frac{4x}{t^2}\right)\right\|_{L^\infty((-\infty,0))}}{t^{\theta_1}}\leq \frac{\|\bar{u}\|_{L^\infty}}{t^{\theta_1}}+\frac{t^{1-\theta_1}}{2}[\alpha_l(t)-\bar{x}(\alpha_l(t))].
\end{equation}
Evaluating \eqref{eq:pro11} at $x=0$, and then passing to the limit as $t\to+\infty$, we have, after using \eqref{eq:leftlimit2},
\begin{equation}\label{eq:eq3}
\lim_{t\to+\infty} t^{1-\theta_1}[\alpha_l(t)-\bar{x}(\alpha_l(t))]=4\left(\frac{A_1}{4}\right)^{\frac{1+\theta_1}{2}},
\end{equation}
since $\bar{u}$ is bounded and $\alpha_l \equiv \alpha_0$.
Therefore, passing to the limit in \eqref{eq:eq1} and  \eqref{eq:eq2} as $t\to +\infty$, as well as using the limit \eqref{eq:eq3} and the fact that $\bar{u}$ is bounded, we obtain  ${\eqref{eq:Linfinitycontrolleft1}}_1$. 

To prove ${\eqref{eq:Linfinitycontrolleft1}}_2$, let us first recall the following important fact, which was shown in \cite[Theorem~2.1]{gao2021regularity}: there are at most countably many time $t$ so that the singular measure $\mu_{s}(t)\not\equiv 0$, and $\mathcal{L}(({B^L_t})^c)=0$ except at these countably many times. It is worth noting that if $\mathcal{L}(({B^L_t})^c)=0$, then the inequality in \eqref{eq:termI1} will become an equality. Therefore, combining \eqref{eq:termI1} and \eqref{eq:eq3} yields ${\eqref{eq:Linfinitycontrolleft1}}_2$. Let us also remark that for any conservative solution satisfying $\mu_{s}(t)\equiv 0$ for all large $t$\footnote{Using the characterization of regularity structure shown in \cite{gao2021regularity}, one can prove that there are many conservative solutions satisfying this property. For an explicit example, see Example~\ref{ex:nosingular} for instance.}, the limit superior in ${\eqref{eq:Linfinitycontrolleft1}}_2$ is indeed a limit, since \eqref{eq:termI1} is an equality for all large $t$.

Regarding the optimality of other cases of the estimates \eqref{eq:Linfinitycontrolleft}-\eqref{eq:Linfinitycontrolmiddle} in Theorem \ref{thm:mainthm2}, under the limits \eqref{eq:leftcondition2} and \eqref{eq:rightcondition2} with different parameters $\theta_i$ and $A_i$ for $i=1$, $2$, one can also obtain some other different estimates similar to \eqref{eq:Linfinitycontrolleft1}. Since the proofs are very similar, we leave the details to interested readers. 

%
%
\end{remark}
\begin{remark}\label{rem:singularpart}
	We conjecture that the convergence rate \eqref{eq:singularpartdeacyrate2} for the singular part $\mu_{s}(t)$ can be improved in general. For example, when the initial energy measure $\bar{\mu}$ has compact support, then we have, as $t\to\pm\infty$,
	\begin{align}\label{eq:fasterdecayforsingularmeasure}
		{\mu}_{s}(t)(\mathbb{R})=o\left({|t|}^{-2}\right).
	\end{align}
Let us outline the proof of \eqref{eq:fasterdecayforsingularmeasure} as follows. First of all, let us recall from \cite[Theorem 1.1(iii)]{gao2021regularity} that there are at most countably many times $\{t_n\}_{n=1}^\infty$ such that $\mu_{s}(t_n)\not\equiv 0$. 
%
%
Therefore, it suffices to show that for any $\epsilon>0$, there exists a constant $T_\epsilon>0$ such that if $|t_n|\geq T_\epsilon$, then $0\leq t_n^2\mu_s(t_n)(\mathbb{R})\leq\epsilon$. Seeking for a contradiction, we assume that there exist a constant $M>0$ and a subsequence $|t_{n_k}|\to +\infty$ such that $t_n^2\mu_s(t_{n_k})(\mathbb{R})\geq M$. It follows from \cite[Remark 2.3(iii)]{gao2021regularity} that $\mu_s(t)(\mathbb{R}) = \frac{4}{t^2}\mathcal{L}\left(\left\{x:\bar{u}_x(x)=-\frac{2}{t} \right\}\right)$, so $\mathcal{L}\left\{x:\bar{u}_x(x)=-\frac{2}{t_{n_k}} \right\}=\frac{1}{4}t_{n_k}^2\mu_s(t_{n_k})(\mathbb{R})\geq \frac{1}{4} M$. Since all of the sets $\left\{x:\bar{u}_x(x)=-\frac{2}{t_{n_k}} \right\}$'s are mutually disjoint and $\bigcup_{k=1}^{\infty}\left\{x:\bar{u}_x(x)=-\frac{2}{t_{n_k}} \right\}\subset \left\{x:\bar{u}_x(x)\neq 0\right\}$, we have $\mathcal{L}\left(\left\{x:\bar{u}_x(x)\neq 0\right\}\right)=+\infty$, which contradicts the assumption $\bar{\mu}$ has compact support.  This proves \eqref{eq:fasterdecayforsingularmeasure}.
To obtain the optimal decay rates for the singular part $\mu_{s}(t)(\mathbb{R})$, one may have to carefully analyze the set $\left\{x:\bar{u}_x(x)=-\frac{2}{t} \right\}$, and we will not study this issue here.
\end{remark}
Next, we will consider the pointwise convergence for $u$, and more precisely, prove Theorem \ref{thm:mainthm3}. 
As we will see, the pointwise large time behavior of $u(x,t)$ is determined by the behavior of $\bar{u}$ and $\bar{\mu}$  around the negative infinity only. 

\begin{proof}[Proof of Theorem~\ref{thm:mainthm3}]
The pointwise limits stated in Theorem~\ref{thm:mainthm3} can be directly verified by using the following argument: for each fixed $x\in\mathbb{R}$, we aim at first dividing \eqref{eq:measuresolutionu} by an appropriate power of $t$, and then passing to the limit as $t\to +\infty$. In order to compute the limit corresponding to the last term in \eqref{eq:measuresolutionu}, we can make use of the identity \eqref{eq:pro11} and Lemma~\ref{lemma:properties}. Using the continuity and boundedness of $\bar{u}$ as well, we immediately obtain the results.
To illustrate the idea, we will prove part~(i), namely \eqref{eq:uxtlimit1}, below, and part~(ii)-(iv) in Theorem~\ref{thm:mainthm3}, which can be obtained by using Lemma~\ref{lemma:properties} in a similar manner, will be left to interested readers.
 Moreover, for the case that $\bar{\mu}$ satisfies \eqref{eq:leftcondition3}, the proof of \eqref{eq:limtingforfixedx} can be found in Appendix \ref{app:compact}.

Now let us prove part (i) as follows. Let $A_1>0$ and $0<\theta_1<1$. It follows from \eqref{eq:measuresolutionu} that
\begin{align}\label{eq:limitalpha2}
\frac{u(x,t)}{t^{\theta_1}}=\frac{\bar{u}(\bar{x}(\alpha_x(t)))}{t^{\theta_1}}+\frac{t^{1-\theta_1}}{2}[\alpha_x(t)-\bar{x}(\alpha_x(t))].
\end{align}
Passing to the limit in \eqref{eq:limitalpha2} as $t \to +\infty$, one finds that in order to  to obtain \eqref{eq:uxtlimit1}, we only need to show 
\begin{equation}\label{eq:limt^1-theta_1/2(alpha_xt-barxalpha_xt)}
	\lim_{t\to +\infty}\frac{t^{1-\theta_1}}{2}[\alpha_x(t)-\bar{x}(\alpha_x(t))] =2\left(\frac{A_1}{4}\right)^{\frac{1+\theta_1}{2}},
\end{equation}
since $\bar{u}$ is bounded. Passing to the limit in \eqref{eq:pro11} as $t \to +\infty$, and using \eqref{eq:leftlimit2}, we immediately obtain \eqref{eq:limt^1-theta_1/2(alpha_xt-barxalpha_xt)} since both $\bar{u}$ and $x$ are bounded.
\end{proof}

\begin{remark}
Combining \eqref{eq:pro11} and Lemma~\ref{lemma:properties}, we can obtain the decay properties of $t^{1-\theta_1}[\alpha_x(t)-\bar{x}(\alpha_x(t))]$ for all $0\leq \theta_1<1$.
Similarly, combining \eqref{eq:pro22} and Lemma~\ref{lemma:propertiesright}, we can also obtain the decay properties of $t^{1-\theta_2}[\alpha_r(t)-\bar{x}(\alpha_r(t))-\bar{\mu}(\mathbb{R})]$ for all $0\leq \theta_2<1$. In general, for any $\bar{\mu}$ that does not satisfy \eqref{eq:leftcondition3}, all the characteristics will asymptotically travel to the right with positive acceleration (as $t\to\pm\infty$), so the energy will be eventually distributed at the right hand side. Hence, the energy inside the spatial region $(-\infty,0)$ will basically decay. More precisely, it follows from \eqref{eq:L2} (as well as \eqref{eq:Linfinitycontrolmiddle} and \eqref{eq:singularpartdeacyrate2} if applicable) that almost all the energy will be eventually trapped in the region $0\leq x \leq \frac{t^2}{4}\bar{\mu}(\mathbb{R})$, which is the support of the support of $\partial_x\left[\frac{t}{2} v\left(\frac{4x}{t^2}\right)\right]$, as $t\to\pm\infty$.
\end{remark}

\begin{remark}[Pointwise limit of $u(x,t)$ for compactly supported $\bar{\mu}$]\label{rem:Discussion_of_Figure1}

In some cases, such as the case (a) of Figure~\ref{fig:chara}, in which Example~\ref{exm:k} with $k=1$ and $\ell=4$ is depicted, the pointwise limit can directly be visualized. Since $u(x,t)=u(y(\alpha,t),t)=y_t(\alpha,t)$, to find the value of $\lim_{t \to +\infty}u(x,t)$, we can first fix a particular position $x\in\mathbb{R}$, and then travel along the positive time direction, namely moving upward. This vertical line will keep intersecting different (generalized) characteristics $y(\alpha,t)$, and the value of $u(x,t)$ is indeed the slope of the corresponding (generalized) characteristics at $(x,t)$. In the case (a) of figure~\ref{fig:chara}, since $\bar{u}(-\infty)=1$, one can see that the vertical line will eventually intersect the straight lines with slope $1$ for all sufficiently large times, and hence, we read from the figure that $\lim_{t \to +\infty}u(x,t)=1$.

However, some pointwise limits, such as the case (b) of Figure~\ref{fig:chara}, in which Example~\ref{exm:k} with $k=-1$ and $\ell=4$ is depicted, may not be that directly visualized. For any fixed position $x\in\mathbb{R}$, as $t\to+\infty$, it is very hard to read the limit of $u(x,t)$ from the figure, by using the method of reading the slope of $y(\alpha,t)$. This implies that when $\bar{u}(-\infty)<0$, the limit \eqref{eq:limtingforfixedx} is somewhat nontrivial.
This example also shows that when the initial energy measure $\bar{\mu}$ has compact support, the sign of $\bar{u}(-\infty)$ (if it exists) plays an important role in determining the large time pointwise limit of conservative solutions.
\end{remark}
\begin{figure}[H]
\begin{center}
\includegraphics[width=1.0\textwidth]{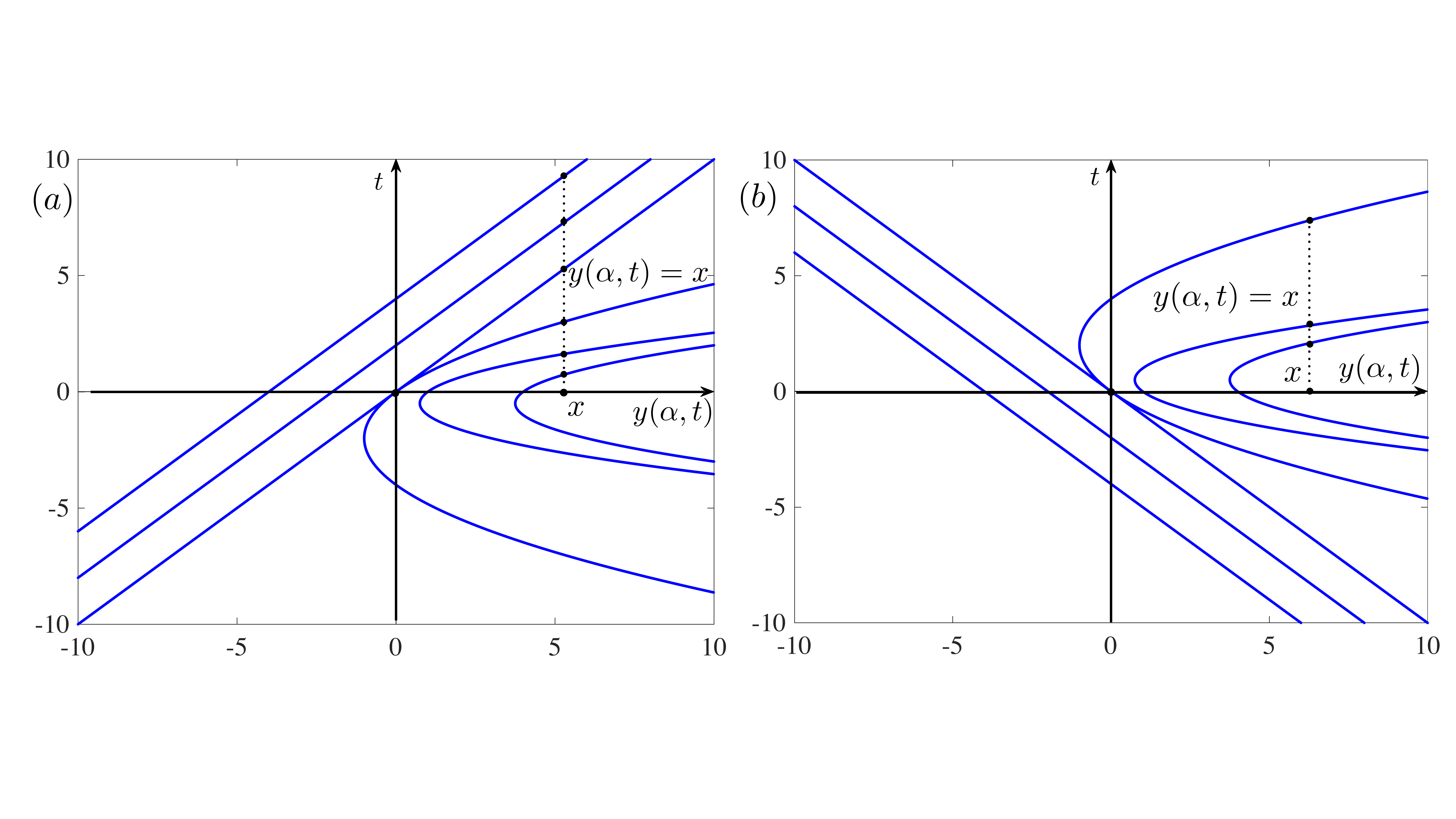}
\end{center}
\caption{Characteristics $y(\alpha,t)$: (a) The six blue lines are the characteristics $y(\alpha,t)$ in Example~\ref{exm:k} with $k=1$ and $\ell=4$, when $\alpha=-4$, $-2$, $0$, $1$, $5$, and $8$.  (b) The six blue lines are the characteristics $y(\alpha,t)$ in Example~\ref{exm:k} with $k=-1$ and $\ell=4$, when $\alpha=-4$, $-2$, $0$, $1$, $5$, and $8$. 
}
\label{fig:chara}
\end{figure}

\section{Asymptotic behavior for another form of Hunter-Saxton equation}\label{sec: another form}
In this section, we consider the generalized framework 
for the Hunter-Saxton equation in the form of \eqref{eq:HSanother}, and study the large time asymptotic behavior of its (energy) conservative solutions. The generalized framework for \eqref{eq:HSanother} is in the following form:
\begin{align}
&u_t+uu_x=\frac{1}{4}\left(\int_{-\infty}^x-\int^{+\infty}_x\right)\di \mu(t),\label{eq:gHS21}\\
&\mu_t+(u\mu)_x=0,\label{eq:gHS22}\\
&\di \mu_{ac}(t)=u_x^2(\cdot,t)\di x.\label{eq:gHS23}
\end{align} 
Similar to system \eqref{eq:gHS1}-\eqref{eq:gHS3}, the conservative solutions to \eqref{eq:gHS21}-\eqref{eq:gHS23} will also be solved subject to initial data in the space $\mathcal{D}$, which is defined in Definition~\ref{def:D}. The definition of conservative solutions to system \eqref{eq:gHS21}-\eqref{eq:gHS23} is almost the same as that of system \eqref{eq:gHS1}-\eqref{eq:gHS3}, namely Definition~\ref{def:weak}, so we will not provide the precise definition of conservative solutions to system \eqref{eq:gHS21}-\eqref{eq:gHS23} here. Instead, we will discuss the corresponding generalized characteristics and their applications as follows. 

For any initial data $(\bar{u},\bar{\mu})\in \mathcal{D}$, we can first define $\bar{x}(\alpha)$ via \eqref{eq:barx1}. Then one can directly verify\footnote{For the formal derivation of Formula~\eqref{eq:generalizedchara1}, see Appendix \ref{app:D} for instance.}
that the explicit formula of the global characteristics for \eqref{eq:gHS21}-\eqref{eq:gHS23} is given by
\begin{align}\label{eq:generalizedchara1} 
y(\alpha,t):=x(\beta(t),t)=\bar{x}(\alpha)+\bar{u}(\bar{x}(\alpha))t+\frac{t^2}{4}\left[\alpha-\bar{x}(\alpha)-\frac{1}{2}\bar{\mu}(\mathbb{R})\right].
\end{align}
Moreover, the global-in-time (energy) conservative solution $(u,\mu)$ can be expressed explicitly as follows: 
\begin{equation}\label{eq:measuresolutionu1}
\begin{aligned}
u(x,t)&:=\bar{u}(\bar{x}(\alpha))+\frac{t}{2}\left[\alpha-\bar{x}(\alpha)-\frac{1}{2}\bar{\mu}(\mathbb{R})\right], ~\textrm{for}~  x =y(\alpha,t),\\
\mu(t)&:=y(\cdot,t)\# (f\di \alpha), 
\end{aligned}
\end{equation}
where $f(\alpha) := 1-\bar{x}'(\alpha)$. 
Following the same argument as in \cite{gao2021regularity}, one can show that $(u,\mu)$ given in \eqref{eq:measuresolutionu1} is the unique global-in-time conservative solution to the generalized framework \eqref{eq:gHS21}-\eqref{eq:gHS23} subject to the initial data $(\bar{u},\bar{\mu})$. Using \eqref{eq:generalizedchara1} and \eqref{eq:measuresolutionu1}, one can also prove the following rescaled limit:
\begin{theorem}\label{thm:limitfunction2}
Let $(u,\mu)$ be the conservative solution to the generalized framework~\eqref{eq:gHS21}-\eqref{eq:gHS23} of the Hunter-Saxton equation \eqref{eq:HSanother} subject to an initial data $(\bar{u},\bar{\mu})\in\mathcal{D}$.
Then
\begin{equation}\label{eq:v2}
v_1(x) := \lim\limits_{t\to\pm\infty} \frac{2}{t}{u\left(\frac{t^2}{4}x,t\right)}= 
\left\{
\begin{aligned}
&-\frac{1}{2}\bar{\mu}(\mathbb{R}), \quad x< -\frac{1}{2}\bar{\mu}(\mathbb{R}),\\
&x, \quad -\frac{1}{2}\bar{\mu}(\mathbb{R})\le x\le \frac{1}{2}\bar{\mu}(\mathbb{R}),\\
&\frac{1}{2}\bar{\mu}(\mathbb{R}), \quad x > \frac{1}{2}\bar{\mu}(\mathbb{R}).
\end{aligned}
\right.
\end{equation}
\end{theorem}
\noindent
The proof of Theorem~\ref{thm:limitfunction2}, which is basically the same as the proof of \eqref{eq:v} given in Section \ref{sec:Asymptotic behavior}, will be left to interested readers.  

Asymptotic expansions similar to Theorem \ref{thm:mainthm1} and error estimates (in the $L^\infty(\mathbb{R})$ and $\dot{H}^1(\mathbb{R})$-norms) similar to Theorem \ref{thm:mainthm2} will also hold, if we use this new $v_1$ to define the leading order term $\frac{t}{2} v_1\left(\frac{4x}{t^2}\right)$. However, as previously mentioned in Section~\ref{sec:intro}, the pointwise limit $\lim\limits_{t\to+\infty} u(x,t) $ is very different; for instance, see Theorem~\ref{mainthm:pointwiselimit2}, which will be shown in this section. The key observation is that in this case for each fixed $x$, the value of $\alpha_x(t)$, which is chosen so that $y(\alpha_x(t),t)=x$ where $y(\alpha,t)$ is in the form of \eqref{eq:generalizedchara1}, 
is bounded in $t$. The main difference is that for the estimates in $L^{\infty}({\mathbb{R}})$ and $\dot{H}^1(\mathbb{R})$-norms, all the characteristics will contribute, no matter whether we consider the Hunter-Saxton equation in the form of \eqref{eq:HS} or \eqref{eq:HSanother}. As a result, the error estimates 
similar to \eqref{eq:Linfinitycontrolleft}, \eqref{eq:Linfinitycontrolright} and \eqref{eq:Linfinitycontrolmiddle} should also hold, except that we have to separate the whole real line into another three different regions: $(-\infty, -\frac{t^2}{8}\bar{\mu}(\mathbb{R}))$, $(-\frac{t^2}{8}\bar{\mu}(\mathbb{R}), \frac{t^2}{8}\bar{\mu}(\mathbb{R}))$ and $(\frac{t^2}{8}\bar{\mu}(\mathbb{R}), +\infty)$. On the other hand, for the large time pointwise behavior, the specific form of the   Hunter-Saxton equation does matter. Heuristically, this is due to the fact that for any fixed $x$, the $\alpha_x(t)$, which is chosen so that $y(\alpha_x(t),t)=x$, will converge to the point that makes the characteristics have zero acceleration as $t \to \pm \infty$, namely the point that makes the coefficient of $t^2$ in $y(\alpha,t)$  become 0. 
Then it follows the form of $y(\alpha,t)$ in \eqref{eq:xbarbeta} and  \eqref{eq:generalizedchara1} respectively, we have $\lim_{t \to \pm\infty} \alpha_x(t)=-\infty$ (at least when the support of  $\bar{\mu}$ is not compact) and $\lim_{t \to \pm\infty} \alpha_x(t) = \alpha^*$ respectively, where $\alpha^* -\bar{x}(\alpha^*)-\frac{1}{2}\bar{\mu}(\mathbb{R})=0$ (for simplicity, we assume such $\alpha^*$ is unique here), so the pointwise behavior really depends on the specific form of the Hunter-Saxton equation.

\begin{proof}[Proof of Theorem~\ref{mainthm:pointwiselimit2}]
 For the special case $\bar{\mu}(\mathbb{R})=0$, one can verify the limit~\eqref{eq:v2} via a direct computation, so we omit it. In the following, without loss of generality, we assume that $\bar{\mu}(\mathbb{R})>0$.
It follows from \eqref{eq:generalizedchara1} and \eqref{eq:measuresolutionu1} that
\begin{align*}
u(x,t)=\bar{u}(\bar{x}(\alpha))+\frac{t}{2}\left[\alpha-\bar{x}(\alpha)-\frac{1}{2}\bar{\mu}(\mathbb{R})\right]~\textrm{for}~ x = \bar{x}(\alpha)+\bar{u}(\bar{x}(\alpha))t+\frac{t^2}{4}\left[\alpha-\bar{x}(\alpha)-\frac{1}{2}\bar{\mu}(\mathbb{R})\right],
\end{align*}
where $\alpha = \alpha_x(t)$, which is defined in \eqref{eq:alphaxt}, depends on both $x$ and $t$. 
We claim that for any fixed $x$, $\alpha_x(t)$ is bounded in $t$. 
Seeking for a contradiction, we assume that there exists a sequence $t_n\to +\infty$ such that $|\alpha_x(t_n)| \to +\infty$ as $n\to +\infty$. Without loss of generality, let us further assume that $\alpha_x(t_n) \to +\infty$ as $n\to +\infty$. First of all, it follows from \eqref{eq:alphaxt} and \eqref{eq:generalizedchara1} that for any $x$ and $t$, we have 
\begin{align}\label{eq:relationnewalpha_x(t)}
	\frac{2}{t}x = \frac{2}{t}\bar{x}(\alpha_x(t)) + 2\bar{u}(\bar{x}(\alpha_x(t))) +\frac{t}{2}\left[\alpha_x(t)-\bar{x}(\alpha_x(t))-\frac{1}{2}\bar{\mu}(\mathbb{R})\right].
\end{align}
Evaluating at $t=t_n$ in \eqref{eq:relationnewalpha_x(t)}, and then passing to the limit as $n\to +\infty$, we have
\begin{align*}
	0 &= \lim_{n\to \infty}\frac{2}{t_n}\bar{x}(\alpha_x(t_n)) + 2\bar{u}(\bar{x}(\alpha_x(t_n))) +\frac{t_n}{2}\left[\alpha_x(t_n)-\bar{x}(\alpha_x(t_n))-\frac{1}{2}\bar{\mu}(\mathbb{R})\right]\\
	&\geq \lim_{n\to +\infty} \frac{t_n}{2}\cdot \frac{1}{4} \bar{\mu}(\mathbb{R}) - 2\|\bar{u}\|_{L^{\infty}} = +\infty,
\end{align*}
which is a contradiction. In the last inequality, we applied the facts that $\bar{x}(\alpha_x(t_n))\to +\infty$ and $\alpha_x(t_n)-\bar{x}(\alpha_x(t_n)) \geq \bar{\mu}((-\infty,\alpha_x(t_n))) \to \bar{\mu}(\mathbb{R})$, as $n\to +\infty$.

We have just shown that for any fixed $x$,  $\alpha_x(t)$ must be bounded in $t$. Using \eqref{eq:relationnewalpha_x(t)} again, we have
\begin{equation}\label{eq:alpha_x(t)-xbaralpha_x(t)}
\begin{aligned}
\alpha_x(t)-\bar{x}(\alpha_x(t))-\frac{1}{2}\bar{\mu}(\mathbb{R}) = \frac{4}{t^2}x - \frac{4}{t^2}\bar{x}(\alpha_x(t)) - \frac{4}{t}\bar{u}(\bar{x}(\alpha_x(t))).
\end{aligned}
\end{equation}
Since $\alpha_x(t)$ is bounded and $0\leq \alpha_x(t)-\bar{x}(\alpha_x(t))\leq\bar{\mu}(\mathbb{R})$,  we also have the boundedness of $\bar{x}(\alpha_x(t))$. Hence, passing to the limit in \eqref{eq:alpha_x(t)-xbaralpha_x(t)} as $t \to +\infty$, we conclude that for any fixed $x$,
\begin{align}\label{eq:lim}
\lim_{t\to +\infty}[\alpha_x(t)-\bar{x}(\alpha_x(t))]=\frac{1}{2}\bar{\mu}(\mathbb{R}),
\end{align}
since $\bar{u}$ is also a bounded function. Substituting \eqref{eq:relationnewalpha_x(t)} into $\eqref{eq:measuresolutionu1}_1$, we obtain, for any $x$ and $t$,
\begin{align}\label{eq:expressionforunewalpha_x(t)}
u(x,t)= - \bar{u}(\bar{x}(\alpha_x(t)))+\frac{2}{t}x- \frac{2}{t}\bar{x}(\alpha_x(t)).
\end{align}
It is worth noting that because of \eqref{eq:expressionforunewalpha_x(t)}, to show \eqref{eq:pointwizsev2}, it suffices to verify the following limit:
\begin{equation}\label{eq:target}
\lim_{t\to+\infty}\bar{u}(\bar{x}(\alpha_x(t)))=\bar{u}(\bar{x}(\alpha^*))
\end{equation}
for any $\alpha^*$ satisfying $\alpha^*-\bar{x}(\alpha^*)=\frac{1}{2}\bar{\mu}(\mathbb{R})$. To this end, we define
\[
\alpha_1^*=\inf\left\{\alpha:~~\alpha-\bar{x}(\alpha)=\frac{1}{2}\bar{\mu}(\mathbb{R})\right\},\quad\mbox{and}\quad \alpha_2^*=\sup\left\{\alpha:~~\alpha-\bar{x}(\alpha)=\frac{1}{2}\bar{\mu}(\mathbb{R})\right\}.
\]
If $\alpha^*_1=\alpha^*_2$, due to the non-decreasing property of $\alpha-\bar{x}(\alpha)$ in $\alpha$ and \eqref{eq:lim} we must have 
$
\lim_{t\to+\infty}\alpha_x(t)=\alpha^*_1=\alpha^*_2,
$
and hence, \eqref{eq:target} holds. For the case $\alpha^*_1<\alpha^*_2$, since $\alpha_1^*-\alpha_2^*=\bar{x}(\alpha^*_1)-\bar{x}(\alpha^*_2)$, we have $\bar{x}(\alpha^*_1)<\bar{x}(\alpha^*_2)$. Evaluating $\alpha$ at $\alpha^*_1$ and $\alpha^*_2$ in \eqref{eq:barx1}, we obtain, for any $i=1$, $2$, 
\[
\bar{\mu}((-\infty,\bar{x}(\alpha_i^*)))\leq\alpha_i^*-\bar{x}(\alpha_i^*)\leq \bar{\mu}((-\infty,\bar{x}(\alpha_i^*)]),
\]
so 
\[
0=[\alpha_2^*-\bar{x}(\alpha_2^*)]-[\alpha_1^*-\bar{x}(\alpha_1^*)]\geq\bar{\mu}((\bar{x}(\alpha_1^*),\bar{x}(\alpha_2^*)))\geq\int_{\bar{x}(\alpha^*_1)}^{\bar{x}(\alpha^*_2)}\bar{u}_x^2(x)\di x \geq 0,
\]
which implies that $\bar{u}_x(x)=0$ for  all $x\in[\bar{x}(\alpha^*_1),\bar{x}(\alpha^*_2)]$.  Using the monotonicity of $\bar{x}$, we have
\begin{align}\label{eq:equiv}
\bar{u}(\bar{x}(\alpha^*))\equiv \bar{u}(\bar{x}(\alpha_1^*)),\quad\mbox{for all } \alpha^*\in[\alpha_1^*,\alpha_2^*].
\end{align}

Now, for any sequence $t_n\to+\infty$, since $\alpha_x(t_n)$ is bounded, there exists a subsequence $t_{n_k} \to+\infty$ as $k\to+\infty$ such that 
\[
\lim_{k\to+\infty}\alpha_x(t_{n_k})=\tilde{\alpha}.
\] 
It follows from \eqref{eq:lim} that $\tilde{\alpha}\in[\alpha_1,\alpha_2]$, so passing to the limit in \eqref{eq:expressionforunewalpha_x(t)} and using \eqref{eq:equiv}, we obtain
\begin{align*}
\lim_{k\to+\infty}\bar{u}(\bar{x}(\alpha_x(t_{n_k})))=\bar{u}(\bar{x}(\tilde{\alpha}))=\bar{u}(\bar{x}(\alpha_1^*)).
\end{align*}
Indeed, we have just proved that for any sequence of $\{\bar{u}(\bar{x}(\alpha_x(t_n)))\}_{n=1}^\infty$, there exists a subsequence $\{\bar{u}(\bar{x}(\alpha_x(t_{n_k})))\}_{k=1}^\infty$ converging to the same limit $\bar{u}(\bar{x}(\alpha_1^*)).$ This is sufficient to show \eqref{eq:target}. Actually, if \eqref{eq:target} does not hold, there exists a subsequence  $\{\bar{u}(\bar{x}(\alpha_x(t_n)))\}_{n=1}^\infty$ converging to some value $C\neq \bar{u}(\bar{x}(\alpha_1^*))$. For this sequence, we cannot find a subsequence converging to $\bar{u}(\bar{x}(\alpha_1^*))$, which contradicts the above conclusion. This completes the proof.

\end{proof}

\appendix
\section{Explicit examples}\label{app}
In this appendix we will provide two explicit examples that illustrate two important facts.
The first explicit example, namely Example~\ref{exm:k} below, illustrates the fact that the rescaled limit $\lim\limits_{t\to\pm\infty} \frac{2}{t}{u\left(\frac{t^2}{4}x,t\right)}$, and hence, the leading order term in the large time expansion of the conservative solution with initial data $(\bar{u},\bar{\mu})$ is completely determined by the total energy $\bar{\mu}(\mathbb{R})$ only and is independent of the profile of $\bar{u}$. This is consistent with the results shown in Theorem \ref{thm:mainthm1}. In the second example, namely Example~\ref{ex:nosingular} below, we explicitly construct an initial data $(\bar{u},\bar{\mu})$ that fulfills the conditions \eqref{eq:leftcondition2} and \eqref{eq:rightcondition2}. 
\begin{example}\label{exm:k}
Consider the generalized framework \eqref{eq:gHS1}-\eqref{eq:gHS3} subject to an initial data $(\bar{u},\bar{\mu}):=(k, \ell\delta_{x_0})$, where $k\in\mathbb{R}$, $\ell:=\bar{\mu}(\mathbb{R})>0$, and $x_0\in\mathbb{R}$ are three arbitrary constants. Then a direction computation yields
\begin{equation*}
\bar{x}(\alpha)=\left\{
\begin{split}
&\alpha,\quad \alpha<x_0,\\
&x_0,\quad x_0\leq\alpha\leq x_0+\ell,\\
&\alpha-\ell,\quad\alpha>x_0+\ell,
\end{split}
\right.\quad \quad
f(\alpha):=1-\bar{x}'(\alpha)=\left\{
\begin{split}
&0,\quad \alpha<x_0,\\
&1,\quad x_0\leq\alpha\leq x_0+\ell,\\
&0,\quad \alpha>x_0+\ell.
\end{split}
\right.
\end{equation*}
Using \eqref{eq:xbarbeta}, we have
\begin{equation*}
y(\alpha,t)  = \left\{
\begin{aligned}
&\alpha+kt, \quad \alpha<x_0,\\
&x_0+kt+\frac{t^2}{4}(\alpha-x_0), \quad x_0\le \alpha \le x_0+\ell,\\
&\alpha-\ell+kt+\frac{t^2}{4}\ell, \quad \alpha>x_0+\ell.
\end{aligned}
\right.
\end{equation*}
It follows from the definition \eqref{eq:defalpha} of $\alpha(x,t)$ that for any $(x,t)$,
\begin{equation*}
\alpha(x,t)  = \left\{
\begin{aligned}
&\frac{t^2}{4} x-kt, \quad x<\frac{4}{t^2}(x_0+kt),\\
&x+x_0-\frac{4}{t}k-\frac{4}{t^2}x_0, \quad \frac{4}{t^2}(x_0+kt)\le x \le \frac{4}{t^2}(x_0+kt)+\ell,\\
&\frac{t^2}{4}(x-\ell)-kt+\ell, \quad x> \frac{4}{t^2}(x_0+kt)+\ell,
\end{aligned}
\right.
\end{equation*}
and hence,
\begin{equation*}
\lim\limits_{t\to\pm \infty} \frac{4 \alpha(x,t)  }{t^2}= \left\{
\begin{aligned}
&x, \quad x<0,\\
&0, \quad 0\le x \le \ell,\\
&x-\ell, \quad x>\ell.
\end{aligned}
\right.
\end{equation*}
Now, using \eqref{eq:asy}, we obtain the following large time asymptotic behavior:
\begin{equation}\label{eq:example}
\lim\limits_{t\to \pm\infty} \frac{2}{t}{u\left(\frac{t^2}{4}x,t\right)}= 
\left\{
\begin{aligned}
0&, \quad x< 0,\\
x&, \quad 0\le x\le \ell,\\
\ell&, \quad x > \ell.
\end{aligned}
\right.
\end{equation}
In conclusion, for the special initial data $(\bar{u},\bar{\mu}):=(k, \ell\delta_{x_0})$, the asymptotic behavior \eqref{eq:example} of the corresponding conservative solution only depends the initial total energy measure $l:=\bar{\mu}(\mathbb{R})$, but not on $k$ and $x_0$. This is consistent with the asymptotic limit \eqref{eq:v} for $v$.
\end{example}

\begin{example}\label{ex:nosingular}
	 We construct an initial data that satisfies \eqref{eq:leftcondition2} and \eqref{eq:rightcondition2}, as well as other properties as follows. Let $a\in \left(\frac{1}{2},1\right]$ be a constant, and define
	\begin{equation*}
		\bar{u}(x) := \int_{0}^{x} \frac{\sin y}{|y|^a} \di y,\quad \mbox{for any $x\in \mathbb{R}$}.
	\end{equation*}
	Then one can directly verify that $\bar{u}$ is a bounded and continuously differentiable function; furthermore, both the limits $\bar{u}(-\infty):=\lim_{x\to-\infty}\bar{u}(x)$ and $\bar{u}(+\infty):=\lim_{x\to+\infty}\bar{u}(x)$ exist, and equal to the conditionally convergent integrals $\int_{0}^{+\infty}\frac{\sin y}{|y|^a} \di y$ and $-\int_{-\infty}^{0}\frac{\sin y}{|y|^a} \di y$ respectively.
	A direct differentiation yields $\bar{u}_x(x)=\frac{\sin x}{|x|^a}$ for all $x\in \mathbb{R}$, and we can define $\di \bar{\mu}:= \bar{u}_x^2(x) \di x = \frac{\sin^2 x }{|x|^{2a}} \di x$. Since $2a>1$, we have $\int_{-\infty}^{+\infty} \frac{\sin^2 x }{|x|^{2a}} \di x <+\infty$, namely $\bar{u}_x \in L^2(\mathbb{R})$. Therefore, $(\bar{u},\bar{\mu})\in\mathcal{D}$, which is defined in Definition~\ref{def:D}.

	Let $\theta_1:=\theta_2:=\frac{1}{a}-1\in[0,1)$. 
	Using the fact that $a>\frac{1}{2}$, one can directly verify that there exists a constant $C(a)>0$ such that
	\begin{equation*}
		\lim_{x\to +\infty}x^{1-\theta_2}{\bar{\mu}((x^{1+\theta_2},+\infty))}
		=\lim_{x\to +\infty}x^{1-\theta_2} \int_{x^{1+\theta_2}}^{+\infty}\frac{\sin^2 y}{y^{2a}} \di y =: C(a).
	\end{equation*}
	Similarly, we also have $
	\lim_{x\to -\infty}|x|^{1-\theta_1}{\bar{\mu}((-\infty, -|x|^{1+\theta_1}))} =C(a)$. In addition, one can use different values of $a$ around $-\infty$ and $+\infty$ to construct a function $\bar{u}$, so that \eqref{eq:leftcondition2} and \eqref{eq:rightcondition2} hold with different $\theta_1$ and $\theta_2$.

	Finally, let us also discuss the singular part $\mu_{s}(t)$ of the energy measure generated by this initial data. It follows from \cite[Theorem 2.1]{gao2021regularity} that the set $\{x:\bar{u}_x (x) =-\frac{2}{t}\}$ will generate singular part $\mu_{s}(t)$, and from \cite[Remark 2.3(iii)]{gao2021regularity} that $\mu_s(t)(\mathbb{R}) = \frac{4}{t^2}\mathcal{L}\left(\left\{x:\bar{u}_x(x)=-\frac{2}{t} \right\}\right)$, where $\mathcal{L}$ is the Lebesgue measure. However, the set $\left\{x:{\sin x} =-\frac{2}{t}|x|^a\right\}$ only has finite elements, and hence, is of measure $0$ for all $t\neq0$. As a result, the singular part (of the energy measure) $\mu_{s}(t)\equiv0$ for all $t\in \mathbb{R}$. 
\end{example}

\section{Proof of $\mathrm{(ii)(b)}$ of Lemma \ref{lemma:properties}}\label{app:proofiib}

First, it is worth noting that the hypotheses stated in part (ii)(b) of Lemma~\ref{lemma:properties} allows us to apply Claim~\ref{claim:alpha_x->infty}, so \eqref{eq:negative} also holds in this case.
\begin{proof}[Proof of \eqref{eq:leftlimit3}.]
Seeking for a contradiction, we assume that the limit \eqref{eq:leftlimit3} does not hold. 
 Then it follows from \eqref{eq:negative} that the quotient $\frac{\bar{x}(\alpha_x(t))}{t^{1+\theta_1}}$ is always negative for all sufficiently large $t$, so $\limsup_{t\to +\infty}\frac{\bar{x}(\alpha_x(t))}{t^{1+\theta_1}}\leq 0$. Hence, the invalidity of limit \eqref{eq:leftlimit3} implies that $\liminf_{t\to +\infty}\frac{\bar{x}(\alpha_x(t))}{t^{1+\theta_1}}<0$, so there exist a constant $c>0$ and a sequence $t_n\to+\infty$ such that for all sufficiently large $n$,
\[
	\bar{x}(\alpha_x(t_n))<-ct_n^{1+\theta_1}.
\]

Hence, using \eqref{eq:pro11}, the second inequality in \eqref{eq:barx1} and hypothesis \eqref{eq:leftcondition2}, we obtain
\begin{align*}
0 &< c \leq -\liminf_{n\to +\infty}\frac{4}{t_n^{1+\theta_1}}\bar{x}(\alpha_x(t_n)) \leq \limsup_{n\to +\infty}t_n^{1-\theta_1}[\alpha_x(t_n)-\bar{x}(\alpha_x(t_n))]\\
&\leq \lim_{n\to +\infty}t_n^{1-\theta_1}\bar{\mu}((-\infty, -ct_n^{1+\theta_1})) = \frac{1}{c^{1-\theta_1}}\lim_{n\to +\infty}(ct_n)^{1-\theta_1}\bar{\mu}((-\infty, -ct_n^{1+\theta_1})) = 0,
\end{align*}
which is a contradiction. Therefore, \eqref{eq:leftlimit3} holds.
\end{proof}

\begin{proof}[Proof of \eqref{eq:leftlimit4}.]
To prove  \eqref{eq:leftlimit4}, we of course  consider  two cases: 1. $\bar{u}(-\infty)\leq 0$; and 2. $\bar{u}(-\infty)>0$.

\textbf{Case 1:} {$\bar{u}(-\infty)\leq 0$. The proof of this case is highly similar to that of \eqref{eq:leftlimit3}.
Seeking for a contradiction, we assume that \eqref{eq:leftlimit4} does not hold. Then similar to the beginning of the proof of \eqref{eq:leftlimit3}, it follows from \eqref{eq:negative} that there exist a constant $c>0$ and a sequence $t_n\to+\infty$ such that  $\bar{x}(\alpha_x(t_n))<-ct_n$ for all sufficiently large $n$. Using \eqref{eq:pro11}, \eqref{eq:barx1} and hypothesis \eqref{eq:leftcondition2} with $A_1=\theta_1=0$, we have 
\begin{equation*}
	\begin{aligned}
		0 &< c \leq -\liminf_{n \to +\infty}\frac{4}{t_n}\bar{x}(\alpha_x(t_n)) - 4\bar{u}(-\infty)= \limsup_{n \to +\infty} t_n\left[\alpha_x(t_n)-\bar{x}(\alpha_x(t_n)) \right]\\
		&\leq \limsup_{n \to +\infty} t_n\bar{\mu}((-\infty,\bar{x}(\alpha_x(t_n))])\leq \limsup_{n \to +\infty} t_n\bar{\mu}((-\infty,-ct_n))= 0,
	\end{aligned}
\end{equation*}
which is a contradiction. Hence, \eqref{eq:leftlimit4} actually holds.}

\textbf{Case 2:} $\bar{u}(-\infty)>0$. 
First of all, we claim
\begin{claim}\label{claim:B1}
There exist positive constants $a_1$ and $b_1$ such that
\begin{align}\label{eq:A10theta1equal0}
\liminf_{t\to +\infty}\frac{4}{t}\bar{x}(\alpha_x(t))=-a_1<0~\textrm{ and }~\limsup_{t\to +\infty}\frac{4}{t}\bar{x}(\alpha_x(t))=-b_1<0.
\end{align}
\end{claim}
The proof of \eqref{eq:A10theta1equal0} for $\theta_1=0$ is highly similar to that of \eqref{eq:leftlimclaim1}. For completeness, we also provide it here as follows.
\begin{proof}[Proof of Claim~\ref{claim:B1}]
{
First of all, it follows from \eqref{eq:negative} that the quotient $\frac{4\bar{x}(\alpha_x(t))}{t}$ is always negative for all sufficiently large $t$, so $\liminf_{t\to +\infty}\frac{4\bar{x}(\alpha_x(t))}{t}\leq \limsup_{t\to +\infty}\frac{4\bar{x}(\alpha_x(t))}{t}\leq 0$. Therefore, to verify \eqref{eq:A10theta1equal0}, we will show that $\liminf_{t\to +\infty}\frac{4\bar{x}(\alpha_x(t))}{t^{1+\theta_1}}=-\infty$ and $\limsup_{t\to +\infty}\frac{4\bar{x}(\alpha_x(t))}{t}= 0$ are impossible.
}

{
Seeking for a contradiction, we first assume that $\liminf_{t\to +\infty}\frac{4\bar{x}(\alpha_x(t))}{t^{1+\theta_1}}=-\infty$. In particular, there exists a sequence $t_n\to+\infty$ such that  $\bar{x}(\alpha_x(t_n))<-t_n$ for all sufficiently large $n$. Hence, using \eqref{eq:barx1}, we have
\begin{equation}\label{eq:alpha_x-xbaralpha_xleqmubar}
	\alpha_x(t_n)-\bar{x}(\alpha_x(t_n))\leq \bar{\mu}((-\infty,\bar{x}(\alpha_x(t_n))])\leq \bar{\mu}((-\infty, -t_n)).
\end{equation}
}
Using \eqref{eq:pro11},  \eqref{eq:alpha_x-xbaralpha_xleqmubar} and hypothesis \eqref{eq:leftcondition2} with $A_1=\theta_1=0$, we have
\begin{equation*}
	\begin{aligned}
		+\infty &= -\liminf_{n\to +\infty}\frac{4}{t_n}\bar{x}(\alpha_x(t_n)) \leq \limsup_{n\to +\infty}t_n[\alpha_x(t_n)-\bar{x}(\alpha_x(t_n))] + 4 \|\bar{u}\|_{L^\infty}\\
		&\leq \lim_{n\to +\infty}t_n\bar{\mu}((-\infty, -t_n))+4\|\bar{u}\|_{L^\infty}= 4\|\bar{u}\|_{L^\infty},
	\end{aligned}
\end{equation*}
which is a contradiction. Therefore, $\liminf_{t\to +\infty}\frac{4\bar{x}(\alpha_x(t))}{t^{1+\theta_1}}>-\infty$.

{Seeking for a contradiction, we now assume that $\limsup_{t\to +\infty}\frac{4\bar{x}(\alpha_x(t))}{t}= 0$. Then there exists a sequence $t_n\to+\infty$ such that $\lim_{n\to +\infty}\frac{4}{t_n}\bar{x}(\alpha_x(t_n))=0$. In particular, $\bar{x}(\alpha_x(t_n))>-t_n$ for all sufficiently large $n$, so using \eqref{eq:barx1}, we have
\begin{equation}\label{eq:alpha_x-xbaralpha_xgeqmubar}
\alpha_x(t_n)-\bar{x}(\alpha_x(t_n))\geq\bar{\mu}((-\infty, \bar{x}(\alpha_x(t_n))))\geq\bar{\mu}((-\infty, -t_n)).
\end{equation}
Then it follows from  \eqref{eq:pro11}, \eqref{eq:negative}, \eqref{eq:alpha_x-xbaralpha_xgeqmubar} and hypothesis \eqref{eq:leftcondition2} with $A_1=\theta_1=0$ again that
\begin{equation*}
	\begin{aligned}
		0&=-\lim_{n\to +\infty}\frac{4}{t_n}\bar{x}(\alpha_x(t_n))=\lim_{n\to +\infty}t_n[\alpha_x(t_n)-\bar{x}(\alpha_x(t_n))] + 4\bar{u}(-\infty)\\
		&\geq \lim_{n\to +\infty}t_n\bar{\mu}((-\infty, -t_n)) +4\bar{u}(-\infty)=4\bar{u}(-\infty) >0,
	\end{aligned}
\end{equation*}
} 
which is also a contradiction. Hence, $\limsup_{t\to +\infty}\frac{4\bar{x}(\alpha_x(t))}{t}< 0$. This completes the proof of \eqref{eq:A10theta1equal0}. 
\end{proof}
In order to show the limit \eqref{eq:leftlimit4} exists, it suffices to show $a_1=b_1$. More precisely, to verify \eqref{eq:leftlimit4}, we will apply \eqref{eq:A10theta1equal0} to prove that $a_1=4 \bar{u}(-\infty)=b_1$ as follows. 
Consider any sequence $t_n\to+\infty$ satisfying 
\[
\lim_{n\to+\infty}\frac{4}{t_n}\bar{x}(\alpha_x(t_n))=-a_1.
\]
For any sufficiently large $n$, we have
\[
-\frac{a_1}{2}t_n<\bar{x}(\alpha_x(t_n))<-\frac{a_1}{8}t_n,
\]
and hence,
\[
\begin{aligned}
0&= \liminf_{n\to +\infty}t_n{\bar{\mu}\left(\left(-\infty, -\frac{a_1}{2}t_n\right)\right)} \leq \liminf_{n\to +\infty}t_n\bar{\mu}\left(\left(-\infty, \bar{x}(\alpha_x(t_n))\right)\right) \\
&\leq \limsup_{n\to +\infty}t_n\bar{\mu}\left((-\infty, \bar{x}(\alpha_x(t_n))]\right)\leq \limsup_{n\to +\infty}t_n{\bar{\mu}\left(\left(-\infty, -\frac{a_1}{8}t_n\right)\right)} = 0,
\end{aligned}
\]
where we applied \eqref{eq:leftcondition2} to both the leftmost and rightmost limits.
This implies
\begin{align*}
\lim_{n\to+\infty}t_n\bar{\mu}((-\infty,\bar{x}(\alpha_x(t_n))))=\lim_{n\to+\infty}t_n\bar{\mu}((-\infty,\bar{x}(\alpha_x(t_n))])=0.
\end{align*}
Therefore using \eqref{eq:barx1}, we have
\[
	\lim_{n\to+\infty}t_n[\alpha_x(t_n)-\bar{x}(\alpha_x(t_n))]=0,
\]
so evaluating \eqref{eq:pro11} at $t=t_n$, and then passing to the limit as $n\to +\infty$, as well as using \eqref{eq:negative}, we have 
\begin{align*}
a_1-4 \bar{u}(-\infty) = -\lim_{n\to+\infty}\frac{4}{t_n}\bar{x}(\alpha_x(t_n)) -4\lim_{n\to +\infty}\bar{u}(\bar{x}(\alpha_x(t_n))
=\lim_{n\to+\infty}t_n[\alpha_x(t_n)-\bar{x}(\alpha_x(t_n))] =0.
\end{align*}
Thus, $a_1=4 \bar{u}(-\infty)$. Similarly, one can also show that $b_1 = 4\bar{u}(-\infty)$ in the same manner, by considering any $t_n\to+\infty$ satisfying 
$
\lim_{n\to+\infty}\frac{4}{t_n}\bar{x}(\alpha_x(t_n))=-b_1$.
As a result, $a_1=4 \bar{u}(-\infty)=b_1$, and hence, the limit $\lim_{t\to+\infty}\frac{\bar{x}(\alpha_x(t))}{t}$ exists and \eqref{eq:leftlimit4} holds when $\bar{u}(-\infty) >0$.

\end{proof}

\section{Properties of conservative solutions with initial energy measure $\bar{\mu}$ supported on the right half-line $[\ell,+\infty)$}\label{app:compact} 

In this section we will study the properties of conservative solutions subject to the initial energy measure $\bar{\mu}$ supported on the half-line $[\ell,+\infty)$, namely satisfying the condition \eqref{eq:leftcondition3}. As a consequence, we will also provide proofs of part $\mathrm{(iii)}$ of Lemma~\ref{lemma:properties}, and \eqref{eq:limtingforfixedx} for $\bar{\mu}$ satisfying \eqref{eq:leftcondition3}

Throughout this appendix, we will consider the unique conservative solution $(u,\mu)$ subject to the initial data $(\bar{u},\bar{\mu})\in\mathcal{D}$ that satisfies the condition \eqref{eq:leftcondition3}. Without loss of generality\footnote{For the trivial case $\bar{\mu}\equiv 0$, the initial data $\bar{u}$ must be a constant. The unique conservative solution to this special case is a constant solution, so all the properties can be directly verified by using the explicit solution.}, we assume that $\bar{\mu}\not\equiv 0$. First of all, we have the following
\begin{proposition}\label{prop:Basic_Properties_under_leftcondition3}
	Let $\bar{\mu}\not\equiv 0$ and satisfy \eqref{eq:leftcondition3}. Then we have the following statements:
	\begin{enumerate}[(i)]
		\item The given initial data $\bar{u}$ and the function $\bar{x}$ defined by \eqref{eq:barx1} satisfy
		\begin{align}\label{eq:left}
			\bar{u}(x)\equiv \bar{u}(-\infty)~\textrm{ for }~x\in (-\infty, \ell],
		\end{align}
		\begin{equation}\label{eq:leftright}
			\bar{x}(\alpha)=\alpha,\quad \mbox{for all }\alpha <\ell.
		\end{equation}
		\item The function $\alpha_x(t)$ defined by definition \eqref{eq:alphaxt} satisfies 
		\begin{align}\label{eq:leqell}
			\limsup_{t\to+\infty}\bar{x}(\alpha_x(t))\leq\ell,
		\end{align}
		\begin{align}\label{eq:baru}
			\lim_{t\to+\infty}\bar{u}(\bar{x}(\alpha_x(t)))=\bar{u}(-\infty).
		\end{align}
		\item The functions $\bar{x}$ and $\alpha_x(t)$ also satisfy the following inequalities:
		\begin{align}\label{eq:limsupxbaralpha_xt}
			\limsup_{t\to+\infty} \frac{\bar{x}(\alpha_x(t))}{t} \leq 0,
		\end{align}
		\begin{align}\label{eq:liminftalpha_xt-xbaralpha_xt}
			\liminf_{t\to+\infty} t [\alpha_x(t)-\bar{x}(\alpha_x(t))]\geq 0.
		\end{align}
	\end{enumerate}
\end{proposition}
Before proving Proposition~\ref{prop:Basic_Properties_under_leftcondition3}, let us recall that for any fixed number $x\in\mathbb{R}$, it follows from \eqref{eq:xbarbeta} and the definition \eqref{eq:alphaxt} of $\alpha_x(t)$ that
\begin{align}\label{eq:xbarbetax}
	x=y(\alpha_x(t),t)=\bar{x}(\alpha_x(t))+\bar{u}(\bar{x}(\alpha_x(t)))t+\frac{t^2}{4}(\alpha_x(t)-\bar{x}(\alpha_x(t))),
\end{align} 
which is equivalent to (as long as $t\neq 0$)
\begin{equation}\label{eq:xbarbetax/t}
	\frac{x}{t}=\frac{\bar{x}(\alpha_x(t))}{t}+\bar{u}(\bar{x}(\alpha_x(t)))+\frac{t}{4}(\alpha_x(t)-\bar{x}(\alpha_x(t))).
\end{equation}
Furthermore, it follows from \eqref{eq:measuresolutionu} that
\begin{align}\label{eq:Formula_for_u_in_terms_of_alpha_xt}
	u(x,t)=\bar{u}(\bar{x}(\alpha_x(t)))+\frac{t}{2}(\alpha_x(t)-\bar{x}(\alpha_x(t))).
\end{align}
All of \eqref{eq:xbarbetax}, \eqref{eq:xbarbetax/t} and \eqref{eq:Formula_for_u_in_terms_of_alpha_xt} will be useful in the analysis below.
\begin{proof}[Proof of Proposition~\ref{prop:Basic_Properties_under_leftcondition3}]
	\begin{enumerate}[(i)]
		\item Since $\bar{\mu}\not\equiv 0$ and satisfies \eqref{eq:leftcondition3},  $\ell:=\inf\mathrm{supp}\{\bar{\mu}\}\in\mathbb{R}$. It follows from $(\bar{u},\bar{\mu})\in\mathcal{D}$ that $\bar{u}_x^2\di x=\di \mu_{ac}$, so $\bar{u}_x(x)=0$ for a.e. $x\in (-\infty, \ell)$, and hence, we have \eqref{eq:left}. Furthermore, \eqref{eq:leftright} follows immediately from definition~\eqref{eq:barx1} and the fact that $\bar{\mu}(-\infty,\ell)=0$.
		\item First of all, let us show \eqref{eq:leqell} by using proof by contradiction. Seeking for a contradiction, we assume that there exists a sequence $t_n\to+\infty$ such that $\lim_{n\to+\infty}\bar{x}(\alpha_x(t_n))>\ell$. Using \eqref{eq:barx1} and the definition $\ell:=\inf\mathrm{supp}\{\bar{\mu}\}$, we have
		\[
		\lim_{n\to+\infty}[\alpha_x(t_n)-\bar{x}(\alpha_x(t_n))]\geq \lim_{n\to+\infty}\bar{\mu}((-\infty,\bar{x}(\alpha_x(t_n)))>\bar{\mu}((-\infty,\ell))=0,
		\] 
		and hence, evaluating \eqref{eq:xbarbetax} at $t=t_n$ and passing to the limit as $n\to +\infty$, we know that the limit on the right hand side of \eqref{eq:xbarbetax} is $+\infty$ since $\lim_{n\to+\infty}\bar{x}(\alpha_x(t_n))>\ell$ and $\bar{u}$ is bounded. This contradicts the finiteness of $x$ on the left hand side of \eqref{eq:xbarbetax}. Therefore, \eqref{eq:leqell} holds. Finally, using the continuity of $\bar{u}$, \eqref{eq:left} and \eqref{eq:leqell}, we immediately have \eqref{eq:baru}.
		\item The inequality~\eqref{eq:limsupxbaralpha_xt} follows directly from \eqref{eq:leqell} and the fact that $\ell\in\mathbb{R}$. On the other hand, \eqref{eq:liminftalpha_xt-xbaralpha_xt} follows immediately from the first inequality in \eqref{eq:barx1}.
	\end{enumerate}
\end{proof}
In order to effectively analyze the large time asymptotic behavior in different cases, let us introduce the following useful lemma, which basically means that both inequalities \eqref{eq:limsupxbaralpha_xt} and \eqref{eq:liminftalpha_xt-xbaralpha_xt} cannot be simultaneously strict.
\begin{lemma}\label{lemma:dilemma}
	Let $\bar{\mu}\not\equiv 0$ and satisfy \eqref{eq:leftcondition3}. Then the following two statements hold:
	\begin{enumerate}[(i)]
		\item If
		$
			\limsup_{t\to+\infty} \frac{\bar{x}(\alpha_x(t))}{t} < 0,
		$
		then
		\begin{equation}\label{eq:limtalpha_xt-xbaralpha_xt=0}
			\lim_{t\to+\infty} t[\alpha_x(t)-\bar{x}(\alpha_x(t))]= 0.
		\end{equation}
		\item If
		$
			\liminf_{t\to+\infty} t[\alpha_x(t)-\bar{x}(\alpha_x(t))]> 0,
		$
		then
		\begin{equation}\label{eq:limxbaralpha_xt/t=0}
			\lim_{t\to+\infty} \frac{\bar{x}(\alpha_x(t))}{t} = 0.
		\end{equation}
	\end{enumerate}
\end{lemma} 
\begin{proof}
	\begin{enumerate}[(i)]
		\item Using the second inequality in \eqref{eq:barx1} and the assumption that $\limsup_{t\to+\infty} \frac{\bar{x}(\alpha_x(t))}{t} < 0$, we have
		\[
			\limsup_{t\to+\infty} \frac{\alpha_x(t)}{t} \leq \limsup_{t\to+\infty} \frac{\bar{x}(\alpha_x(t))}{t} + \lim_{t\to+\infty} \frac{\bar{\mu}(\mathbb{R})}{t} < 0,
		\]
		and hence, there exists a constant $c>0$ such that $\limsup_{t\to+\infty} \frac{\alpha_x(t)}{t} < -c < 0$. As a result, for any sequence $t_n\to +\infty$, we have
		\[
			\alpha_x(t_n) \leq -c t_n,
		\]
		so for any sufficiently large $n$, $\alpha_x(t_n) < \ell$, which and \eqref{eq:leftright} imply that $\alpha_x(t_n)=\bar{x}(\alpha_x(t_n))$. Thus, we have $\lim_{n\to+\infty} t_n[\alpha_x(t_n)-\bar{x}(\alpha_x(t_n))]= 0$. Since the sequence $\{t_n\}_{n=1}^\infty$ is arbitrary,
		the limit~\eqref{eq:limtalpha_xt-xbaralpha_xt=0} follows immediately.
		\item First of all, we claim that there exists a $T>0$ such that
		\begin{equation}\label{eq:alpha_xt_geq_ell}
			\alpha_x(t) \geq \ell, \quad\mbox{for all }t>T.
		\end{equation}
		Seeking for a contradiction, we assume that there exists a sequence $t_n\to +\infty$ such that $\alpha_x(t_n)<\ell$ for all $n$. Using \eqref{eq:leftright}, we know that $\alpha_x(t_n)=\bar{x}(\alpha_x(t_n))$ for all $n$, and hence, $\lim_{n\to +\infty} t_n[\alpha_x(t_n)-\bar{x}(\alpha_x(t_n))] = 0$, which contradicts the assumption that $\liminf_{t\to+\infty} t[\alpha_x(t)-\bar{x}(\alpha_x(t))]> 0$.
		
		Now, it follows from \eqref{eq:alpha_xt_geq_ell} that $\liminf_{t\to+\infty} \frac{\alpha_x(t)}{t} \geq 0$. On the other hand, using \eqref{eq:barx1} and the boundedness of $\bar{\mu}$, one can verify that $\liminf_{t\to+\infty} \frac{\bar{x}(\alpha_x(t))}{t} = \liminf_{t\to+\infty} \frac{\alpha_x(t)}{t}$, and hence,
		\begin{equation}\label{eq:liminf_xbaralpha_xt/tgeq0}
			\liminf_{t\to+\infty} \frac{\bar{x}(\alpha_x(t))}{t} \geq 0.
		\end{equation}
		The limit \eqref{eq:limxbaralpha_xt/t=0} is an immediate consequence of \eqref{eq:limsupxbaralpha_xt} and \eqref{eq:liminf_xbaralpha_xt/tgeq0}.
	\end{enumerate}
\end{proof}
Now, we are ready to analyze the situation in three different cases: $\bar{u}(-\infty)>0$; $\bar{u}(-\infty)<0$; and $\bar{u}(-\infty)=0$. In particular, we will verify \eqref{eq:leftlimit4} and \eqref{eq:compact10}, as well as \eqref{eq:limtingforfixedx}.

\textbf{Case 1.} Let $\bar{u}(-\infty)>0$. Then taking the limit superior in \eqref{eq:xbarbetax/t} as $t\to +\infty$, we obtain
\begin{equation*}\label{e:}
	\begin{aligned}
		0 = \limsup_{t\to +\infty} \frac{x}{t} &= \limsup_{t\to +\infty} \left\{ \frac{\bar{x}(\alpha_x(t))}{t} + \bar{u}(\bar{x}(\alpha_x(t))) + \frac{1}{4} t [\alpha_x(t) - \bar{x}(\alpha_x(t))] \right\}\\
		&\geq \limsup_{t\to +\infty} \frac{\bar{x}(\alpha_x(t))}{t} + \bar{u}(-\infty) + \frac{1}{4} \liminf_{t\to +\infty} t [\alpha_x(t) - \bar{x}(\alpha_x(t))] > \limsup_{t\to +\infty} \frac{\bar{x}(\alpha_x(t))}{t},
	\end{aligned}
\end{equation*}
where we applied \eqref{eq:baru} and \eqref{eq:liminftalpha_xt-xbaralpha_xt} above. Therefore, it follows from part (i) of Lemma~\ref{lemma:dilemma} that the limit \eqref{eq:limtalpha_xt-xbaralpha_xt=0} holds. Now, passing to the limit in \eqref{eq:xbarbetax/t} as $t\to +\infty$ again, and using \eqref{eq:baru} and \eqref{eq:limtalpha_xt-xbaralpha_xt=0}, we finally obtain
\[
	\lim_{t\to+\infty}\frac{\bar{x}(\alpha_x(t))}{t}=-\bar{u}(-\infty),
\]
which is \eqref{eq:leftlimit4} and implies \eqref{eq:compact10}. Using \eqref{eq:barx1}, we also obtain
\[
	\lim_{t\to+\infty}\frac{\alpha_x(t)}{t}=-\bar{u}(-\infty),\quad\mbox{and}\quad \lim_{t\to+\infty}\alpha_x(t) = -\infty.
\]
Finally, passing to the limit in \eqref{eq:Formula_for_u_in_terms_of_alpha_xt} as $t\to +\infty$, and using the limits \eqref{eq:baru} and \eqref{eq:limtalpha_xt-xbaralpha_xt=0}, we also have \eqref{eq:limtingforfixedx}.

\textbf{Case 2.} Let $\bar{u}(-\infty)<0$. Then taking the limit inferior in \eqref{eq:xbarbetax/t} as $t\to +\infty$, we obtain
\begin{equation*}\label{e:}
	\begin{aligned}
		0 = \liminf_{t\to +\infty} \frac{x}{t} &= \liminf_{t\to +\infty} \left\{ \frac{\bar{x}(\alpha_x(t))}{t} + \bar{u}(\bar{x}(\alpha_x(t))) + \frac{1}{4} t [\alpha_x(t) - \bar{x}(\alpha_x(t))] \right\}\\
		&\leq \limsup_{t\to +\infty} \frac{\bar{x}(\alpha_x(t))}{t} + \bar{u}(-\infty) + \frac{1}{4} \liminf_{t\to +\infty} t [\alpha_x(t) - \bar{x}(\alpha_x(t))] < \frac{1}{4} \liminf_{t\to +\infty} t [\alpha_x(t) - \bar{x}(\alpha_x(t))],
	\end{aligned}
\end{equation*}
where we applied \eqref{eq:baru} and \eqref{eq:limsupxbaralpha_xt} above. Therefore, using part (ii) of Lemma~\ref{lemma:dilemma}, we obtain the limit \eqref{eq:limxbaralpha_xt/t=0}, which is indeed \eqref{eq:leftlimit4} in this case. Using \eqref{eq:barx1}, \eqref{eq:limxbaralpha_xt/t=0} and the boundedness of $\bar{\mu}$, we also have
\[
	\lim_{t\to+\infty}\frac{\alpha_x(t)}{t}=0.
\]

Seeking for a contradiction, we assume that $\liminf_{t\to +\infty} \bar{x}(\alpha_x(t)) < \ell$. Then there exist a constant $\delta>0$ and a sequence $t_n\to +\infty$ such that $\bar{x}(\alpha_x(t_n)) < \ell-\delta$ for all $n$. It follows from \eqref{eq:leftright} that $\bar{x}(\ell-\delta)=\ell-\delta$, so we have $\bar{x}(\alpha_x(t_n))<\bar{x}(\ell-\delta)$. Since the function $\bar{x}$ is nondecreasing, we know that $\alpha_x(t_n)\leq\ell-\delta$, which and \eqref{eq:leftright} imply $\bar{x}(\alpha_x(t_n))\equiv\alpha_x(t_n)$. Hence, $\liminf_{n\to +\infty} t_n [\alpha_x(t_n) - \bar{x}(\alpha_x(t_n))]=0$, which contradicts $\liminf_{t\to +\infty} t [\alpha_x(t) - \bar{x}(\alpha_x(t))]>0$. As a result, we have shown $\liminf_{t\to +\infty} \bar{x}(\alpha_x(t)) \geq \ell$, which and \eqref{eq:leqell} imply \eqref{eq:compact10}. Dividing \eqref{eq:xbarbetax} by $t^2$, and then passing to the limit as $t\to +\infty$, as well as using the boundedness of $x$, $\bar{x}(\alpha_x(t))$ and $\bar{u}(\bar{x}(\alpha_x(t)))$, we also obtain
\[
	\lim_{t\to +\infty} \alpha_x(t) = \lim_{t\to +\infty} \bar{x}(\alpha_x(t)) = \ell.
\] 

Now, passing to the limit in \eqref{eq:xbarbetax/t} as $t\to +\infty$ again, and using \eqref{eq:baru} and \eqref{eq:limxbaralpha_xt/t=0}, we finally obtain
\begin{equation}\label{eq:lim_t(alpha-xbaralpha)}
	\lim_{t\to+\infty} t[\alpha_x(t)-\bar{x}(\alpha_x(t))]=-4\bar{u}(-\infty).
\end{equation}
Finally, passing to the limit in \eqref{eq:Formula_for_u_in_terms_of_alpha_xt} as $t\to +\infty$, and using the limits \eqref{eq:baru} and \eqref{eq:lim_t(alpha-xbaralpha)}, we also have \eqref{eq:limtingforfixedx}.

\textbf{Case 3.} Let $\bar{u}(-\infty)=0$. Then taking the limit superior and inferior in \eqref{eq:xbarbetax/t} as $t\to +\infty$, we obtain, after using \eqref{eq:baru},
\begin{equation*}\label{e:}
	\begin{aligned}
		0 &= \limsup_{t\to +\infty} \left\{ \frac{\bar{x}(\alpha_x(t))}{t} + \frac{1}{4} t [\alpha_x(t) - \bar{x}(\alpha_x(t))] \right\} \geq \limsup_{t\to +\infty} \frac{\bar{x}(\alpha_x(t))}{t} + \frac{1}{4} \liminf_{t\to +\infty} t [\alpha_x(t) - \bar{x}(\alpha_x(t))] \\
		&\geq \liminf_{t\to +\infty} \left\{ \frac{\bar{x}(\alpha_x(t))}{t} + \frac{1}{4} t [\alpha_x(t) - \bar{x}(\alpha_x(t))] \right\} = 0,
	\end{aligned}
\end{equation*}
which implies
\begin{equation}\label{eq:limsup+liminf=0}
	\limsup_{t\to +\infty} \frac{\bar{x}(\alpha_x(t))}{t} + \frac{1}{4} \liminf_{t\to +\infty} t [\alpha_x(t) - \bar{x}(\alpha_x(t))] = 0.
\end{equation}
Using \eqref{eq:limsupxbaralpha_xt}, \eqref{eq:liminftalpha_xt-xbaralpha_xt}, \eqref{eq:limsup+liminf=0} and Lemma~\ref{lemma:dilemma}, we can conclude that
\begin{equation}\label{eq:limsup=liminf=0}
	\limsup_{t\to +\infty} \frac{\bar{x}(\alpha_x(t))}{t} = \liminf_{t\to +\infty} t [\alpha_x(t) - \bar{x}(\alpha_x(t))] = 0.
\end{equation}
This verifies \eqref{eq:leftlimit4}. Using \eqref{eq:barx1}, \eqref{eq:limxbaralpha_xt/t=0} and the boundedness of $\bar{\mu}$, we can also conclude that
\[
	\lim_{t\to+\infty}\frac{\alpha_x(t)}{t}=0.
\]
Finally, passing to the limit in \eqref{eq:Formula_for_u_in_terms_of_alpha_xt} as $t\to +\infty$, and using the limits \eqref{eq:baru} and \eqref{eq:limsup=liminf=0}, we also verifies \eqref{eq:limtingforfixedx}.

\section{Formal derivation of \eqref{eq:generalizedchara1} for classical solutions}\label{app:D}
To motivate the generalized characteristics \eqref{eq:generalizedchara1} under the generalized framework \eqref{eq:gHS21}-\eqref{eq:gHS23}, let us begin with the case that the conservative solution $(u(x,t),\mu(t))$ is always classical, in the sense that $\di \mu(t)=u_x^2(x,t)\di x$ for all time $t\in \mathbb{R}$. In this case, the energy conservation  \eqref{eq:conservationlaw} also holds.
Following the same idea as in \cite{gao2021regularity}, we can derive \eqref{eq:generalizedchara1} as follows: for any given $\beta$, $t\in\mathbb{R}$, we define $x(\beta,t)$ by
\begin{align}\label{eq:defx}
x(\beta,t)+\int_{(-\infty, x(\beta,t))}u_y^2(y,t)\di y=\beta. 
\end{align}
Differentiating \eqref{eq:defx} with respect to $t$ and $\beta$ respectively, we obtain
\begin{equation}\label{eq:dtx&dbetax}
\partial_tx(\beta,t)=\frac{(uu_x^2)(x(\beta,t),t)}{1+u_x^2(x(\beta,t),t)},\quad\mbox{and}\quad \partial_{\beta}x(\beta,t)=\frac{1}{1+u_x^2(x(\beta,t),t)},
\end{equation}
where we have applied \eqref{eq:conservationlaw} while computing $\partial_tx(\beta,t)$.
For any fixed $\alpha\in \mathbb{R}$, denote by $\beta(t)$ the unique solution to the following initial-value problem:
\begin{align}\label{eq:betat}
\beta'(t)=u(x(\beta(t),t),t),\quad\mbox{and}\quad \beta(0)=\alpha.
\end{align}
It follows from the chain rule, \eqref{eq:dtx&dbetax} and \eqref{eq:betat} that
\begin{align}\label{eq:translation}
	\frac{\di }{\di t}x(\beta(t),t)=\partial_tx(\beta(t),t)+\partial_\beta x(\beta(t),t)\cdot \beta'(t)=u(x(\beta(t),t),t)=\beta'(t).
\end{align}
Denote by $\bar{x}(\alpha)$ the unique solution to
\begin{align}\label{eq:smooth barx} 
\bar{x}(\alpha)+\int_{(-\infty, \bar{x}(\alpha))}u_y^2(y,0)\di y=\alpha.
\end{align}
Then it follows from \eqref{eq:defx}, \eqref{eq:betat} and \eqref{eq:translation} that 
\begin{align}\label{eq:rela}
\bar{x}(\beta(0),0)=\bar{x}(\alpha),\quad\mbox{and}\quad \beta(t)-x(\beta(t),t)=\alpha-\bar{x}(\alpha).
\end{align}
{ Furthermore, differentiating \eqref{eq:betat} with respect to $t$, and then using the chain rule, \eqref{eq:translation}, \eqref{eq:betat}, \eqref{eq:HSanother} and \eqref{eq:defx}, one can verify that }
\[
\beta''(t)=\frac{1}{2}[\beta(t)-x(\beta(t),t)]-\frac{1}{4}\bar{\mu}(\mathbb{R}).
\]
Taking one more time derivative and using \eqref{eq:translation}, we finally obtain
\begin{align}\label{eq:thirdorder}
\beta'''(t)=0.
\end{align}
Moreover, we have the following initial data for the ODE \eqref{eq:thirdorder}: 
\begin{align}\label{eq:initial}
\beta(0)=\alpha,\quad \beta'(0)=u(\bar{x}(\alpha),0),\quad \beta''(0)=\frac{1}{2}(\alpha-\bar{x}(\alpha))-\frac{1}{4}\bar{\mu}(\mathbb{R}).
\end{align}
Therefore, solving \eqref{eq:thirdorder} and \eqref{eq:initial} yields
\begin{align}\label{eq:global1}
\beta(t)=\alpha+u(\bar{x}(\alpha),0)t+\frac{t^2}{4}\left[\alpha-\bar{x}(\alpha)-\frac{1}{2}\bar{\mu}(\mathbb{R})\right],
\end{align}
and hence, using \eqref{eq:rela}, we also have 
\begin{align}\label{eq:global2}
x(\beta(t),t)=\bar{x}(\alpha)+u(\bar{x}(\alpha),0)t+\frac{t^2}{4}\left[\alpha-\bar{x}(\alpha)-\frac{1}{2}\bar{\mu}(\mathbb{R})\right].
\end{align}
Therefore, to obtain the global formulae for $\beta(t)$ and $x(\beta(t),t)$, we only need the information of initial data $u(x,0)$. 
For the general initial data $(\bar{u},\bar{\mu})$ in $\mathcal{D}$, we first define $\bar{x}(\alpha)$ by \eqref{eq:barx1}, and then we can still use \eqref{eq:global2}
to define the generalized characteristics. Finally we set
$y(\alpha,t)=x(\beta(t),t)$, this is exactly \eqref{eq:generalizedchara1}. Moreover, one can directly check that \eqref{eq:measuresolutionu1} indeed provides the global-in-time conservative solution to \eqref{eq:gHS21}-\eqref{eq:gHS23} subject to the initial data $(\bar{u},\bar{\mu})\in \mathcal{D}$.

\

\noindent\textbf{Acknowledgements} Y. Gao is supported by  the National Natural Science Foundation grant 12101521 of China  and  the Start-up fund from The Hong Kong Polytechnic University with project number P0036186. 
T. K. Wong is partially supported by the HKU Seed Fund for Basic Research under the project code 201702159009, the Start-up Allowance for Croucher Award Recipients, and Hong Kong General Research Fund (GRF) grants with project number 17306420, 17302521, and 17315322.

\bibliographystyle{plain}
\bibliography{bibofChara}

\end{document}